\documentclass[11pt,reqno]{amsart}
\usepackage[T1]{fontenc}
\usepackage{graphicx, amssymb, amsmath, stmaryrd, mathrsfs, wasysym}
\usepackage{cite}
\usepackage{color}
\definecolor{MyLinkColor}{rgb}{0,0,0.4}

\newcommand{\R}{{\mathbb R}}

\newcommand{\bA}{{\mathbb A}}
\newcommand{\cA}{{\mathcal A}}
\newcommand{\bB}{{\mathbb B}}
\newcommand{\bD}{{\mathbb D}}

\newcommand{\mK}{\mathcal{K}}
 
\newcommand{\N}{{\mathbb N}}
\newcommand{\s}{\mathbb S}

\newcommand{\clu}{\mathcal{U}}

\newcommand{\cO}{\mathcal{O}}

\newcommand{\kL}{\mathcal{L}}

\newcommand{\cV}{\mathcal{V}}
\newcommand{\cR}{\mathcal{R}}
\newcommand{\sfB}{\mathsf{B}}
\newcommand{\wh}{\widehat}
\newcommand{\wt}{\widetilde}

\newcommand{\Lip}{\widetilde{\rm Lip}}
\newcommand{\re}{\mathop{\rm Re}\nolimits}

\newcommand{\PV}{\mathop{\rm PV}\nolimits}
\newcommand{\ov}{\overline}
\newcommand{\p}{\partial}

\newcommand{\e}{\varepsilon}

\newcommand{\0}{\Omega}
\newcommand{\G}{\Gamma}

\newcommand{\supp}{\mathop{\rm supp}\nolimits}
\newcommand{\eps}{\varepsilon}

\newcommand{\be}{\begin{equation}}
\newcommand{\ee}{\end{equation}}
\newcommand{\clf}{\mathcal{F}}
\newcommand{\locsim}{\;\;\underset{\rm loc}{{\stackrel{j,\e}{\sim}}}\;\;}
\newtheorem{thm}{Theorem}[section]
\newtheorem{prop}[thm]{Proposition}
\newtheorem{lemma}[thm]{Lemma}
\newtheorem{cor}[thm]{Corollary}
\theoremstyle{remark}

\numberwithin{equation}{section} 
\usepackage[colorlinks=true,linkcolor=MyLinkColor,citecolor=MyLinkColor]{hyperref}

\begin{document}

\title[The gravity driven Muskat problem]{The $N$-dimensional gravity driven Muskat problem}

\thanks{}
\author{Bogdan-Vasile Matioc}
\address{Fakult\"at f\"ur Mathematik, Universit\"at Regensburg \\ D--93040 Regensburg, Deutschland}
\email{bogdan.matioc@ur.de}

\author{Georg Prokert}
\address{Department of Mathematics and Computer Science, Eindhoven University of Technology, The Netherlands.}
\email{g.prokert@tue.nl}

\begin{abstract}
We study the Muskat problem, which describes the motion of two immiscible, incompressible fluids in a homogeneous porous medium  occupying the full space~${\mathbb{R}^{N+1}}$, $N \geq 2$, driven by gravity. 
The interface between the fluids is given as graph of a function over~$\R^N$.
The problem is reformulated as a nonlinear, nonlocal evolution problem for  this function, involving singular integrals arising from potential representations of the velocity and pressure fields.  Using results from harmonic analysis, we demonstrate that the  evolution is of parabolic type  in the open  set identified by   the Rayleigh-Taylor condition.
We  use the abstract theory of such problems to establish that the Muskat problem defines a semiflow on  this set in  all subcritical Sobolev spaces $H^s(\mathbb{R}^N)$, $s>s_c$, where ~${s_c=1+N/2}$ is the critical exponent. We additionally obtain parabolic smoothing up to ${\rm C}^\infty$.
\end{abstract}

%%% NEED TO ADD MSC %%%

\maketitle

%\tableofcontents

\pagestyle{myheadings}
\markboth{\sc{B.-V.~Matioc \&  G. Prokert }}{\sc{The gravity driven Muskat problem in $\R^N$}}

%%%%%%%%%%%%%%%%%%%%%%%%%%%%%%%%%%%%%%%%%%%%%%%%%%
%%%%%%%%%%%%%%%%%%%%%%%%%%%%%%%%%%%%%%%%%%%%%%%%%%
%%%%%%%%%%%%%%%%%%%%%%%%%%%%%%%%%%%%%%%%%%%%%%%%%%
%%%%%%%%%%%%%%%%%%%%%%%%%%%%%%%%%%%%%%%%%%%%%%%%%%
 \section{Introduction}\label{Sec:1}
%%%%%%%%%%%%%%%%%%%%%%%%%%%%%%%%%%%%%%%%%%%%%%%%%%
%%%%%%%%%%%%%%%%%%%%%%%%%%%%%%%%%%%%%%%%%%%%%%%%%%
%%%%%%%%%%%%%%%%%%%%%%%%%%%%%%%%%%%%%%%%%%%%%%%%%%
%%%%%%%%%%%%%%%%%%%%%%%%%%%%%%%%%%%%%%%%%%%%%%%%%%
 In this paper we study the motion of two immiscible and incompressible  Newtonian fluids occupying a homogeneous porous medium, modeled as the entire space $\mathbb{R}^{N+1}$, with $N \geq 2$. 
The fluids occupy time-dependent domains $\Omega^+(t)$ and $\Omega^-(t)$, respectively, and are separated by a sharp interface $\Gamma(t)$. Gravity is considered the sole driving force, so, in particular, surface tension effects are neglected. The motion within each fluid domain is governed by Darcy's law \cite{Da56}.  
This model is commonly known as the Muskat problem \cite{Mu34}. With $\mu^\pm$, 
$\rho^\pm$ as viscosities and densities of the respective fluids, $k$ as its porosity, and $g$ as gravity, it 
is described by the  system of equations
\begin{subequations}\label{Muskat}
\begin{equation}\label{MuskatEE}
\left.
\arraycolsep=1.4pt
\begin{array}{rclll}
v^\pm(t)+\cfrac{k}{\mu^\pm}\nabla\big(p^\pm(t)+g\rho^\pm y\big)
&=&0&\text{in $\0^\pm(t)$}\\[1ex]
{\rm div\,} v^\pm(t)&=&0&\mbox{in $\Omega^\pm(t)$,}\\[1ex]
[p(t)]&=&0 &\mbox{on $\Gamma(t)$,}\\[1ex]
[v(t)]\cdot\wt\nu(t)&=&0 &\mbox{on $\Gamma(t)$,}\\[1ex]
  (v^\pm(t,x,y),p^\pm(t,x,y)+g\rho^\pm y)&\to&0&\mbox{for $|(x,y)|\to\infty$,}\\[1ex]
V_{\wt \nu}(t)&=& v^\pm(t)\cdot\wt \nu(t)&\mbox{on $\Gamma(t)$}
\end{array}\right\}
\end{equation}
for $t\geq 0$, where  the fluid domains $\0^\pm(t)$ and their common boundary $\G(t)$ are given by
\begin{align*}
\0^\pm(t)&:=\{(x,y)\in\R^N\times\R\,:\, y\gtrless f(t,x)\},\\
\G(t)&:=\p\0^\pm(t):=\{(x, f(t,x))\,:\, x\in\R^N\}.
\end{align*}
Additionally, the  interface $\Gamma(t)$ is assumed to be known at time~${t=0}$,  i.e.
 \begin{align}\label{IC}
 f(0,\cdot) =f_0.
 \end{align}
 \end{subequations}

 In \eqref{MuskatEE}, $v^\pm(t)$ and $p^\pm(t)$ are the velocity and pressure fields of the fluids in the respective domains. The constants $\mu^{\pm}$, $\rho^\pm$, $g$, and $k$ are all assumed positive. 

We set $\wt \nu(t) $ to be  the unit normal at $\G(t)$ pointing into $\0^+(t)$,   $a\cdot b $ is the Euclidean inner product of the vectors~$a$ and~$b$, and $V_{\wt \nu}(t)$ is the normal velocity of~$\G(t)$. 
 
 Moreover, if $u$  is a function defined of $\R^{N+1}\setminus\Gamma(t)$ such that the restrictions~$u^\pm:=u|_{\Omega^\pm(t)}$  have continuous extensions on $\overline{\Omega^\pm(t)}$ then we denote by
 \[
  [u]:=u^+|_{\G(t)}-u^-|_{\G(t)} 
 \] 
 the jump of  $u$  across {$\Gamma(t)$.}

Our approach to solving problem \eqref{Muskat} consists in deriving and investigating a nonlocal, nonlinear evolution equation  for the function $f$ describing the interface $\Gamma(t)$. For this evolution equation, 
 the Rayleigh-Taylor condition identifies an open subset of the state space where the problem is parabolic.

 Let $H^r(\R^N)$, $r\geq 0$, denote the usual Bessel potential spaces.  For integer $r$ these spaces coincide with the usual Sobolev spaces $W^r_2(\R^N)$, and for noninteger $r$ with the Sobolev-Slobodeckii spaces $W^r_2(\R^N)$.
 
 From the point of view of scaling invariance, the  space \(  H^{1+N/2}(\mathbb{R}^N) \), 
can be identified as a critical space for~\eqref{Muskat}; see, e.g.~\cite{FN21}. 
Our goal is to establish the well-posedness of the Muskat problem~\eqref{Muskat}  in all subcritical spaces \( H^{s}(\mathbb{R}^N) \), where 
\begin{equation}\label{eq:s}
s >s_c:=1+\frac{N}{2}.
\end{equation}
  Our  main result, given in Theorem \ref{MT}, improves upon the current knowledge in the sense that we show that \eqref{Muskat} defines a semiflow on the 
  set determined by the Rayleigh-Taylor condition
  in all  subcritical spaces~$H^s(\R^N)$, $s>s_c$. In particular, all equations of \eqref{Muskat} are satisfied pointwise in the classical sense. The proof provides uniqueness under  natural preliminary assumptions on the smoothness and the decay at infinity for  the velocity and pressure fields.

 \subsection*{Summary of known results}\ 
We  introduce the characteristic velocity
 \be\label{lambda}
\Lambda:= \frac{2kg(\rho^--\rho^+)}{\mu^++\mu^-}
\ee
 and the dimensionless number
\be\label{defamu}
a_\mu:=\frac{\mu^+-\mu^-}{\mu^++\mu^-}\in(-1,1).
\ee

 The Rayleigh-Taylor condition \cite{ST58} plays a crucial role in the analysis of the gravity-driven Muskat problem, as it ensures parabolicity \cite{EM11a}. In our geometry and notation it reads
 
\begin{equation}\label{RT}
[\nabla p]\cdot \wt \nu>0 \quad \text{on } \G.
\end{equation}
  For fluids with equal viscosities~(${a_\mu=0}$), this  condition simplifies to the requirement that the constant $\Lambda$ from Equation \eqref{lambda} is positive; see \eqref{RTref}. 
In the  general case where~${a_\mu \in (-1,1)}$, condition \eqref{RT}  is equivalent to a system of two inequalities,  namely,~${\Lambda>0}$  and a more complex inequality involving both $a_\mu$ and (nonlinearly and nonlocally) the function~$f$. To our knowledge, the question whether $\Lambda>0$ alone implies \eqref{RT} in general is still open.

 The Muskat problem with equal viscosities has been extensively studied in the mathematical literature; see, for instance, the surveys \cite{G17, GL20}. 
  However, the case of different viscosities has been significantly less explored in the literature.
  This is primarily due to  more complex nonlinearity  and nonlocality in the mathematical formulation when~${a_\mu \neq 0}$, as well as the  resulting more intricate nature of the Rayleigh-Taylor condition in this setting.

Local well-posedness of \eqref{Muskat} with $N=1$ has been established in several works.  Specifically, \cite{CCG11} proves local well-posedness for $H^3$-initial data, while \cite{A04, BCS16, MBV18, MBV20} address the case of~$H^2$-initial data. The results in \cite{CCG11, A04, BCS16} are obtained via energy methods, whereas \cite{MBV18, MBV20} use an approach similar to ours.

 In addition, for $N=1$, well-posedness with a classical solution concept has been shown in \cite{AM22} for initial data belonging to  the fractional order Sobolev space $W^s_p(\R)$ with $p \in (1,\infty)$ and $s \in (1 + 1/p, 2)$, i.e. in subcritical spaces arbitrarily close to the critical space $W^{1+1/p}_p(\R)$.

Furthermore, the Wiener space $\dot {\mathcal{F}}^{1,1}(\R^N)$ has been identified in \cite{GGPS19} as a critical scaling invariant space for \eqref{Muskat}. For $N \in \{1, 2\}$, the same paper establishes the existence and uniqueness of global strong solutions for initial data in $L_2(\R^N)\cap \dot {\mathcal{F}}^{1,1}(\R^N)$ that satisfy certain size constraints.

 In arbitrary spatial dimensions, \cite{NP20} proves local in time existence and uniqueness of strong solutions to \eqref{Muskat} in subcritical spaces $H^s(\R^N)$,  $s >s_c$, through the application of paradifferential calculus,  with the function $f$ belonging to the regularity class
\[
 L_\infty([0,T]; H^s(\R^N))\cap L_2([0,T]; H^{s+\tfrac{1}{2}}(\R^N))\qquad\text{for some $T>0$,}
\]
and  the Rayleigh-Taylor condition \eqref{RT} holding for the initial interface.
 Under the same assumptions, \cite{FN21} shows that  strong solutions to the Muskat problem with surface tension converge towards solutions  
to the gravity driven Muskat problem as surface tension approaches zero.
Recently, local well-posedness of~\eqref{Muskat} with equal viscosity constants has been shown using modulus of continuity techniques~\cite{CCRNY26}. 

For small data in the  critical homogeneous Besov space~${\dot {B}^1_{\infty,1}(\R^N)}$, $N\geq 1$, global existence and uniqueness of strong solutions is  proved in~\cite{HuN22}.
For \(N=1\), global existence for small data and stability results   have been obtained in~\cite{SCH04, MBV20, BCS16}.
  
  The Muskat problem with  $N=1$, different viscosities, and without surface tension
  in geometries  other than the one considered here  has been studied
  in \cite{EM11a, EMM12a, BM25, EMW18, Y03, Va25, BV14, GS19}. 
 The available results include local well-posedness in various bounded geometries, including settings where the interface separating the fluids has a corner point \cite{BV14} or forms acute corners with the fixed boundaries \cite{Va25}. Further results address local well-posedness and stability/instability of flat or finger-shaped equilibria in periodic strip-like geometries \cite{EM11a, EMM12a, EMW18}, or in the case of discontinuous permeability of the porous medium \cite{GS19}. The existence of global solutions for small initial data in a non-periodic strip-like geometry is established in \cite{Y03}. A generalization of local well-posedness results to the case of three fluid phases with general vorticity and densities is given in \cite{BM25}.

 \subsection*{Structure of the paper and main result} 
Our approach to the Muskat problem \eqref{Muskat} is based on potential theory, harmonic analysis, and abstract parabolic theory.  It relies crucially on the  investigation of a class of singular integral operators generalizing Riesz transforms,  the details of which are { presented in Appendices~\ref{Sec:C} and~\ref{Sec:D}. While their one-dimensional versions have been studied and used before in \cite{AM22,MBV18,MBV19,MM21,MM23}, this approach is new in the multidimensional case.

 We start in Section~\ref{Sec:2} by showing that at each fixed time, the sharp interface $\Gamma(t)$ between the two fluids determines 
the pressure and velocity fields in both layers. This is based on classical potentials. More precisely, the pressure and the velocity are given as integrals over the graph $\G(t)$, with the density function $\beta$ implicitly defined as the solution to the singular integral equation \eqref{resolvent?}, involving the classical double layer potential for the Laplacian. However, our unbounded graph geometry is somewhat nonstandard for these techniques, and we collect the results we need in Appendix A, as they may be of independent interest.

The unique solvability of the integral equation   \eqref{resolvent?} in $L_2(\R^N)$ is established in Section~\ref{Sec:3}, and in~$H^s(\R^N)$ in Section~\ref{Sec:4}. 
The analysis in these sections relies on a Rellich identity and on  mapping properties of  the family of generalized Riesz transforms $B^\phi_{n, \nu}$, detailed in Appendix~\ref{Sec:C}, which are of broader interest.

Building on these results, we then show in Section~\ref{Sec:5} that the Muskat problem \eqref{Muskat} can be formulated as a fully nonlinear and nonlocal evolution problem:
\[
\frac{{\rm d} f}{{\rm d} t}(t) = \Phi(f(t)), \quad t \geq 0, \qquad f(0) = f_0,
\]
where $\Phi: H^s(\R^N) \to H^{s-1}(\R^N)$ is smooth.  Concerning the dependence on the problem parameters, we point out that $\Phi(f)=\Lambda \widetilde\Phi(f)$, where  $\widetilde\Phi(f)$ depends only on $a_\mu$ (but not on both viscosities individually or the other problem parameters).

Moreover, we prove that the Rayleigh-Taylor condition \eqref{RT} is equivalent to the inequality
\begin{equation}\label{RTref}
 \Lambda\big(1-2a_\mu\widetilde\Phi(f)\big) > 0.
\end{equation}
As $\widetilde\Phi(f)$ vanishes as $|x|\to\infty$ for any $f$ and $a_\mu$, 
the condition $\Lambda>0$ is necessary for the Rayleigh-Taylor condition to hold. 
 This necessary condition is equivalent to  the property  that the fluid with the lower density lies above the one with the higher density. Apart from this, for any given $f$, the validity of the Rayleigh-Taylor condition depends only on $a_\mu$.

 We will henceforth assume that $\Lambda>0$.  In this case, the set 
 \be\label{eq:defO}
\cO := \big\{ f \in H^s(\R^N) :  2 a_\mu\widetilde\Phi(f)<1 \big\}
\ee
  consists precisely of the functions that describe interfaces for which the Rayleigh-Taylor condition holds. It is open in $H^s(\R^N)$. As $f=0$ corresponds to a trivial equilibrium, we have $\Phi(0)=\widetilde\Phi(0)=0$, and thus  $\cO$ 
 is nonempty for any $a_\mu$. While obviously $\cO=H^s(\R^N)$ for $a_\mu=0$, we reiterate that  the question whether this also holds in the case of different viscosities seems to be open.
 
With Theorem~\ref{T:GP} we then prove that  the Muskat problem is of parabolic type in $\cO$.
To establish Theorem~\ref{T:GP}, we localize the Fréchet derivative~$\partial \Phi(f)$ for~$f\in\cO$ and prove that  this unbounded operator generates a strongly continuous and analytic semigroup  on~$H^{s-1}(\R^N)$.
This is done using results from Appendix~\ref{Sec:D}, which provide commutator-type estimates and localization  results in the context of the generalized Riesz transforms~$B^\phi_{n, \nu}$. 
It is reasonable to conjecture that the set $\cO$ is the full domain of parabolicity of~ \eqref{Muskat} in the sense that the operator $\p\Phi(f) $ generates a strongly continuous analytic semigroup on $H^{s-1}(\R^N)$ if and only if $f\in \cO$ (although  this does not follow directly from our analysis).

Finally, based on abstract theory for fully nonlinear parabolic problems, we establish the following local well-posedness and parabolic smoothing result for~\eqref{Muskat}.

 \begin{thm}\label{MT}
Assume  $\Lambda>0$, $s>s_c$, and let $f_0\in \cO$. 
Then the following hold true:
\begin{itemize}
\item[{\rm (i)}] {\bf (Well-posedness)} Problem~\eqref{Muskat} has a unique maximal  solution $(f,v^\pm,p^\pm)$  with existence time $T^+:=T^+(f_0)\in(0,\infty]$ such that 
\begin{itemize}
\item[$\bullet$] $f:=f(\cdot;f_0)\in{\rm C}([0,T^+), \cO)\cap {\rm C}^1([0,T^+), H^{s-1}(\R^N))$;\\[-2ex]
\item[$\bullet$]  ${v^\pm(t)\in {\rm C}(\overline{\0^\pm(t)})\cap {\rm C}^1(\0^\pm(t))}$, $p^\pm(t)\in {\rm C}^1( \ov{\0^\pm(t)})\cap {\rm C}^2(\0^\pm(t))$ for~${t\in[0, T^+)}$.\\[-2ex]
\end{itemize}
Moreover, the solution mapping $[(t,f_0)\mapsto f(t;f_0)]$ defines a semiflow on $\cO$.\\[-2ex]
\item[{\rm (ii)}] {\bf (Parabolic smoothing)} We have 
$ [(t,x)\mapsto f(t ;f_0)(x)]\in {\rm C}^\infty((0,T^+)\times\R^N)$.
\end{itemize}
\end{thm}

 We point out that the integral operators $B_{n,\nu}^\phi$ are flexible tools that we expect to be useful for the treatment of other moving boundary problems in the same geometric setting, whenever the underlying elliptic problems have constant coefficients and are therefore amenable to solutions by classical layer potentials. This includes e.g. quasistationary Stokes flow problems.

\subsection*{Notation  and preliminaries} 
Given  Banach spaces $E,\,E_1,\ldots,E_n,\,F$, $n\in\N$,  we denote by~$\kL^n\big(E_1,\ldots,E_n,F\big)$  the 
Banach space of bounded $n$-linear maps from $\prod_{i=1}^n E_i$ to~$F$  (we simplify the notation to~$\kL^n(E,F)$   if $E_1=\ldots=E_n=E$). Similarly,
 $\kL^n_{\rm sym}(E,F)$ stands for the space of $n$-linear, bounded, and symmetric maps $A: E^n\to F$.
Furthermore, the sets of all locally Lipschitz continuous mappings  and of all smooth mappings from an open set $\clu\subset E$ 
to~$F$ are denoted by ${\rm C}^{1-}(\clu,F)$  and~${{\rm C}^{\infty}(\clu,F)}$, respectively.
 We also write~${\p \Phi:\clu\to \kL(E,F)}$ for the Fr\'echet derivative of a Fr\'echet differentiable map~${\Phi:\clu\to F}$.
  Given \(r\geq 0\),  ${\rm BUC}^{r}(\R^N) $ denotes  the Banach space of functions with bounded continuous derivatives of order less or equal to $\lfloor r\rfloor :=\max\{k\in\N,\ k\leq r\}$ and uniformly $(r-\lfloor r\rfloor)$-H\"older continuous derivatives of order $\lfloor r\rfloor$.  We also define ${\rm BUC}^{\infty}(\R^N) $ as the intersection of all spaces ${\rm BUC}^{r}(\R^N) $ with $r\geq0$, and  ${\rm C}_0^{\infty}(\R^N)$ is its subspace  consisting of functions with compact support.
  Moreover, we write $e_j$ for the standard basis vectors of $\R^n$  whenever $n\geq j$, with components $(e_j)_i=\delta_{ij}$,  $1\leq i,j\leq n$ (where $\delta_{ij}$ is the Kronecker delta). 
 
 To economize notation, we fix  the function 
\begin{equation}\label{deefphi}
\bar{\phi} \in \mathrm{C}^\infty([0, \infty)) \quad \text{given by} \quad \bar{\phi}(x) = (1 + x)^{-(N+1)/2} \quad \text{for} \, x \geq 0.
\end{equation}

 In our arguments we will use the interpolation property
  \begin{align}\label{IP}
[H^{r_0}(\mathbb{R}^N),H^{r_1}(\mathbb{R}^N)]_\theta=H^{(1-\theta)r_0+\theta r_1}(\mathbb{R}^N),\qquad\theta\in(0,1),\,  0\leq r_0\leq r_1<\infty,
\end{align}
where $[\cdot,\cdot]_\theta$ denotes the complex interpolation functor of exponent $\theta$;  see, e.g., \cite{BL76}.

 Furthermore,  throughout the paper we make repeated use of the following norm equivalences on the spaces $H^r(\R^N)$.
 
 For any $r\geq 1$, there is a constant $C_0=C_0(r)>1$ such that
\begin{align}
\label{equivgrad}
C_0^{-1}\|h\|_{H^r}&\leq \|h\|_2+\|\nabla h\|_{H^{r-1}}\leq C_0\|h\|_{H^r},\qquad h\in H^r(\R^N).
\end{align}

 Moreover, for any  $k\in\N$ and $\alpha\in(0,1)$, there is a constant $C_1=C_1(k,\alpha)>1$ such that
\begin{align}
\label{equivHs}
C_1^{-1}\|h\|_{H^{k+\alpha}}&\leq \|h\|_2+\sum_{i=1}^N\big[\p_i^{k}h\big]_{H^\alpha}
\leq C_1\|h\|_{H^{k+\alpha}},\qquad h\in H^{k+\alpha}(\R^N),
\end{align}
  where the seminorm $[\cdot]_{H^\alpha}$ is given by 
\be\label{semiHs}
[u]_{H^\alpha}^2:=\int_{\R^N}\frac{\|\tau_\zeta u-u\|_2^2}{|\zeta|^{N+2\alpha}}\,{\rm d}\zeta 
\qquad\text{ with}\quad \tau_\zeta u:=u(\cdot+\zeta).
\ee

We will denote by $|\s^N|$ the $N$-dimensional (hyper)surface area of the unit sphere $\s^N$ in~${\R^{N+1}}$.
Where no confusion is likely, summation is carried out over indices occurring twice in a product without indicating this.

Where appropriate, we will  shorten notation by writing $\llbracket A,B\rrbracket$ for the commutator of two linear operators $A$ and $B$, and $\llbracket \varphi,A\rrbracket$ for the commutator of $A$ and the multiplication with a function $\varphi$, i.e. 
\begin{align*}
    \llbracket A,B\rrbracket &:=AB-BA,\\
    \llbracket \varphi,A\rrbracket  [h]&:=\varphi A[h]-A[\varphi h].
\end{align*}

%%%%%%%%%%%%%%%%%%%%%%%%%%%%%%%%%%%%%%%%%%%%%%%%%%
%%%%%%%%%%%%%%%%%%%%%%%%%%%%%%%%%%%%%%%%%%%%%%%%%%
%%%%%%%%%%%%%%%%%%%%%%%%%%%%%%%%%%%%%%%%%%%%%%%%%%
%%%%%%%%%%%%%%%%%%%%%%%%%%%%%%%%%%%%%%%%%%%%%%%%%%
 \section{Unique solvability  for  the fixed-time problem}\label{Sec:2}
%%%%%%%%%%%%%%%%%%%%%%%%%%%%%%%%%%%%%%%%%%%%%%%%%%
%%%%%%%%%%%%%%%%%%%%%%%%%%%%%%%%%%%%%%%%%%%%%%%%%%
%%%%%%%%%%%%%%%%%%%%%%%%%%%%%%%%%%%%%%%%%%%%%%%%%%
%%%%%%%%%%%%%%%%%%%%%%%%%%%%%%%%%%%%%%%%%%%%%%%%%%
In  this section we prove that the interface between the fluids  determines  the velocity and the pressure in the fluid layers  at each fixed time~${t\geq0}$. This is a consequence of the unique solvability result for the boundary value problem \eqref{MuskatBVP'}; see Proposition~\ref{P:1}.

In the following we fix $f \in H^{s}(\R^N)$,  with $s$ satisfying \eqref{eq:s}, and set 
\begin{align}\label{opm}
\0^\pm&:=\{(x,y)\in\R^N\times\R\,:\, y\gtrless f(x)\}\quad\text{and}\quad\G:=\p\0^\pm:=\{(x, f(x))\,:\, x\in\R^N\}.
\end{align}
Then $\G$ is the image of the diffeomorphism $\Xi:=\Xi_f:=({\rm id\,}_{\R^N},f):\R^N\to\G.$
Let further
\begin{equation}\label{omeganu}
\omega:=1+|\nabla f|^2\qquad\text{and}\qquad\nu:=\wt \nu \circ\Xi=(\nu^1,\ldots,\nu^{N+1})=\Big(-\frac{\nabla f}{\sqrt{\omega}},\frac{1}{\sqrt{\omega}}\Big).
\end{equation}
We also set
 \begin{equation}\label{rzx}
 z_\xi :=(\xi,f(\xi))\in\G\qquad\text{for $\xi\in\R^N$.}
 \end{equation}

Let 
\begin{equation}\label{gamma}
\gamma^\pm:=k g\rho^\pm\in\R  \qquad\text{and}\qquad q^\pm(x,y):=\cfrac{k}{\mu^\pm}  p^\pm(x,y)+ \cfrac{\gamma^\pm}{\mu^\pm} y,\quad(x,y)\in\0^\pm.
\end{equation}
With this substitution the boundary value problem
\begin{equation}\label{MuskatBVP}
\left.
\arraycolsep=1.4pt
\begin{array}{rclll}
v^\pm+\cfrac{k}{\mu^\pm}\nabla \big(p^\pm+g\rho^\pm y\big)&=&0&\text{in $\0^\pm$}\\[1ex]
{\rm div\,} v^\pm&=&0&\mbox{in $\Omega^\pm$,}\\[1ex]
[p]&=&0 &\mbox{on $\Gamma$,}\\[1ex]
[v]\cdot\wt\nu&=&0 &\mbox{on $\Gamma$,}\\[1ex]
 (v^\pm,p^\pm) (x,y)+(0,\rho^\pm gy)&\to&0&\mbox{for $|(x,y)|\to\infty$,}
\end{array}\right\}
\end{equation} 
whose solution determines the motion of the interface  via the kinematic boundary condition~\eqref{MuskatEE}$_1$, may be recast, setting $\varphi:= [\gamma]f\in H^{s}(\R^N)$, as
\begin{equation}\label{MuskatBVP'}
\left.
\arraycolsep=1.4pt
\begin{array}{rclll}
v^\pm+\nabla q^\pm&=&0&\text{in $\0^\pm$}\\[1ex]
{\rm div\,} v^\pm&=&0&\mbox{in $\Omega^\pm$,}\\[1ex]
[\mu q]&=&\varphi\circ \Xi^{-1}&\mbox{on $\Gamma$,}\\[1ex]
[v]\cdot\wt\nu&=&0 &\mbox{on $\Gamma$,}\\[1ex]
(v^\pm,q^\pm) (x,y)&\to&0&\mbox{for $|(x,y)|\to\infty$.}
\end{array}\right\}
\end{equation}

We are going to solve this problem by representing $q^\pm$ as a double-layer potential generated by a suitable density $\beta\circ \Xi^{-1}$ on the interface $\Gamma$. The corresponding integral operator~${\cV :=\cV(f)[\beta]}$  with $\cV=(\cV_1,\ldots,\cV_{N+1})$ for the representation of $v^\pm$ is given by
\be\label{defcV}
\cV_i(f)[\beta](z):=\frac{1}{|\s^N|}\int_{\R^N}K_{ij}(z,\xi)\p_j\beta(\xi)\, {\rm d}\xi,\qquad 1\leq i\leq N+1,
\ee
for  $z=(x,y)\in(\R^N\times\R)\setminus\Gamma$,  where, given $\xi\in\R^N$, we set (recalling~\eqref{rzx})
\begin{subequations}\label{Kij}
\begin{align}\label{Kija}
K_{ij}(z,\xi):=&K_{f,ij}(z,\xi)
  :=\frac{\big(-(x-\xi)\cdot\nabla f(\xi)+y-f(\xi)\big)\delta_{ij}+(x_j-\xi_j)\p_if(\xi)}{|z-z_\xi|^{N+1}}
  \end{align}
 for $1\leq i, j \leq N$,  and
  \begin{align}\label{Kijb}
   K_{(N+1)\,j}(z,\xi):=K_{f,(N+1)\,j}(z,\xi)
  :=\frac{-(x_j-\xi_j)}{|z-z_\xi|^{N+1}},\qquad 1\leq j\leq N.
\end{align}
\end{subequations}
Note that if $\beta \in {\rm BUC}^{1+\alpha}(\mathbb{R}^N) \cap  W^1_p(\mathbb{R}^N)$ for some $\alpha \in (0,1)$ and $p \in (1,\infty)$  (so, in particular, if $\beta\in H^{s}(\R^N)$), then, by Proposition \ref{P:W1}, it follows that
\[\cV^\pm:=\cV(f)[\beta]|_{\0^\pm}\in {\rm C}(\overline{\0^\pm}),\]
and the  limits of $\cV^\pm$  on $\Gamma$ are given by 
\begin{equation}\label{forvelo'}
  \begin{aligned}&\cV^\pm\circ\Xi(x)\\
  &=\frac{1}{|\s^N|}\PV\int_{\R^N}  
\frac{1}{|z_x-z_\xi|^{N+1}}
\begin{pmatrix}
(z_x-z_\xi)\cdot(-\nabla f(\xi),1)\nabla\beta(\xi)+(x-\xi)\cdot\nabla\beta(\xi)\nabla f(\xi)\\[1ex]
-(x-\xi)\cdot\nabla\beta(\xi)
\end{pmatrix}^\top
\,{\rm d}\xi\\[1ex]
&\quad\pm\frac{1}{2}
\begin{pmatrix}
\displaystyle\nabla \beta-\frac{(\nabla f\cdot\nabla \beta) \nabla f}{1+|\nabla f|^2},
\displaystyle \frac{ \nabla f\cdot\nabla \beta }{1+|\nabla f|^2}
\end{pmatrix}(x),\qquad x\in\R^N.
  \end{aligned}
\end{equation}

The  results on the boundary value problem \eqref{MuskatBVP'} 
(with general inhomogeneity~$\varphi$) are summarized in the following  proposition.

\begin{prop}\label{P:1}
 Assume \eqref{eq:s} and let $f,\varphi\in H^s(\R^N)$.
\begin{itemize}
\item[\rm (i)] {\bf\em(The integral equation)}
Let $\bD(f)$ be the double-layer potential defined in \eqref{dlpot} and~$a_\mu$ as in \eqref{defamu}. Then the singular integral equation 
\begin{equation}\label{resolvent?}
 \frac{\beta}{2}+a_\mu\bD(f)[\beta]=-\frac{\varphi}{\mu^++\mu^-}
  \end{equation}
  has precisely one solution $\beta=\beta_\varphi\in H^s(\R^N)$. 
  \item [\rm (ii)] {\bf\em (Representation of the solution)}  Let $\beta_\varphi\in H^s(\R^N)$ denote the unique solution to \eqref{resolvent?} and define $(v,q): \R^{N+1}\setminus\Gamma\to\R^{N+1}\times\R$ by
  \begin{align*}
  v(z)&:=\cV(f)[\beta_\varphi](z),\\
  q(z)&:=-\frac{1}{|\s^N|}\int_{\R^N}G(z,\xi)\beta_\varphi(\xi)\, {\rm d}\xi
  \end{align*}
   for  $z=(x,y)\in(\R^N\times\R)\setminus\Gamma$, where, given $\xi\in\R^N$, we set
  \[G(z,\xi):=G_f(z,\xi) :=\frac{-(x-\xi)\cdot\nabla f(\xi)+y-f(\xi)}{|z-z_\xi|^{N+1}}.\]
  Then $(v^\pm,q^\pm):=(v,q)|_{\0^\pm}$ is a solution  to  \eqref{MuskatBVP'}
such  that
\begin{equation}\label{cond}
{v^\pm\in {\rm C}(\overline{\0^\pm})\cap {\rm C}^1(\0^\pm)},\qquad q^\pm\in {\rm C}^1( \ov{\0^\pm})\cap {\rm C}^2(\0^\pm).
\end{equation}
\item[\rm (iii)] {\bf\em (Uniqueness)}
The solution given in {\rm (ii)} is the only solution to \eqref{MuskatBVP'} in the space indicated in \eqref{cond}.
\end{itemize}
\end{prop}
\begin{proof}
{\noindent (i):} This follows directly from Theorem \ref{T:51} below. \medskip

{\noindent (ii):}  Set  $K(z,\xi):=(K_{ij}(z,\xi))\in\R^{(N+1)\times N}$ for $z=(x,y)\in\R^{N+1}\setminus\Gamma$ and $\xi\in\R^N$; see \eqref{Kij}.
As $f\in {\rm BUC}^{s-N/2}(\R^N)\hookrightarrow {\rm BUC}^1(\R^N)$, we have

\[ (K,G)(\cdot,\xi)\in {\rm C}^\infty\big(\R^{N+1}\setminus\Gamma,\R^{(N+1)\times N}\times\R\big)\]
for each fixed $\xi\in\R^N$. Moreover, for each $\alpha\in\N^N$, it holds that $\p_z^\alpha (K,G)(z,\xi)=O(|\xi|^{-N})$ as~$|\xi|\to \infty$, locally uniformly in
$z\in\R^{N+1}\setminus\Gamma$.  In view of $\beta_\varphi\in H^1(\R^N)$, the theorem on the differentiation of parameter integrals ensures that  $v$ and $q$ are both well-defined and smooth in $\R^{N+1}\setminus\Gamma$. 

Noticing that for each fixed $z\in\R^{N+1}\setminus\Gamma$ all mappings $K_{ij}(z,\cdot)$ belong to~$H^1(\R^N)$, integration by parts leads to
\be
\cV_i(f)[\beta_\varphi](z)=-\frac{1}{|\s^N|}\int_{\R^N}\p_{\xi_j}K_{ij}(z,\xi) \beta_\varphi(\xi)\, {\rm d}\xi,\qquad 1\leq i\leq N+1,
\ee
 and  \eqref{MuskatBVP'}$_1$  is a consequence of the  identities
\[\p_{\xi_j}K_{ij}(z,\xi)+\p_{z_i}G(z,\xi)=0, \qquad 1 \leq i\leq N+1,\quad\xi\in\R^N.\]
Eq. \eqref{MuskatBVP'}$_2$ immediately follows from the identity
\[\sum_{i=1}^{N+1}\p_{z_i}K_{ij}(z,\xi)=0,\qquad  1\leq j\leq N,
\quad z\in\R^{N+1}\setminus\Gamma,\quad \xi\in\R^N.\]

Since   $f, \beta_\varphi\in {\rm BUC}^{s-N/2}(\R^N)$ and $ \beta_\varphi,\,\p_j\beta_\varphi\in {\rm BUC}^{s- s_c}(\R^N)
\cap L_2(\R^N)$, ~${1\leq j\leq N}$,  by~\eqref{eq:s},  we may apply Proposition \ref{P:W1} to verify the boundary conditions~\eqref{MuskatBVP'}$_3$-\eqref{MuskatBVP'}$_4$. 
Indeed, by this proposition and the definition of $\beta_\varphi$,
\begin{align*}
[\mu q]\circ \Xi&=(\mu^+q^+-\mu^-q^-)\circ \Xi=-\big(\mu^+(\bD(f)[\beta_\varphi]+ \beta_\varphi/2)
 +\mu^-(\bD(f)[ \beta_\varphi]-\beta_\varphi/2)\big)\\[1ex]
&=-(\mu^++\mu^-)( \beta_\varphi/2+a_\mu\bD(f)[\beta_\varphi])=\varphi
\end{align*}
in $\R^N$, and,  recalling \eqref{forvelo'},
\[[v]\circ\Xi=\left(\nabla \beta_\varphi-\frac{\nabla f\cdot\nabla \beta_\varphi}{1+|\nabla f|^2}\nabla f,
\frac{\nabla f\cdot\nabla \beta_\varphi}{1+|\nabla f|^2}\right)\quad\text{ in $\R^N$,}\]
hence the jump conditions \eqref{MuskatBVP'}$_3$ and \eqref{MuskatBVP'}$_4$ are satisfied.

Finally, Proposition \ref{P:W2} ensures the validity of the far-field condition \eqref{MuskatBVP'}$_5$.\medskip

{\noindent (iii):} We prove that if $(v^\pm,q^\pm)$ satisfies \eqref{cond} and solves the boundary value problem
\begin{equation}\label{Muskathom}
\left.
\arraycolsep=1.4pt
\begin{array}{rclll}
v^\pm+\nabla q^\pm&=&0&\text{in $\0^\pm$}\\[1ex]
{\rm div\,} v^\pm&=&0&\mbox{in $\Omega^\pm$,}\\[1ex]
[\mu q]&=&0&\mbox{on $\Gamma$,}\\[1ex]
[v]\cdot\wt\nu&=&0 &\mbox{on $\Gamma$,}\\[1ex]
 (v^\pm,q^\pm) (x,y)&\to&0&\mbox{for $|(x,y)|\to\infty$,}
\end{array}\right\}
\end{equation} 
then $(v^\pm,q^\pm)\equiv 0$.  To this end, we define
\[w:= \mu^+q^+{\bf 1}_{\bar\0^+}+ \mu^-q^-{\bf 1}_{\0^-}\]
 and note that, due to \eqref{Muskathom}$_3$, $w$ is continuous
and furthermore 
$w\in H^1_{\rm loc}(\R^{N+1})$. 

For any $\psi \in H^1(\R^{N+1})$ with compact support we have from \eqref{Muskathom}$_{1,2,4}$
\begin{equation}
\begin{aligned}\label{muharm}
&\int_{\0^+}\mu^-\nabla w\cdot\nabla\psi\,{\rm d}z + \int_{\0^-}\mu^+\nabla w\cdot \nabla\psi\,{ \rm d}z\\
&= \mu^-\mu^+\int_\Gamma( \nabla q^-- \nabla q^+)\cdot\tilde\nu\psi\,{\rm d}\Gamma
=\mu^-\mu^+\int_\Gamma [v]\cdot\tilde\nu\psi\,{ \rm d}\Gamma=0.
\end{aligned}
\end{equation}
Let now $\eps>0$ be chosen arbitrary and set $\psi:=\max\{w-\eps,0\}$. By \eqref{Muskathom}$_5$, $\psi$ has compact support. Furthermore,  $\psi\in H^1(\R^{N+1})$ with
\[\nabla\psi={\bf 1}_{\{w>\eps\}}\nabla w\;\text{a.e.,}\]
where $\{w>\eps\}:=\{z\in\R^{N+1}\,:\,w(z)>\eps\}$; see, e.g. 
\cite[Theorem II.A.1]{KiSt00}.     
Applying \eqref{muharm} with this choice of $\psi$ yields 
\[\int_{\0^+\cap\{w>\eps\}}\mu^-|\nabla w|^2\,{\rm d}z
+\int_{\0^-\cap\{w>\eps\}}\mu^+|\nabla w|^2\,{ \rm d}z=0.\]
Suppose the set $\{w>\eps\}$ is nonempty. Then $w$ is constant on each of its connected components, contradicting $w=\eps$ on $\p\{w>\eps\}$ as $w$ is continuous. 

Thus $w\leq\eps$, and, as $\eps>0$ was arbitrary, $w\leq 0$. Upon replacing $(v^\pm,q^\pm)$
by $-(v^\pm,q^\pm)$ we obtain $w\equiv 0$. This proves the statement.
\end{proof}

 %%%%%%%%%%%%%%%%%%%%%%%%%%%%%%%%%%%%%%%%%%%%%%%%%%
%%%%%%%%%%%%%%%%%%%%%%%%%%%%%%%%%%%%%%%%%%%%%%%%%%
%%%%%%%%%%%%%%%%%%%%%%%%%%%%%%%%%%%%%%%%%%%%%%%%%%
%%%%%%%%%%%%%%%%%%%%%%%%%%%%%%%%%%%%%%%%%%%%%%%%%%
 \section{On the resolvent of  the double layer potential  $\bD(f)\in\kL(L_2(\R^N))$}\label{Sec:3}
%%%%%%%%%%%%%%%%%%%%%%%%%%%%%%%%%%%%%%%%%%%%%%%%%%
%%%%%%%%%%%%%%%%%%%%%%%%%%%%%%%%%%%%%%%%%%%%%%%%%%
%%%%%%%%%%%%%%%%%%%%%%%%%%%%%%%%%%%%%%%%%%%%%%%%%%
%%%%%%%%%%%%%%%%%%%%%%%%%%%%%%%%%%%%%%%%%%%%%%%%%%
 
  In this section, we define the double layer potential $\bD(f)\in\kL(L_2(\R^N))$ 
 for the Laplace operator associated with the unbounded graph ${\Gamma = \{y = f(x)\}}$, where $f:\R^N\to\R$ is a Lipschitz continuous function, 
 and investigate the intersection of its resolvent set with the real line.

 \subsection*{Generalized Riesz transforms}
We start by introducing some notation and a class of generalized Riesz transforms used throughout the paper, which may prove useful also in other contexts where layer potentials in a noncompact graph geometry are considered.

Given  $1\leq p\in\N$  and $x=(x_1,\ldots,x_p)\in\R^p$,   we set
\[
x^{\ov 2}:=(x_1^2,\ldots,x_p^2).
\]
For each $\phi\in {\rm C}^\infty([0,\infty)^p)$, $\nu \in\N^N$, and $n\in\N$ with $n+|\nu|$ odd,
  we define the singular integral operator
  \[B^\phi_{n,\nu}:=B^\phi_{n,\nu}(a)[b,\cdot]:=B^\phi_{n,\nu}(a)[b_1,\ldots,b_n,\cdot]\]
by
\begin{equation}\label{operatorsB}
B^\phi_{n,\nu}(a)[b,\beta](x):=\frac{1}{|\s^N|} \PV\int_{\R^N}\phi\Big((D_{[x,\xi]} a)^{\ov 2} \Big)\bigg[\prod_{i=1}^n D_{[x,\xi]} b_i\bigg]\frac{\xi^\nu}{|\xi|^{|\nu|}}
\frac{\beta(x-\xi)}{|\xi|^N}\,{\rm d}\xi,
\end{equation}
  where  $a=(a_1,\ldots,a_p):\R^N\to\R^p$ and $b=(b_1,\ldots,b_n):\R^N\to\R^n$ are Lipschitz continuous functions,~$\beta\in L_2(\R^N),$ and~$x\in\R^N$.
  We use the shorthand notation 
  \[
\delta_{[x,\xi]} u:=u(x)-u(x-\xi),\qquad 
D_{[x,\xi]}u:=\frac{\delta_{[x,\xi]}u}{|\xi|}.
  \]
The operators $B^\phi_{n,\nu}$ are generalized Riesz transforms with a singular integral kernel that depends nonlinearly on $a$ and linearly on $b_i$ for $1\leq i\leq n$.  
To simplify notation when repeated linear arguments $b:\R\to\R$ occur we will write 
\begin{equation*} 
    b^{[k]}:=(b,\ldots,b):\R\to \R^k,\qquad k\in\N,
\end{equation*}
with the additional definition $B_{n,\nu}^\phi(a)[b^{[0]},b',\beta]:=B_{n,\nu}^\phi(a)[b',\beta]$.

We emphasize that the  operators involved in the analysis of the Muskat problem can be expressed using (a particular version of) the operators $B^\phi_{n,\nu}$ with $p=1$.  
Specifically, for~$p=1$, $a = b_1 = \ldots = b_n = f$,
$\phi \in {\rm C}^\infty( [0,\infty))$, and~${\nu \in \mathbb{N}^N}$ with $n + |\nu|$ being odd, we define 
\begin{equation}\label{trueRO}  
\sfB^\phi_{n,\nu}(f) := B^\phi_{n,\nu}(f)[f^{[n]}, \cdot].  
\end{equation}
  However, since it will be useful to represent the difference of two operators~$B^\phi_{n,\nu}(a)$ and~$B^\phi_{n,\nu}(\wt a)$, with~$a=(a_1,\ldots, a_p)$ and~$\wt a=(\wt a_1,\ldots,\wt a_p),$   as
  \begin{equation}\label{difference}
 \big(B^\phi_{n,\nu}(a)-B^\phi_{n,\nu}(\wt a)\big)[b,\beta]=\sum_{i=1}^p B^{\phi^i}_{n+2,\nu}(a,\wt a )[a_i-\wt a_i, a_i+\wt a_i,b,\beta],
  \end{equation}
  with~$\phi^i\in {\rm C}^\infty([0,\infty)^{2p})$ given by the formula
   \begin{equation}\label{gis}
\phi^i(x,y)=\int_0^1\p_i\phi(sx +(1-s)y)\, \,{\rm d}s,\qquad x,y\in [0,\infty)^{p}, \quad 1\leq i\leq p, 
  \end{equation}
  it is natural to consider $p\geq 1$.
Let us point out that the classical  Riesz transforms \cite{St93}
  \begin{equation*} 
R_k[\beta](x):=2B^1_{0,e_k}(0)[\beta](x):=\frac{2}{|\s^N|} \PV\int_{\R^N} \frac{\xi_k}{|\xi|}\frac{\beta(x-\xi)}{|\xi|^N}\,{\rm d}\xi,\qquad 1\leq k\leq N,
\end{equation*}
 belong to the class of operators introduced in~\eqref{operatorsB}.

  In Appendix~\ref{Sec:C} we establish the following results which are used  in the analysis below.
   \begin{lemma}\label{L:MP0}
  Given $M>0$, there is a  constant $C>0$ such that for all Lipschitz continuous functions $a:\R^N\to\R^p$  and $b:\R^N\to\R^n$
   with $\|\nabla a\|_\infty\leq M$  we have
  \begin{equation}\label{E:MP0} 
\|B^\phi_{n,\nu}(a)[b,\cdot]\|_{\kL(L_2(\R^N))}\leq C\prod_{i=1}^n\|\nabla b_i\|_\infty .
  \end{equation}
  \end{lemma} 
 
It follows directly from Lemma~\ref{L:MP0} that:
  \begin{cor}\label{C:MP0} \,
  \begin{itemize}
  \item[(i)]   Given $a\in {W^1_\infty(\R^N)}^ p$, we have $B^\phi_{n,\nu}(a)\in\kL^n_{ \rm sym}(W^1_\infty(\R^n), \kL(L_2(\R^N)))$.
   \item[(ii)]  $[ a\mapsto B^\phi_{n,\nu}(a)]\in{\rm C}^{1-}({W^1_\infty(\R^N)}^p,\kL^n_{ \rm sym}(W^1_\infty(\R^n), \kL(L_2(\R^N)))).$
  \end{itemize}
  \end{cor}
  \begin{proof}
  The claim (i) is a direct consequence of Lemma~\ref{L:MP0} and (ii) follows from Lemma~\ref{L:MP0} and~\eqref{difference}.
  \end{proof}

   For Lipschitz continuous functions $f:\R^N\to\R$, consider the equivalence classes ``up to constants'', i.e. 
   \[[f]:=\{f+c\,:\,c\in\R\}.\]
   The space of these equivalence classes will be denoted by $\Lip(\R^N)$ and   given the norm $\big[[f]\mapsto\|\nabla f\|_\infty\big]$. It is natural to consider the operators $B_{n,\nu}^\phi$ with $a$ and $b$ replaced by their equivalence classes. 
   With some abuse of notation, we reformulate Corollary \ref{C:MP0} as
   \be\label{Lipclasses}
   \big[\, [a]\mapsto  B^\phi_{n,\nu}(a)\,\big]\in{\rm C}^{1-}({\Lip(\R^N)}^p,\kL^n_{\rm sym}(\Lip(\R^n), \kL(L_2(\R^N)))).
   \ee

\subsection*{The double layer potential}
The double layer $\bD(f)$ for the Laplace operator associated with the Lipschitz graph ${\Gamma = \{y = f(x)\}}$ is defined by the formula
 \begin{align}\label{dlpot}
\bD(f)[\beta](x)&:=\frac{1}{|\s^N|}\PV\int_{\R^N}\frac{\delta_{[x,\xi]} f-\xi\cdot\nabla f(x-\xi)}{\big(|\xi|^2+(\delta_{[x,\xi]} f)^2 \big)^{\frac{N+1}{2}}}\beta(x-\xi)\,{\rm d}\xi
\end{align}
for $\beta\in L_2(\R^N)$ and $x\in\R^N$. 
Observe that whenever  $f\in {\rm BUC}^{r}(\R^N)$ for some $r>1$, the integral operator is weakly singular only.
 Since $\bD(f)$ can be expressed in terms of the operators defined in \eqref{trueRO},  that is, with $\bar\phi$ from \eqref{deefphi},
\begin{equation}\label{sionsD} 
\bD(f)[\beta]=\sfB_{1,0}^{\bar\phi}(f)[\beta]-\sum_{i=1}^N \sfB_{0, e_i}^{ \bar\phi}(f)[\beta\p_i f],
\end{equation}
 Lemma~\ref{L:MP0} ensures that $\bD(f)\in \kL(L_2(\R^N))$.
We  note that its $L_2$-adjoint~$\bD(f)^*$ is  given by
 \begin{align*}
\bD(f)^*[\beta](x)&:=\frac{1}{|\s^N|}\PV\int_{\R^N}\frac{-\delta_{[x,\xi]} f+\xi\cdot\nabla f(x)}{\big(|\xi|^2+(\delta_{[x,\xi]} f)^2 \big)^{\frac{N+1}{2}}}\beta(x-\xi)\,{\rm d}\xi 
\end{align*}
for $\beta\in L_2(\R^N)$ and $x\in\R^N$. 
It holds  that
\begin{equation}\label{sionsD*} 
\bD(f)^*[\beta]=-\sfB_{1,0}^{\bar\phi}(f)[\beta]+\sum_{i=1}^N \p_i f\sfB^{\bar\phi}_{0, e_i}(f)[\beta].
\end{equation}

The main goal of this section is to establish the following theorem.

\begin{thm}\label{T:1}
Given $M>0$, there exists a constant $C=C(M)\in(0,1)$ such that for all~$a\in[-2,2]$,   $\beta\in L_2(\R^N)$, 
and~$f\in {\rm BUC}^1(\R^N)$  with~${\|\nabla f\|_\infty\leq M}$  we have 
\begin{equation}\label{resolvent}
\|(1-a\bD(f))[\beta]\|_2\geq C\|\beta\|_2.
\end{equation}
Moreover, $1-a\bD(f)\in\kL(L_2(\R^N))$ is an isomorphism for all $f\in {\rm BUC}^1(\R^N) $ and~${a\in[-2,2]}$.
\end{thm}

Before establishing Theorem~\ref{T:1}, we prove  the following preparatory result.
\begin{lemma}\label{L:31}
Given $f\in {\rm BUC}^\infty(\R^N)$ and $\beta\in{\rm C}_0^\infty(\R^N)$, let $w:\R^{N+1}\setminus\G\to\R$ be given by
 \begin{equation}\label{defv}
w(z)=-\frac{1}{(N-1)|\s^N|}\int_{\G}\frac{1}{|z-\ov z|^{N-1}}\wt\beta(\ov z)\wt\nu^{N+1}(\ov z)\,{\rm d}\G(\ov z)
\end{equation}
 for $z\in\R^{N+1}\setminus\G$, where $\wt\beta=\beta\circ\Xi^{-1}$ and $\wt \nu^{N+1}$ is the $(N+1)$-th component of $\wt \nu$; see~\eqref{omeganu}.
 Then $w^{\pm}:=w|_{\0^\pm}\in{\rm C}^\infty(\0^\pm)$ and~${\nabla w^\pm\in{\rm C}(\overline{\0^\pm})}$.
Moreover,  there exist constants $C, \, R\geq 1$ such that 
  \begin{equation}\label{estimate}
|\nabla w(z)|\leq C\frac{\|\beta\|_1}{|z|^N}\qquad \text{for all $z\in\R^{N+1}\setminus\G$ with $|z|\geq  R$.}
  \end{equation}
\end{lemma}
\begin{proof}
The function $w$ is obviously smooth in $\R^{N+1}\setminus\G$ with
\begin{equation}\label{gradw0}
\nabla w(z)=\frac{1}{|\s^N|}\int_{\G}\frac{z-\ov z}{|z-\ov z|^{N+1}}\wt\beta(\ov z)\wt\nu^{N+1}(\ov z)\,{\rm d}\G(\ov z).
\end{equation}
In now follows from Proposition~\ref{P:W1} that $ {\nabla w^\pm\in{\rm C}(\overline{\0^\pm})}$  and, for $ x\in\R^N,$ we have
\begin{equation}\label{gradw}
(\nabla w^\pm)|_\G\circ\Xi(x)=\frac{1}{|\s^N|}\PV\int_{\G}\frac{z_x-\ov z}{|z_x-\ov z|^{N+1}}\wt\beta(\ov z)\wt\nu^{N+1}(\ov z)\,{\rm d}\G(\ov z)\pm\frac{\nu\nu^{N+1}\beta}{2}(x).
\end{equation}
The claim \eqref{estimate} is now a direct consequence  of \eqref{gradw0} (we omit the elementary details). 
\end{proof}

We conclude this section with the proof of Theorem~\ref{T:1}.
\begin{proof}[Proof of Theorem~\ref{T:1}]
Let $M>0$. 
We first  prove that there is a constant $C=C(M)\in(0,1)$ such that for all~$a\in[-2,2]$,   $\beta\in{\rm C}_0^\infty(\R^N)$, 
and~$f\in {\rm BUC}^\infty(\R^N)$  with~${\|\nabla f\|_\infty\leq M}$  we have 
\be\label{adresolvent}
\|(1-a\bD(f)^\ast)[\beta]\|_2\geq C\|\beta\|_2.
\ee
  Fix such $f$ and $\beta$, let $w^\pm$ be the functions defined in Lemma~\ref{L:31}, and set
\[W^\pm:=2(\p_{N+1} w^{\pm})\nabla w^{\pm}-|\nabla w^{\pm}|^2 e_{N+1}.\]
Since ${\rm div \,} \nabla w^\pm=0$ in $\0^\pm$,  it follows that ${\rm div\,} W^\pm=0$ in $\0^\pm$.
Using  Stokes' formula  together with the estimate~\eqref{estimate}, we derive the Rellich identities
\begin{equation}\label{rellich}
\int_\G W^\pm  \cdot\wt \nu\,{\rm d}\G=
\int_{\G}  2\p_{N+1}w^{\pm }\nabla w^\pm\cdot \wt \nu-|\nabla w^\pm|^2\wt \nu^{N+1}\,{\rm d}\G =0.
\end{equation} 
Further, transforming \eqref{gradw} to $\Gamma$ and taking the normal component we obtain
\[\nabla w^\pm \cdot\wt \nu
=\wt\nu^{N+1}\Big(\pm\frac{\beta}{2}-\bD(f)^*[\beta]\Big)\circ\Xi^{-1}.
\]
We also define $F\in L_2(\R^N)^N$ as the (transformed and rescaled) tangential part of $\nabla w^\pm|_\G$, that is
\[
 \wt\nu^{N+1}F\circ\Xi^{-1}:=\nabla w^\pm|_\G-( \nabla w^\pm|_\G\cdot\wt \nu)\wt\nu.
\]
We then have
\[|\nabla w^\pm|^2=(\wt\nu^{N+1})^2\Big(\Big|\pm\frac{\beta}{2}-\bD(f)^*[\beta]\Big|^2+|F|^2\Big)\circ\Xi^{-1}\qquad\text{on $\G$}\]
and
\[\p_{N+1} w^\pm =(\wt\nu^{N+1})^2 \Big(\pm\frac{\beta}{2}-\bD(f)^*[\beta]\Big) \circ\Xi^{-1}+\wt\nu^{N+1}(F\cdot e_{N+1})\circ\Xi^{-1}\qquad\text{on $\G$.}\]
 Using these representations for $|\nabla w^\pm|^2$, $\p_{N+1} w^\pm$, and 
$\nabla w^\pm \cdot\wt \nu$ in
 \eqref{rellich}  we obtain, recalling~\eqref{omeganu}
\begin{equation}\label{rellich'}
\begin{aligned}
&\int_{\R^N} \Big[ \frac{1}{4\omega} (\pm\beta-2\bD(f)^*[\beta])^2 
+\frac{1}{\sqrt\omega} \big(\pm \beta -2\bD(f)^*[\beta]\big) (F\cdot e_{N+1})-\frac{1}{\omega} |F|^2\Big]\,{\rm d}\xi =0.
\end{aligned}
\end{equation}  
Consequently, there exists a   constant $C=C(M)\in(0,1)$ such that 
\begin{equation*}
C\|(\pm1 -2\bD(f)^*)[\beta]\|_2\leq\|F\|_2.
\end{equation*}
In view of $2\beta=(1 -2\bD(f)^*)[\beta]-(-1 -2\bD(f)^*)[\beta]$ we then get
\begin{equation}\label{usest}
C\|\beta\|_2\leq\|F\|_2 .
\end{equation}
For $a\in[-1,1]\setminus\{0\}$ we substitute
\[
 \pm\beta -2\bD(f)^*[\beta]=-\frac{(1 +2a\bD(f)^*)[\beta]-(1\pm a)\beta}{a}
\]
 in \eqref{rellich'} and  obtain
\begin{equation}\label{rellich3}
\begin{aligned}
\int_{\R^N}&\Big[\frac{1}{4\omega}\big[|(1 +2a\bD(f)^*)[\beta]|^2-2(1\pm a)\beta(1 +2a\bD(f)^*)[\beta]+|(1\pm a)\beta|^2\big]\\
&-\frac{a}{\sqrt\omega}(1 +2a\bD(f)^*)[\beta]F\cdot e_{N+1}+\frac{a(1\pm a)}{\sqrt\omega}\beta F\cdot e_{N+1}-\frac{a^2}{\omega} |F|^2\Big]\,{\rm d}\xi =0.
\end{aligned}
\end{equation}

We now multiply the identity \eqref{rellich3} with $+$ by $(1-a)$ and the identity \eqref{rellich3} with $-$ by~${-(1+a)}$ to find, after summing up the resulting identities, that
\begin{equation}\label{rellich4}
\begin{aligned}
&\int_{\R^N} \frac{1}{ 4\omega} |(1 +2a\bD(f)^*)[\beta]|^2\,{\rm d}\xi\\
&=\int_{\R^N}\Big[\frac{(1- a^2)}{ 4\omega}|\beta|^2 +\frac{a}{\sqrt\omega}(1 +2a\bD(f)^*)[\beta]F\cdot e_{N+1}+\frac{a^2}{\omega} |F|^2\Big]\,{\rm d}\xi.
\end{aligned}
\end{equation}
H\"older's inequality, Young's inequality, \eqref{usest}, and \eqref{rellich4} combined imply  there exists a  constant $C=C(M) \in(0,1)$ such that for all $a\in[-1,1]$ we have
 \[
\|(1 +2a\bD(f)^*)[\beta]\|_2^2\geq C((1-a^2)\|\beta\|_2^2+a^2\|F\|_2^2) \geq C\|\beta\|_2^2.
 \]
 Using a standard density argument, \eqref{sionsD*}, and Corollary~\ref{C:MP0}~(ii), we infer from the latter inequality  that   estimate~\eqref{adresolvent} holds for all~$a\in[-2,2]$,   $\beta\in L_2(\R^N)$, 
and~$f\in {\rm BUC}^1(\R^N)$  with~${\|\nabla f\|_{\infty} \leq M}$.
Moreover, since  $1-a\bD(f)^*\in\kL(L_2(\R^N))$ is an isomorphism for~${a=0}$, the method of continuity; see, e.g. \cite[Proposition I.1.1.1]{Am95}, together with \eqref{adresolvent} implies that, for each~$a\in[-2,2]$ and $f\in {\rm BUC}^1(\R^N) $, the operator
$1-a\bD(f)^*\in\kL(L_2(\R^N))$ (hence also~${1-a\bD(f)}$) is an isomorphism.
The claim~\eqref{resolvent} is now a straightforward consequence of~\eqref{adresolvent}.
\end{proof}

 %%%%%%%%%%%%%%%%%%%%%%%%%%%%%%%%%%%%%%%%%%%%%%%%%%
%%%%%%%%%%%%%%%%%%%%%%%%%%%%%%%%%%%%%%%%%%%%%%%%%%
%%%%%%%%%%%%%%%%%%%%%%%%%%%%%%%%%%%%%%%%%%%%%%%%%%
%%%%%%%%%%%%%%%%%%%%%%%%%%%%%%%%%%%%%%%%%%%%%%%%%%
 \section{on the resolvent of  the double layer potential $\bD(f)\in\kL(H^s(\R^N))$}\label{Sec:4}
%%%%%%%%%%%%%%%%%%%%%%%%%%%%%%%%%%%%%%%%%%%%%%%%%%
%%%%%%%%%%%%%%%%%%%%%%%%%%%%%%%%%%%%%%%%%%%%%%%%%%
%%%%%%%%%%%%%%%%%%%%%%%%%%%%%%%%%%%%%%%%%%%%%%%%%%
%%%%%%%%%%%%%%%%%%%%%%%%%%%%%%%%%%%%%%%%%%%%%%%%%%

We now assume $f \in H^s(\R^N)$,  with $ s$  satisfying \eqref{eq:s}.
Building on Theorem~\ref{T:1},   we  will obtain in Theorem~\ref{T:51} a parallel result on the resolvent set of $\bD(f)$ in $\kL(H^s(\R^N))$.

 \subsection*{Generalized Riesz transforms in $\kL(H^{s-1}(\R^N))$} 
To show that $\bD(f) \in \kL(H^s(\R^N))$, additional mapping properties for the generalized Riesz transforms ${B^\phi_{n,\nu}(a)[b,\cdot]}$ are required. 
These properties are presented in Lemma~\ref{L:MP1}-Lemma~\ref{L:MP1b} below, with their proofs provided in Appendix~\ref{Sec:C}  (Lemma~\ref{L:MP1} is actually a particular case of  the more general result stated in Lemma~\ref{Bestgen}).

\begin{lemma}\label{L:MP1}
Given $M>0$, there exists a   constant $C>0$ such that for all
 $a \in  H^s(\R^N)^p$ with $\|a \|_{H^{s}}\leq M$,     $b=(b_1,\ldots,b_n)\in H^{s}(\R^N)^n,$  and~${\beta\in H^{s-1}(\R^N)}$,
   the mapping~${B^\phi_{n,\nu}(a)[b,\beta]}$ belongs to $H^{s-1}(\R^N)$ and
  \begin{equation}\label{F:MP1} 
\|B^\phi_{n,\nu}(a)[b,\beta]\|_{H^{s-1}}\leq C\|\beta\|_{H^{s-1}}  \prod_{i=1}^n \|b_i\|_{H^{s}}.
  \end{equation}
  \end{lemma}
  
We now extend this statement to establish the smooth dependence of
$B^\phi_{n,\nu}$ on $a$. 
  Since we only require this result for $p=1$, we restrict ourselves to this case (although the result also holds for $p \geq 2$).
  \begin{lemma}\label{L:MP1d}
We have $[ a\mapsto B^\phi_{n,\nu}(a)]\in {\rm C}^{\infty}(H^s(\R^N),\kL^n_{\rm sym} (H^{s}(\R^n) , \kL(H^{s-1 }(\R^N))))$.
  \end{lemma}
  As a straightforward consequence of Lemma~\ref{L:MP1d}, the operators defined in \eqref{trueRO} satisfy
\begin{equation}\label{regtruero}
[ f\mapsto \sfB^\phi_{n,\nu}(f)]\in {\rm C}^{\infty}(H^s(\R^N),  \kL(H^{s-1}(\R^N))).
\end{equation}

 The proof of Lemma \ref{L:MP1} will also provide representations for spatial derivatives (a ``chain rule'') for (weak) spatial derivatives of $B^\phi_{n,\nu}(a)[b,\beta]$, whenever these exist in~$L_2(\R^N)$. In particular, this implies that the class of singular integral operators $B^\phi_{n,\nu}$ is closed under differentiation.   For the case $p=1$, which is sufficient for our purposes, we make this explicit in the following lemma.

\begin{lemma}\label{L:MP1b}
    Let $a\in H^s(\R^n)$, $b_i\in H^{s-\sigma_i}(\R^N)$, $1\leq  i\leq n$, and $\beta\in H^{s-1-\sigma_0}(\R^N)$  with
    \[\sigma_0,\ldots,\sigma_n\in[0,s-1] \qquad\text{and}\qquad \sigma_0+\ldots+\sigma_n=:\sigma\leq s-2.\]
    Then  $B^\phi_{n,\nu}(a)[b,\beta]\in H^{s-1-\sigma}(\R^N)$ and
    \begin{equation}\label{F:MP2a} 
 \begin{aligned}
 \p_j\big(B^\phi_{n,\nu}(a)[b,\beta]\big)&=B^\phi_{n,\nu}(a)[b,\p_j\beta]
 +\sum_{i=1}^n B^\phi_{n,\nu}(a)[ b_1,\ldots,b_{i-1}, \p_j b_i,  b_{i+1},\ldots,  b_n,   \beta]\\
 &\quad+2 B^{\phi'}_{n+2,\nu}(a)[\p_j a, a ,  b,  \beta],\qquad 1\leq j\leq N.
 \end{aligned}
  \end{equation}
\end{lemma}

In particular, we point out for further reference that,   given $f\in H^s(\R^N)$,
\be\label{commsfb}
\llbracket\p_j,\sfB_{n,\nu}^\phi(f)\rrbracket
=nB_{n,\nu}^\phi[\p_jf,f^{[n-1]},\cdot]+2B_{n+2,\nu}^{\phi'}[\p_jf,f^{[n+1]},\cdot],\qquad  1\leq j\leq N.
\ee

\subsection*{The double layer potential in $\kL(H^s(\R^N))$}\  Due to  \eqref{sionsD}, it is straightforward to deduce from Lemma~\ref{L:MP1} that $\bD(f) \in \kL(H^{s-1}(\R^N))$.
 Furthermore, we prove below that a stronger property holds, namely that $\bD(f) \in \kL(H^{s}(\R^N))$. 
 With Theorem~\ref{T:51}, we provide a  result
 which plays a crucial role in reformulating the Muskat problem as an evolution problem for the free interface
  between the fluids in Section~\ref{Sec:5}.
\begin{thm}\label{T:51}
Given $M>0$, there exists a constant $C=C(M)\in(0,1)$ such that for all~$a\in[-2,2]$ and   $f, \beta\in H^s(\R^N)$ 
  with~${\|f\|_{H^s}\leq M}$ we have $\bD(f)[\beta]\in H^s(\R^N) $ and
\begin{equation}\label{resolvents}
\|(1-a\bD(f))[\beta]\|_{H^s}\geq C\|\beta\|_{H^s}.
\end{equation}
Moreover, $1-a\bD(f)\in\kL(H^s(\R^N))$ is an isomorphism for all $f\in H^s(\R^N) $ and~${a\in[-2,2]}$.
\end{thm}

As a first step we prove that  $\bD(f)[\beta]\in H^s(\R^N) $. 
This regularity issue has  been considered  in the special case $N=1$ in  \cite[Proposition 2.3]{MM23}.
In  Lemma~\ref{L:51} below we establish  its counterpart for the case $N\geq 2$ considered here. 
 To this end we introduce a matrix-type singular integral operator $\cA(f)$ by setting, for $b=(b_1,\ldots,b_N) \in L_2(\R^N)^N$ and~$x\in\R^N$,
\begin{equation}\label{aof}
 \cA(f)[b](x)
 :=\frac{1}{|\s^N|}\PV\int_{\R^N} \frac{(z_x-z_\xi)\cdot(-\nabla f(\xi),1)b(\xi)-(x-\xi)\cdot b(\xi)\big(\nabla f(x)-\nabla f(\xi))}{|z_x-z_\xi|^{N+1}}\,{\rm d}\xi.
 \end{equation}
This operator can be expressed in terms of the generalized Riesz transforms defined in~\eqref{trueRO}.  
Indeed, with $\bar\phi$  from \eqref{deefphi} and  $\cA(f)=:(\cA_1(f),\ldots,\cA_N(f))$, for $1\leq k\leq N$ we have
\be\label{sionscA}
\cA_k(f)[b]=\sfB_{1,0}^{\bar\phi}(f)[b_k]+\sum_{i=1}^N  \Big(\sfB^{\bar\phi}_{0,e_i}(f)[\p_kfb_i-\p_i fb_k]-\p_kf \sfB^{\bar\phi}_{0,e_i}(f)[b_i]\Big).
\ee
 The representation~\eqref{sionscA} together with Lemma~\ref{L:MP1d} implies that
\begin{equation}\label{reg:cla}
    \mathcal{A}\in {\rm C}^{\infty}(H^s(\R^N),  \kL(H^{s-1}(\R^N),H^{s-1}(\R^N)^N)).
\end{equation}

The next results provides  a correlation between the double layer potential $\bD(f)$ and~$\cA(f)$.

\begin{lemma}\label{L:51}  Given $f\in H^s(\R^N)$ and $\beta\in H^{1}(\R^N)$, we have $\bD(f)[\beta]\in  H^{1}(\R^N)$ and
\begin{equation}\label{cder}
 \nabla (\bD(f)[\beta])=\cA(f)[\nabla \beta].
 \end{equation}
 Moreover,   $\bD(f)\in\kL(H^s(\R^N) )$ for $f\in H^s(\R^N)$.
\end{lemma}
\begin{proof} Let us first assume that $f,\,\beta\in {\rm C}^\infty_0(\R^N)$.
The   representation \eqref{sionsD} of $\bD(f)$ together  with  Lemma~\ref{L:MP1b} leads  us to the conclusion  that~$ \bD(f)[\beta]\in H^{1}(\R^N)$ with 
\begin{align*}
\p_k(\bD(f)[\beta])(x)&=\bD(f)[\p_k \beta](x)\\
&\quad+\frac{1}{|\s^N|}\int_{\R^N}\p_{x_k}\Bigg(\frac{ \delta_{[x,\xi]}f-\xi\cdot\nabla f(x-\xi)}{\big(|\xi|^2+(\delta_{[x,\xi]} f)^2 \big)^{\frac{N+1}{2}}}\Bigg)\beta(x-\xi)\,{\rm d}\xi
\end{align*}
for $1\leq k\leq N$ and $x\in\R^N$.
In view of the relation
\begin{equation*}
\p_{x_k}\Bigg(\frac{  \delta_{[x,\xi]}f-\xi\cdot\nabla f(x-\xi)}{\big(|\xi|^2+(\delta_{[x,\xi]} f)^2 \big)^{\frac{N+1}{2}}}\Bigg)
=-{\rm div}_\xi\Bigg(\frac{\xi\delta_{[x,\xi]}\p_k f}{\big(|\xi|^2+(\delta_{[x,\xi]} f)^2 \big)^{\frac{N+1}{2}}}\Bigg),\qquad x,\xi\in\R^N,\,\xi\neq 0,
\end{equation*}
we further have
\begin{align*}
&\int_{\R^N}\p_{x_k}\Bigg(\frac{  \delta_{[x,\xi]}f-\xi\cdot\nabla f(x-\xi)}{\big(|\xi|^2+(\delta_{[x,\xi]} f)^2 \big)^{\frac{N+1}{2}}}\Bigg)\beta(x-\xi)\,{\rm d}\xi\\
&=-\int_{\R^N}\nabla _\xi\Bigg(\frac{\xi\delta_{[x,\xi]}\p_k f}{\big(|\xi|^2+(\delta_{[x,\xi]} f)^2 \big)^{\frac{N+1}{2}}}\Bigg)\beta(x-\xi)\,{\rm d}\xi\\
&=-\int_{\R^N}\frac{\xi\delta_{[x,\xi]}\p_k f}{\big(|\xi|^2+(\delta_{[x,\xi]} f)^2 \big)^{\frac{N+1}{2}}}\nabla\beta(x-\xi)\,{\rm d}\xi ,
\end{align*}
where Stokes' theorem was applied in the last step.
Recalling the definition \eqref{aof} of $\cA(f)$, we have proved \eqref{cder} for $f,\,\beta\in {\rm C}^\infty_0(\R^N)$.
Since  ${\rm C}^\infty_0(\R^N)$ is dense in  $H^s(\R^N)$ and  in~$H^{1}(\R^N)$, the identity~\eqref{cder} follows  for general
~${f\in H^s(\R^N)}$  and  $\beta\in H^{1}(\R^N)$
by a standard density argument from  \eqref{sionsD}, \eqref{sionscA},  and Corollary~\ref{C:MP0}~(ii). %, and  Lemma~\ref{Bestgen}.

Finally, under the assumption $f\in H^s(\R^N)$, it follows from~\eqref{sionsD}, \eqref{regtruero}, \eqref{sionscA}, and~\eqref{cder} that indeed~${\bD(f) \in\kL(H^s(\R^N) )}$.
\end{proof}

 The proof of Theorem~\ref{T:51}  uses continuity results on pointwise multiplication in the scale~${\{H^r(\R^N)\}_{ r\geq0}}$ given in the following lemma. For simplicity, with $r,r_1,r_2\geq 0$ we will write 
\[H^{r_1}(\R^N)\cdot H^{r_2}(\R^N)\hookrightarrow H^r(\R^N)\]
iff the pointwise multiplication
\[[(a,b)\mapsto ab]: \,H^{r_1}(\R^N)\times H^{r_2}(\R^N)\to
H^r(\R^N)\]
is a continuous bilinear map.

\begin{lemma}\label{L:mult}
Let $r_1,r_2,r\geq 0$ such that $r_1,r_2\geq r$, $r_1+r_2>N/2+r$. Then
\[H^{r_1}(\R^N)\cdot H^{r_2}(\R^N)\hookrightarrow H^r(\R^N).\]
\end{lemma} 

\begin{proof}
As $H^{r_1+r_2-r}(\R^N)\hookrightarrow L_\infty(\R^N)$ and as $H^{r_1+r_2-r}(\R^N)$ is a Banach algebra with respect to pointwise multiplication, we have 
\[H^{r_1+r_2-r}(\R^N)\cdot H^0(\R^N)\hookrightarrow H^0(\R^N),
\quad
H^{r_1+r_2-r}(\R^N)\cdot H^{r_1+r_2-r}(\R^N)\hookrightarrow H^{r_1+r_2-r}(\R^N).\]
 From $r\in[0,r_1+r_2-r]$, we get by interpolation and by symmetry of the pointwise product
\[H^{r_1+r_2-r}(\R^N)\cdot H^r(\R^N)\hookrightarrow H^r(\R^N),
\quad
H^r(\R^N)\cdot H^{r_1+r_2-r}(\R^N)\hookrightarrow H^r(\R^N).\]
Since $r_1,r_2\in[r,r_1+r_2-r]$, the statement follows from  \eqref{IP} and the multilinear interpolation 
result \mbox{\cite[Theorem 4.4.1]{BL76}}.
\end{proof}

As a further preparation for the proof of Theorem \ref{T:51} we prove the following commutator estimates:

\begin{lemma}\label{ADcomm}
    Let  $ \alpha\in[0,1)$,  $ \alpha'\in(\alpha-\min\{1,s-s_c\},\alpha)$, $k\in\N$ with  
    $ 2\leq k\leq s-\alpha$, and~$M\geq 0$.
   Then there is a constant $C>0$ such that for all~$ 1\leq i\leq N$, $f\in H^s(\R^N)$ with~${\|f\|_{H^s}\leq M}$, and $\beta\in H^{k+\alpha-1}(\R^N)$,  $ \underline\beta=(\beta_1,\ldots,\beta_N)\in H^{k+\alpha-1}(\R^N)^N$, we have
    \begin{align*}
        \|\llbracket\p_i,\bD(f)\rrbracket [\beta]\|_{H^{k+\alpha-2}}&\leq C\|\beta\|_{H^{k+\alpha'-1}},\\
        \|\llbracket\p_i,\cA_i(f)\rrbracket[\underline\beta]\|_{H^{k+\alpha-2}}&\leq C|\underline\beta\|_{H^{k+\alpha'-1}}.
    \end{align*}
\end{lemma}
\begin{proof}
    The representations \eqref{sionsD} and \eqref{sionscA} together with Lemma \ref{L:MP1b} yield  for~${1\leq i\leq N}$
    \begin{align*}
        \llbracket\p_i,\bD(f)\rrbracket [\beta]&=B^{\bar\phi}_{1,0}(f)[\p_if,\beta]+2B^{\bar\phi'}_{3,0}(f)[\p_if,f^{[2]},\beta]\\
&\quad\,-\sum_{j=1}^N\big(\sfB^{\bar\phi}_{0,e_j}(f)[\p_{ij}f\beta]+2B^{\bar\phi'}_{2,e_j}(f)[\p_if,f,\p_jf\beta]\big),\\
        \llbracket\p_i,\cA_i(f)\rrbracket [\underline\beta]&=B^{\bar\phi}_{1,0}(f)[\p_if,\beta_i]+2B^{\bar\phi'}_{3,0}(f)[\p_if,f^{[2]},\beta_i]\\
        &\quad\,+\sum_{j=1}^N\bigg(\sfB_{0,e_j}^{\bar\phi}(f)[\p_{ii}f\beta_j-\p_{ij}f\beta_i]
        -\p_{ii}f \sfB_{0,e_j}^{\bar\phi}(f)[\beta_j]\\
        &\hspace{1.55cm}+2B_{2,e_j}^{\bar\phi'}(f)[\p_if,f,\p_if\beta_j-\p_jf\beta_i]
        -2\p_i fB_{2,e_j}^{\bar\phi'}(f)[\p_if,f,\beta_j]
        \bigg).
    \end{align*}
    We estimate all terms on the right separately, using Lemma \ref{L:mult} and  Lemma~\ref{Bestgen} in appropriate order and with appropriate choice of the regularity parameters 
    $r,r_1,r_2, \sigma,\sigma_0,\ldots,\sigma_n$.

    1.  For $1\leq i,j\leq N$  we obtain
\begin{align*}
    \|\p_{ii}f \sfB_{0,e_j}^{\bar\phi}(f)[\beta_j]\|_{H^{k+\alpha-2}}&\leq C\|\p_{ii}f\|_{H^{s-2}}
    \|B_{0,e_j}^{\bar\phi}(f)[\beta_j]\|_{H^{k+\alpha'-1}}
    \leq C \|\beta_j\|_{H^{k+\alpha'-1}},\\
    \|\sfB_{0,e_j}^{\bar\phi}(f)[\p_{ii}f\beta_j]\|_{H^{k+\alpha-2}}&\leq C
    \|\p_{ii}f\beta_j\|_{H^{k+\alpha-2}}\leq C\|\p_{ii}f\|_{H^{s-2}} 
    \|\beta_j\|_{H^{k+\alpha'-1}}\\
    &\leq C\|\beta_j\|_{H^{k+\alpha'-1}},
\end{align*}
and proceed analogously for the remaining terms involving $\sfB_{0,e_i}^{\bar\phi}$, $1\leq i\leq N$.

2.  Set $s':=s+\alpha'-\alpha$ and note that  $s'\in ( \max\{s_c,s-1\},s)$. We apply Lemma \ref{Bestgen} with~$s$ replaced by $s'$, $\sigma:=s'+1-k-\alpha$, $\sigma_0:=s-k-\alpha$, and
$ \sigma_1=1-(s-s')$ to obtain
\[\|B^{\bar\phi}_{1,0}(f)[\p_if,\beta]\|_{H^{k+\alpha-2}}\leq C 
\|\p_if\|_{H^{s-1}}\| \beta\|_{H^{k+\alpha'-1}}\leq C\|\beta\|_{H^{k+\alpha'-1}},\qquad 1\leq i\leq N.
\]
Similarly, using also Lemma \ref{L:mult}, for $1\leq i,j\leq N$ we have
\begin{align*}
    \|B_{2,e_j}^{\bar\phi'}(f)[\p_if,f,\p_if\beta_j\|_{H^{k+\alpha-2}}
    &\leq C\|\p_if\|_{H^{s-1}}\|\p_if\beta_j\|_{H^{k+\alpha'-1}}\leq C \|\beta_j\|_{H^{k+\alpha'-1}},\\
    \|\p_i fB_{2,e_j}^{\bar\phi'}(f)[\p_if,f,\beta_j]\|_{H^{k+\alpha-2}}
    &\leq C\|\p_if\|_{H^{s-1}}\|B_{2,e_j}^{\bar\phi'}(f)[\p_if,f,\beta_j]\|_{H^{k+\alpha-2}}\\
    &\leq C\|\p_if\|_{H^{s-1}}^2\|\beta_j\|_{H^{k+\alpha'-1}}\leq C
    \|\beta_j\|_{H^{k+\alpha'-1}},\\
    \|B^{\bar\phi'}_{3,0}(f)[\p_if, f^{[2]},\beta]\|_{H^{k+\alpha-2}}
    &\leq C\|\p_if\|_{H^{s-1}}\|\beta_j\|_{H^{k+\alpha'-1}}\leq C\|\beta_j\|_{H^{k+\alpha'-1}}.
\end{align*}
The remaining terms are estimated analogously, which proves the lemma.
\end{proof}

 We conclude this section with the proof of Theorem~\ref{T:51}.

\begin{proof}[Proof of  Theorem~\ref{T:51}]  
Fix  $M>0$ and $\alpha\in[0,1)$.
We show the following more general statement: 

For all integers  $0\leq k\leq s-\alpha$ the following holds:
\[\left.\text{
\begin{minipage}{0.85\textwidth}  
 There is a constant  $C>1$ such that for all $a\in[-2,2]$,   $f\in H^s(\R^N)$ with~$\|f\|_{H^s}\leq M$, and  $\beta\in H^{k+\alpha}(\R^N)$ we have
\[\|\beta\|_{H^{k+\alpha}}\leq C\|(1-a\bD(f))[\beta]\|_{H^{k+\alpha}}.\]
Moreover, the map $1-a\bD(f)$ is an isomorphism on $H^{k+\alpha}(\R^N)$.
\end{minipage}
}\right\}\eqno{\textnormal{(H)$_{k,\alpha}$}}\]
We show first that for any  $2\leq k\leq s-\alpha$ the implication 
\be\label{imphka}
{\rm(H)}_{k-2,\alpha}\quad\Rightarrow\quad{\rm(H)}_{k,\alpha}
\ee
holds.
Indeed, assuming ${\rm(H)}_{k-2,\alpha}$,    we have
\begin{align*}
    \|\beta\|_{H^{k+\alpha}}&\leq C\bigg(\|\beta\|_{H^{k+\alpha-2}}+\sum_{i=1}^N\|\p_i^2\beta\|_{H^{k+\alpha-2}}\bigg)\\
    &\leq C\bigg(\|(1-a\bD(f))[\beta]\|_{H^{k+\alpha-2}}+\sum_{i=1}^N\|(1-a\bD(f))[\p_i^2\beta]\|_{H^{k+\alpha-2}}\bigg)\\
    &\leq C\bigg(\|(1-a\bD(f)) [\beta]\|_{H^{k+\alpha}}
    +\sum_{i=1}^N\|\llbracket\p_i^2,\bD(f)\rrbracket [\beta]\|_{H^{k+\alpha-2}}
    \bigg).
\end{align*}
 Fix some arbitrary  $\alpha'\in(\alpha-\min\{1,s-s_c\},\alpha)$. Using Lemma~\ref{L:51} and Lemma~\ref{ADcomm}, we obtain for  $1\leq i\leq N$   
\[\|\llbracket\p_i^2,\bD(f)\rrbracket [\beta]\|_{H^{k+\alpha-2}}
\leq \|\llbracket\p_i,\cA_i(f)\rrbracket[\nabla\beta]\|_{H^{k+\alpha-2}}
+\|\llbracket\p_i,\bD(f)\rrbracket
[\p_i\beta]\|_{H^{k+\alpha-2}}\leq C\|\beta\|_{H^{k+\alpha'}}.\]
Hence,
\[\|\beta\|_{H^{k+\alpha}}\leq C\big(\|(1-a\bD(f))\beta\|_{H^{k+\alpha}}
+\|\beta\|_{H^{k+\alpha'}}\big),\]
and the estimate in (H)$_{k,\alpha}$ follows by interpolation;  cf.~\eqref{IP}, and Theorem \ref{T:1}. The isomorphism property follows in the same way as in the proof of Theorem \ref{T:1}.

As (H)$_{0,0}$ holds by Theorem \ref{T:1}, we conclude (H)$_{2,0}$, 
and by interpolation (H)$_{0,\alpha}$  and~(H)$_{1,\alpha}$. 
Now the general result is obtained by using the implication \eqref{imphka} repeatedly.
\end{proof}

%%%%%%%%%%%%%%%%%%%%%%%%%%%%%%%%%%%%%%%%%%%%%%%%%%
%%%%%%%%%%%%%%%%%%%%%%%%%%%%%%%%%%%%%%%%%%%%%%%%%%
%%%%%%%%%%%%%%%%%%%%%%%%%%%%%%%%%%%%%%%%%%%%%%%%%%
%%%%%%%%%%%%%%%%%%%%%%%%%%%%%%%%%%%%%%%%%%%%%%%%%%
\section{The nonlinear and nonlocal evolution equation for $f$}\label{Sec:5}
%%%%%%%%%%%%%%%%%%%%%%%%%%%%%%%%%%%%%%%%%%%%%%%%%%
%%%%%%%%%%%%%%%%%%%%%%%%%%%%%%%%%%%%%%%%%%%%%%%%%%
%%%%%%%%%%%%%%%%%%%%%%%%%%%%%%%%%%%%%%%%%%%%%%%%%%
%%%%%%%%%%%%%%%%%%%%%%%%%%%%%%%%%%%%%%%%%%%%%%%%%%
In the following  $s$ is fixed according to \eqref{eq:s}.
Based on the results from Sections~\ref{Sec:2}–\ref{Sec:4}, we first reformulate the Muskat problem \eqref{Muskat} as an evolution problem for $f$ only; see~\eqref{NNEP}.
We then show that the Rayleigh-Taylor condition is equivalent to the positivity of a function involving the right  side of \eqref{NNEP},
 and that this condition identifies an open subset~${\cO \subset H^s(\R^N)}$,  cf. \eqref{eq:defO}.
Moreover, we prove that the Muskat problem is of parabolic type  within $\cO$;
see Theorem~\ref{T:GP}, and  conclude the section with the proof of the main result as stated in Theorem~\ref{MT}.

\subsection*{Reformulation of the Muskat problem}
In view of \eqref{forvelo'}, another singular integral operator will arise in  our reformulation~\eqref{NNEP} of the Muskat problem, as detailed  below; see~\eqref{NEP}.
Given $f\in H^s(\R^N)$, $ b\in  L_2(\R^N)^N$, and $x\in\R^N$, we set, using~\eqref{rzx},
 \begin{equation*}
  \begin{aligned}
  \bA(f) [b](x)&:=  \frac{1}{|\s^N|}\PV\int_{\R^N}  
\frac{ \big[(x-\xi)\cdot\nabla f(\xi)-(f(x)-f(\xi))\big]\nabla f(x)\cdot b(\xi) }{
 |z_x-z_\xi|^{N+1}}\,{\rm d}\xi\\
&\,\,\quad-\frac{1}{|\s^N|}\PV\int_{\R^N}  
\frac{ (x-\xi)\cdot b(\xi)\big(1+\nabla f(x)\cdot\nabla f(\xi)\big)}{
 |z_x-z_\xi|^{N+1}}\,{\rm d}\xi.
  \end{aligned}
\end{equation*}
In the notation introduced in \eqref{trueRO} and with $\bar\phi$ from \eqref{deefphi}  and~$b=(b_1,\ldots,b_N)$, we have
\begin{equation}\label{forA}
  \bA(f) [b]
  =  \sum_{i,\, k =1}^N\p_k f \sfB^{\bar\phi}_{0,e_i}(f)[b_k\p_i  f-b_i\p_k f] 
  - \sum_{ i =1}^N\big(\sfB^{\bar\phi}_{0,e_i}(f)[b_i]+\p_if \sfB^{\bar\phi}_{1,0}(f)[b_i]\big).
\end{equation}

Let $(f,v^\pm,p^\pm)$ denote a solution to \eqref{Muskat}  on some interval $[0,T)$, such that for all $t\in[0,T) $  we have $f(t)\in H^s(\R^N)$, 
\[
v^\pm(t)\in {\rm C}(\overline{\0^\pm(t)})\cap {\rm C}^1(\0^\pm(t)),\qquad  p^\pm(t)\in {\rm C}^1(\overline{\0^\pm(t)})\cap {\rm C}^2(\0^\pm(t)).
\]

It then follows from the equivalence of the boundary value problems \eqref{MuskatBVP} and \eqref{MuskatBVP'}, 
Proposition \ref{P:1}, and the kinematic boundary condition \eqref{MuskatEE}$_6$ that $f$ solves the evolution problem
\begin{equation}\label{NEP}
\frac{{\rm d}f}{{\rm d}t}(t)=\Lambda\bA(f(t))[\nabla \beta(t)], \quad t\geq 0,\qquad f(0)=f_0,
\end{equation}
where $\beta(t)\in H^s(\R^N)$ denotes the unique solution to the equation
  \begin{equation}\label{solved}
 \beta(t)+2a_\mu\bD(f(t))[\beta(t)]= f(t)
  \end{equation}
and $\Lambda>0$ is given in \eqref{lambda}.

\begin{lemma}\label{L:61} \phantom{a}
\begin{itemize}
\item[(i)] Given $f\in H^s(\R^N)$, let $\beta(f):=(1+2a_\mu\bD(f))^{-1}[f]\in H^{s}(\R^N)$.
Then
\begin{equation}\label{spbeta}
[f\mapsto \beta(f)]\in {\rm C}^{\infty}(H^s(\R^N), H^s(\R^N)).
\end{equation}
\item[(ii)]  Given $f\in H^s(\R^N)$, let $\Phi(f):=\Lambda\bA(f)[\nabla \beta(f)]\in H^{s-1}(\R^N)$.
Then
\begin{equation}\label{spphi}
\Phi\in {\rm C}^{\infty}(H^s(\R^N), H^{s-1}(\R^N)).
\end{equation}
\end{itemize}
\end{lemma}
\begin{proof}  In view  of the representation formulas \eqref{sionsD} and \eqref{forA}  and of Lemma~\ref{L:51}, we deduce  from \eqref{regtruero} that
\begin{align}
&\bD\in {\rm C}^{\infty}(H^s(\R^N),\kL( H^s(\R^N)))\,\label{regD}\\
 &\bA\in {\rm C}^{\infty}(H^s(\R^N),\kL( H^{s-1}(\R^N)^N, H^{s-1}(\R^N))).\label{regA}
\end{align}
Since the operator that maps an isomorphism onto its inverse is smooth,  it now follows from Theorem~\ref{T:51} and \eqref{regD} that the smoothness property \eqref{spbeta} holds true.
The assertion (ii) follows by combining  \eqref{spbeta} and  \eqref{regA}.
\end{proof}

In view of Lemma~\ref{L:61} we may thus formulate \eqref{NEP} as the following evolution problem 
\begin{equation}\label{NNEP}
\frac{{\rm d}f}{{\rm d} t}(t)=\Phi(f(t)), \quad t\geq 0,\qquad f(0)=f_0,
\end{equation}
where the nonlinear and nonlocal operator $\Phi: H^s(\R^N)\to H^{s-1}(\R^N)$ is smooth; see \eqref{spphi}. 
We are thus in a setting where we may apply the abstract parabolic theory from \cite[Section~8]{L95}, provided that we identify an open set $\cO\subset H^s(\R^N)$  such that for each $f\in\cO$, the Fr\'echet derivative $\p\Phi(f)$ 
generates  an analytic semigroup of operators on $H^{s-1}(\R^N)$. 
This is the main task of the subsequent analysis.
We compute for $f\in H^s(\R^N)$ that
\begin{equation}\label{derPhi}
\p\Phi(f)[h]=\Lambda\p\bA(f)[h][\nabla \beta(f)]+\Lambda\bA(f)[\nabla\big(\p \beta(f)[h]\big)],\qquad h\in H^s(\R^N),
\end{equation}
where $\p \beta(f)[h]\in H^s(\R^N)$ satisfies 
\begin{equation}\label{derBeta}
(1+2a_\mu\bD(f))[\p \beta(f)[h]]=h-2a_\mu\p\bD(f)[h][\beta(f)],\qquad h\in H^s(\R^N).
\end{equation}

Moreover, in view of \eqref{cder},  we have 
\begin{equation}\label{cder'}
 \nabla (\p\bD(f)[h][\beta])=\p\cA(f)[h][\nabla \beta],\qquad f, h, \beta\in H^s(\R^N),
 \end{equation}
 hence
 \begin{equation}\label{cder''}
 (1+2a_\mu\cA(f))[\nabla\p \beta(f)[h]]=\nabla h-2a_\mu\p\cA(f)[h][\nabla \beta(f)],\qquad f, h\in H^s(\R^N),
 \end{equation}
and, using  the formula \eqref{derB} together with   \eqref{sionscA} and~\eqref{forA},  we compute  
\begin{equation}\label{derforA}
\begin{aligned}
 \p\bA(f)[h] [\nabla \beta]&=  \sum_{i,\, k =1}^N\Big[\p_k h\sfB^{\bar\phi}_{0,e_i}(f)[\p_k\beta\p_i  f-\p_i\beta\p_k f]+\p_k f\sfB^{\bar\phi}_{0,e_i}(f)[\p_k\beta\p_i  h-\p_i\beta\p_k h]\\
 &\hspace{1.5cm}+2\p_kf B^{\bar\phi'}_{2,e_i}(f) [f,h,\p_k\beta\p_i  f-\p_i\beta\p_k f]\Big] \\
 &\quad - \sum_{ i =1}^N\Big[2B^{\bar\phi'}_{2,e_i}(f)[f,h,\p_i\beta]+\p_ih \sfB^{\bar\phi}_{1,0}(f)[\p_i\beta]\\
  &\hspace{1.5cm}+\p_if B^{\bar\phi}_{1,0}(f)[h,\p_i\beta]+2\p_if B^{\bar\phi'}_{3,0}(f)[f^{[2]},h,\p_i\beta]\Big]
\end{aligned}
\end{equation}
and 
  \begin{equation}\label{dersions} 
\begin{aligned}
\p\cA_k(f)[h][\nabla  \beta]&=B_{1,0}^{\bar\phi}(f)[h,\p_k\beta]+2B_{3,0}^{\bar\phi'}(f)[f^{[2]},h,\p_k\beta]\\
&\quad +\sum_{i=1}^N\Big[ \sfB^{\bar\phi}_{0,e_i}(f)[\p_kh\p_i\beta-\p_i h\p_k\beta]+2B^{\bar\phi'}_{2,e_i}(f)[f,h,\p_kf\p_i\beta-\p_i f\p_k\beta]\Big]\\
&\quad - \sum_{i=1}^N \Big[ \p_kh\sfB^{\bar\phi}_{0,e_i}(f)[ \p_i\beta]+2\p_kfB^{\bar\phi'}_{2,e_i}(f)[f,h, \p_i\beta]\Big]
\end{aligned}
\end{equation}
 for $1\leq k\leq N$ and~${f, h,\beta\in H^s(\R^N)}$.

\subsection*{The Rayleigh-Taylor condition}
In view of  \eqref{MuskatEE}$_1$, Proposition~\ref{P:1}, and    with the notation  \eqref{lambda} and \eqref{gamma},  the Rayleigh-Taylor condition~\eqref{RT}  is equivalent to a fully nonlinear and nonlocal condition  on $f$ which reads
\begin{equation}\label{eq:rt2}
 \Lambda(1-2a_\mu\bA(f)[\nabla\beta(f)])=\Lambda\big(1-2a_\mu\widetilde\Phi(f)\big)>0\qquad\text{in $\R$,}
\end{equation}
where $\beta(f)$ is defined in Lemma~\ref{L:61}~(i),  and
\be\label{deftilPhi}
\widetilde\Phi(f):=\Lambda^{-1}\Phi(f)=\bA(f)[\nabla\beta(f)].
\ee
Since $\widetilde\Phi$ is smooth by Lemma~\ref{L:61}~(ii), $\widetilde\Phi(0)=0$,  $\Lambda>0$, and~${H^{s-1}(\R^N)\hookrightarrow {\rm BUC}(\R^N)}$, it follows that
\begin{equation}\label{dop}
\cO:=\big\{f\in H^s(\R^N): 
\text{$\Lambda\big(1-2a_\mu\widetilde\Phi(f)\big)>0$ in $\R$}\big\}
\end{equation}
 is a nonempty open subset of $H^s(\R^N)$.

 \subsection*{The Fourier multipliers $D^{\phi,A}_{n,\nu}$} Below, we demonstrate that the evolution problem \eqref{NNEP} is of parabolic type in $\cO$ by analyzing the Fréchet derivative $\p\Phi(f)$ for $f \in \cO$ and showing that it generates a strongly continuous analytic semigroup  of operators on $H^{s-1}(\R^N)$.

A critical step in this proof involves the localization of $\p\Phi(f)$. This technique parallels the method of freezing the coefficients of  differential operators.  Localizing  $\p\Phi(f)$ will be reduced to the localization of operators from the class $B^{\phi}_{n,\nu}(f)$, defined in \eqref{trueRO}, by operators 
$D^{\phi,A}_{n,\nu}$, with $A\in\R^N$, 
which form a subclass of the $B^{\phi}_{n,\nu}(f)$. 
 They are introduced below and discussed in Appendix~\ref{Sec:D}.

\medskip
 
Given $\phi\in {\rm C}^\infty([0,\infty))$, ${\nu \in\N^N,\,n\in\N}$ with~$n+|\nu|$ odd, and $A\in\R^N$ 
we  define the singular integral operator $D^{\phi,A}_{n,\nu}$ by setting
 \begin{equation}\label{dfnb}
D^{\phi,A}_{n,\nu}[\beta](x):=\frac{1}{|\s^N|} \PV\int_{\R^N}\phi\left(\Big(\frac{A\cdot \xi}{|\xi|}\Big)^{ 2} \right)  \Big(\frac{A\cdot \xi}{|\xi|}\Big)^n\frac{\xi^\nu}{|\xi|^{|\nu|}}
\frac{\beta(x-\xi)}{|\xi|^N}\,{\rm d}\xi
  \end{equation}
  for $\beta\in L_2(\R^N)$ and $x\in\R^N$.
Defining the Lipschitz function  $\bar f_A:\R^N\to\R $ 
by~$ \bar f_A(x):=A\cdot x$,  we have
   \begin{equation}\label{identif}
D^{\phi,A}_{n,\nu}= \sfB^{\phi}_{n,\nu}( \bar f_A), 
  \end{equation}
 and Lemma~\ref{L:MP0} ensures that $D^{\phi,A}_{n,\nu}\in\kL(L_2(\R^N))$. 
 In  Appendix~\ref{Sec:D}  we prove that~$D^{\phi,A}_{n,\nu}$ are Fourier multipliers and that  near any $x_0\in\R^N$, the operator $\sfB^{\phi}_{n,\nu}(f)$ can be localized
  in a suitable sense by the operator $D^{\phi,\nabla\phi(x_0)}_{n,\nu}$.
  
 Observe, moreover, that the operators $D^{\phi,A}_{n,\nu}$ satisfy the identity
  \be\label{Dn-1rep}
  D^{\phi,A}_{n,\nu}=\sum_{k=1}^N  A_k D^{\phi,A}_{n-1,\nu+e_k},\quad n\geq 1.
  \ee

\subsection*{Parabolicity under the Rayleigh-Taylor condition}
The main goal of this section is to demonstrate that \eqref{NNEP} is of parabolic type within $\cO$, as stated in the following result.

\begin{thm}\label{T:GP}
Given $f\in\cO$, the Fr\'echet derivative $\p\Phi(f)$ generates a strongly continuous  analytic semigroup of operators  on $H^{s-1}(\R^N)$.
\end{thm}

The proof of Theorem~\ref{T:GP} will be deferred to the end of this section, as it necessitates some preliminary work. 
The key step in the proof is outlined in Proposition~\ref{P:Local}, where we  in particular localize~$\p\Phi(f)$. 
To achieve this, we employ appropriate partitions of unity, which we next introduce.

For each~${\varepsilon\in(0,1)}$ we  fix a finite {\em $\varepsilon$-localization family}, that is,  a family  $$\{(\pi_j^\varepsilon, x_j^\e)\,:\, 0\leq j\leq  m(\e)\}\subset  {\rm C}^\infty(\R^N,[0,1])\times\R^N,$$
with $ m(\e)\in\mathbb{N} $ sufficiently large, such that 
\begin{align*}
\bullet\,\,\,\, \,\,  & \mbox{$\supp \pi_j^\varepsilon =\overline{\bB}_\e(x_j^\e)$ for $1\leq j\leq m(\e)$, $\supp \pi_0^\varepsilon  =  \R^N\setminus{\bB}_{\e^{-1}}(0)$, and $x_0^\e:=0$,} \\
\bullet\,\,\,\, \,\, &\mbox{ $\displaystyle\sum_{j=0}^{m(\e)} \pi_j^\varepsilon=1$ in $\R^N$.} 
\end{align*} 
Here and below $\bB _r(x)$ denotes the ball centered at $x\in\R^N$ with radius $r>0$, and $\overline{\bB}_r(x)$ is its closure.
 With such an $\varepsilon$-localization family we associate  a second family~${\{\chi_j^\varepsilon\,:\, 0\leq j\leq m(\e)\}}$  
satisfying
\begin{equation}\label{chije}
\begin{aligned}
\bullet\,\,\,\, \,\,  &\mbox{$\chi_j^\varepsilon\in {\rm C}^\infty(\R^N,[0,1])$ and  $\chi_j^\varepsilon=1$ on $\supp \pi_j^\varepsilon$,} \\[1ex]
\bullet\,\,\,\, \,\,  &\mbox{$\supp \chi_j^\varepsilon =\overline{\bB}_{2\e}(x_j^\e)$ for $1\leq j\leq  m(\e)$ and $\supp \chi_0^\varepsilon  = \R^N\setminus{\bB}_{\e^{-1}-\e}(0)$.} 
\end{aligned} 
\end{equation}
It readily follows  {from the above properties} that, for each $r\ge 0$, the map 
\begin{align}\label{eqnorm}
\bigg[f\mapsto \sum_{j=0}^{ m(\e)}\|\pi_j^\varepsilon f\|_{H^r}\bigg]: H^r(\R^N)\to[0,\infty)
\end{align}
defines a norm on $H^r(\R^N)$ which is equivalent to the standard norm.

In order  to establish Theorem~\ref{T:GP} we fix in the following $f\in\cO$ and set
\[
 \beta:=\beta(f)\in H^s(\R^N);
\]
see Lemma~\ref{L:61}~(i).
Moreover, we define the  path $\Psi\in{\rm C}([0,1], \kL(H^s(\R^N), H^{s-1}(\R^N)))$ by 
\begin{equation}\label{Psi}
\Psi(\tau)=\tau\p\Phi(f)-(1-\tau)\Lambda\big(1-2a_\mu\widetilde\Phi(f)\big)\sum_{k=1}^N \sfB^{\bar\phi}_{0,e_k}(f) \frac{\p}{\p x_k},\qquad \tau\in[0,1],
\end{equation}
 which connects  the Fr\'echet 
derivative $\p\Phi(f)=\Psi(1)$ to the operator~$\Psi(0)$  which has a considerably simpler structure.
 In the proof of Theorem~\ref{T:GP}
it is crucial to establish the invertibility of $\omega-\p\Phi(f)\in\kL(H^s(\R^N), H^{s-1}(\R^N))$ for sufficiently large $\omega>0$. This will rely on the continuity method together with the invertibility of~$\omega-\Psi(0)$ which is provided in Proposition~\ref{P:inv}. For this, the Rayleigh-Taylor condition \eqref{eq:rt2} will be essential, as the positive function 
$\Lambda\big(1-2a_\mu\widetilde\Phi(f)\big)$
appears as pointwise multiplier in  the definition of~$\Psi(0)$.
 Further advantages of our choice for the homotopy~$\Psi$  will become apparent  when we carry out the localization.
We prepare for this by establishing the following identity: 

 \begin{lemma}  Given $f \in H^s(\mathbb{R}^N)$, let $\beta = \beta(f)$ be as defined in Lemma~\ref{L:61}~(i).  Then
 \be\label{eq:wow}
 1-2a_\mu\widetilde\Phi(f)=-\nabla f\cdot\nabla \beta+(1+|\nabla f|^2)\bigg(1+2a_\mu \sum_{k=1}^N \sfB^{\bar\phi}_{0,e_k}(f)[ \p_k\beta]\bigg).
 \ee
 \end{lemma}
 \begin{proof}
     We recall from \eqref{forA} and \eqref{deftilPhi}   that 
     \be\label{mmm1}
     \widetilde\Phi(f)=\sum_{i,\, k =1}^N\p_k f \sfB^{\bar\phi}_{0,e_i}(f)[\p_k\beta\p_i  f-\p_i\beta\p_k f] - \sum_{ k =1}^N\sfB^{\bar\phi}_{0,e_k} (f)[\p_k \beta]- \sum_{ i =1}^N\p_if \sfB^{\bar\phi}_{1,0}(f)[\p_i\beta].
     \ee
     Taking the gradient on both sides of the equation
$ f=(1+2a_\mu\bD(f))[\beta],$
applying \eqref{cder}, and taking the inner product with $\nabla f$ yields via \eqref{sionscA}
\begin{align*}
    \nabla f\cdot\nabla \beta&=|\nabla f|^2-2a_\mu \sum_{i=1}^N\p_i f \sfB_{1,0}^{\bar\phi}(f)[\p_i\beta]+2a_\mu|\nabla f|^2 \sum_{k=1}^N \sfB^{\bar\phi}_{0,e_k}(f)[ \p_k\beta]\\
&\quad+2a_\mu\sum_{i,\,k=1}^N\p_k f \sfB^{\bar\phi}_{0,e_i}(f)[\p_k\beta\p_i f-\p_i\beta\p_kf].
\end{align*}
Using \eqref{mmm1} and the last equation to rewrite the sum $1-2a_\mu\widetilde\Phi(f)+\nabla f\cdot\nabla\beta$,
yields the identity \eqref{eq:wow}.
 \end{proof}

In view of the structure of  the operator $\Psi(0)$, we define for any $x_0\in\R^N$ the Fourier multiplier 
\begin{equation}\label{ax0}
{\sf A}(x_0):=\big(1-2a_\mu\widetilde\Phi(f)\big)(x_0)\sum_{k=1}^N D^{\bar\phi,\nabla f(x_0)}_{0,e_k}\frac{\p}{\p x_k},
\end{equation}
as well as its counterpart at infinity
\begin{equation}\label{a0}
{\sf A}_0=\sum_{k=1}^N D^{\bar\phi,0}_{0,e_k}\frac{\p}{\p x_k}.
\end{equation}

 On the level of these Fourier multipliers, \eqref{eq:wow} implies the following identity:
\begin{lemma}\label{L:devvor} With $x_0\in\R^N$, $f\in H^s(\R^N)$, and $\beta=\beta(f)$ as defined in Lemma \ref{L:61} (i), we have
\begin{equation}\label{devvor}
\begin{aligned}
    {\sf A}(x_0)=&\sum_{k=1}^N\Bigg\{
\sum_{i=1}^N(\p_k f\p_i\beta)(x_0)  D^{\bar\phi,\nabla f(x_0)}_{0,e_i}+2\sum_{i=1}^N\big[\big(1+|\nabla f|^2\big)\p_i\beta\big](x_0) D^{\bar\phi',\nabla f(x_0)}_{1,e_i+e_k} \\
&\hspace{1cm}+\big(1+|\nabla f|^2\big)(x_0)\bigg(1+2 a_\mu\sum_{i=1}^NB^{\bar\phi}_{0,e_i}(f)[\p_i\beta](x_0)\bigg)
 D_{0,e_k}^{\bar\phi,\nabla f(x_0)} \Bigg\}\frac{\p}{\p x_k}.
\end{aligned}
\end{equation}
\end{lemma}
\begin{proof}
    Let $K:\R^N\setminus\{0\}\to\R^N$ be smooth, odd, and homogeneous of degree $-N$. 
    Then ${\rm div\,} K$ is even and homogeneous of degree $-N-1$.
    Consequently, for any $x\in\R^N$, $h\in {\rm C}_0^\infty(\R^N)$,  and~$r>0$   integration by parts  yields
\begin{align*}
    &\int_{\{|\xi|>r\}}K(\xi)\cdot\nabla h(x-\xi)\,{\rm d}\xi
    =\int_{\{|\xi|>r\}}K(\xi)\cdot\nabla_\xi (\delta_{[x,\xi]}h)\,{\rm d}\xi\\
    &=-\int_{\{|\xi|>r\}}(\delta_{[x,\xi]}h){\rm div\,}K(\xi)\,{\rm d}\xi
    -\int_{\{|\xi|=r\}} (\delta_{[x,\xi]}h)K(\xi)\cdot\frac{\xi}{|\xi|}\,{\rm d} S(\xi),
\end{align*}
 where ${\rm d} S$ denotes the surface measure of the sphere $\{|\xi|=r\}$.
  Recalling that $K$ is odd and  using the asymptotics
$\delta_{[x,\xi]}h=\xi\cdot\nabla h(x)+O(|\xi|^2)$ for $\xi\to0$,
in the limit $r\to 0$  the latter integral identity leads to
\[{\rm PV}\int_{\R^N}K(\xi)\cdot\nabla h(x-\xi)\,{\rm d}\xi
=-{\rm PV}\int_{\R^N}(\delta_{[x,\xi]}h){\rm div\,}K(\xi) \,{\rm d}\xi.\]
Choosing for $K$ now in particular 
\begin{align*}
    K_1(\xi):=&\frac{1}{|\xi|^N}\frac{B\cdot\xi}{|\xi|}\,\psi\left(\frac{A\cdot\xi}{|\xi|}\right)A,\\
    K_2(\xi):=& \frac{ \xi}{|\xi|^{N+1}}\bigg[\psi\left(\frac{A\cdot\xi}{|\xi|}\right)
    \left((N+1)\frac{A\cdot\xi}{|\xi|}\frac{B\cdot\xi}{|\xi|}-A\cdot B\right)\\
   & \hspace{1.25cm}+\psi'\left(\frac{A\cdot\xi}{|\xi|}\right)\frac{B\cdot\xi}{|\xi|}
    \bigg(\left(\frac{A\cdot\xi}{|\xi|}\right)^2-|A|^2\bigg)\bigg]
\end{align*}
with $A,\,B\in\R^N$ fixed, $\psi:\R\to\R$ smooth and even, we observe 
${\rm div\,} K_1={\rm div\,} K_2$ and hence
\be\label{K1K2id}
{\rm PV}\int_{\R^N}K_1(\xi)\cdot\nabla h(x-\xi)\,{\rm d}\xi
={\rm PV}\int_{\R^N}K_2(\xi)\cdot\nabla h(x-\xi)\,{\rm d}\xi.
\ee
Specifying further $\psi(z):=\bar\phi(z^2)$, $z\in\R$, we  may recast $K_2$ as
\[K_2(\xi)=\frac{\xi}{|\xi|^{N+1}}\bigg(-\bar\phi\bigg(\bigg(\frac{A\cdot\xi}{|\xi|}\bigg)^2\bigg)A\cdot B
-2(1+|A|^2)\bar\phi'\bigg(\bigg(\frac{A\cdot\xi}{|\xi|}\bigg)^2\bigg)
\frac{A\cdot\xi}{|\xi|}\frac{B\cdot\xi}{|\xi|}\bigg),\]
and \eqref{K1K2id},  together with Lemma~\ref{L:MP0} and \eqref{identif},  implies
\be\label{idAB}\sum_{i,k=1}^N A_kB_iD^{\bar\phi,A}_{0,e_i}\frac{\p}{\p x_k}
=\sum_{k=1}^N\bigg(-2(1+|A|^2)\sum_{i=1}^NB_iD^{\bar\phi',A}_{1,e_i+e_k}
-A\cdot B\,D^{\bar\phi,A}_{0,e_k}\bigg)\frac{\p}{\p x_k}.
\ee
By \eqref{eq:wow} we have
\begin{equation}\label{prelimA}
\begin{aligned}
  {\sf A}(x_0)&=\bigg(-\nabla f(x_0)\cdot\nabla\beta(x_0)\\
  &\qquad+(1+|\nabla f(x_0)|^2)
\bigg(1+2a_\mu\sum_{i=1}^N\sfB^{\bar\phi}_{0,e_i}(f)[\p_i\beta](x_0)\bigg)
\bigg)\sum_{k=1}^N D^{\bar\phi,\nabla f(x_0)}_{0,e_k}\frac{\p}{\p x_k}. 
\end{aligned}
\end{equation}
Using \eqref{idAB} with $A:=\nabla f(x_0)$, $B:=\nabla\beta(x_0)$ to replace the operator
\[-\nabla f(x_0)\cdot\nabla\beta(x_0)\sum_{k=1}^N D^{\bar\phi,\nabla f(x_0)}_{0,e_k}\frac{\p}{\p x_k}\]
in \eqref{prelimA}, we obtain \eqref{devvor}.
\end{proof}

 Let $s'\in(\max\{ s_c,s-1\},s)$ be fixed in the following.
 In order to formulate our localization results, we will use the following notation: 
 With the $\e$-localization family chosen above for given operators 
\[ T,\, T_0,\, T_j^\e\in \kL(H^s(\R^N),H^{s-1}(\R^N)),\qquad \e\in(0,1), \,1\leq j\leq  m(\e),\]
we will write 
\be\label{eq:locsim}
T\locsim (T_0,T_j^\e)
\ee
for the following statement:\medskip

{\em For any $\theta>0$, there exists an $\eps_0\in(0,1]$ such that for all $\e\in(0,\e_0)$ there exists a positive constant $K=K(\theta,\e,s')$ such that for all $ 0\leq j\leq m(\e)$ and $ h\in H^s(\R^N)$
\[\|\pi^\e_jT[h]-T^\e_j[\pi^\e_j h]\|_{H^{s-1}}\leq\theta \|\pi^\e_j h\|_{H^s}+K\|h\|_{H^{s'}},\]
where $T_0^\e:=T_0$. }

 Thus, \eqref{eq:locsim} encodes the estimates ensuring that $T_0$ and $T_j^\eps$ are ``suitable localizations'' of~$T$ at infinity and near $x_j^\e$, respectively.

The relation $\locsim$  is obviously ``linear'' in the sense that for $\lambda\in\R$, we have 
\[T+\lambda S\locsim (T_0+\lambda S_0,T_j^\e+\lambda S_j^\e)\]
 provided that $T\locsim (T_0,T_j^\e)$ and $S\locsim (S_0,S_j^\e)$.

In view of the structure of the operators we are going to localize, we note the following observation on compositions:
\begin{lemma}\label{le:loccomp}
Let 
\[T\in\kL(H^s(\R^N),H^{s-1}(\R^N)),\quad
S\in\kL(H^s(\R^N))\cap\kL(H^{s'}(\R^N)),\]
and assume there is a constant  $C_1>0$
such  that for  each $\e\in(0,1)$
there is  $K=K(\e)>0$ such that for all $h\in H^s(\R^N)$  and $0\leq j\leq  m(\e)$
\[\|\pi_j^\e Sh\|_{H^s}\leq  C_1\|\pi_j^\e h\|_{H^s}+K\|h\|_{H^{s'}}.\]

Then:
\begin{itemize}
    \item[\rm (i)] If $T\locsim(0,0)$, then $ T S\locsim (0,0)$.
    \item[\rm (ii)] If  $\p_kS\locsim\big(S_{0,k},S_{j,k}^\e\big)$,
    $1\leq k\leq N,$ and
\[T\locsim\bigg(\sum_{k=1}^NT_{0,k}\frac{\p}{\p x_k},
\sum_{k=1}^NT_{j,k}^\e\frac{\p}{\p x_k}\bigg)\]
    with $T_{0,k},T_{j,k}^\e\in\kL(H^{s-1}(\R^N))$,
and there exists a constant  $C_2>0$ such that for all~${\e\in(0,1)}$,
  $0\leq j\leq m(\e)$, and $1\leq k\leq N$
\[\|T_{j,k}^\e\|_{\kL(H^{s-1}(\R^N))}\leq C_2,\]
 where $T^\e_{0,k}:=T_{0,k}$, then
\[ T S\locsim\bigg(\sum_{k=1}^N  T_{0,k} S_{0,k},
\sum_{k=1}^N  T_{j,k}^\e S_{j,k}^\e\bigg).\]
\end{itemize}
\end{lemma}
\begin{proof} The proof of (i) is straightforward. To show (ii),  fix 
  $\theta>0$. Then, if $\e\in(0,1)$ is small enough, we estimate for $0\leq j\leq m(\e)$  and $h\in H^s(\R^N)$
\begin{align*}
    &\bigg\|\pi^\e_j TS[h]-\sum_{k=1}^NT^\e_{j,k}S^\e_{j,k}[\pi^\e_j h]\bigg\|_{H^{s-1}}\\
    &\leq \bigg\|\pi^\e_j T[S[h]]
    -\sum_{k=1}^NT^\e_{j,k} \frac{\p}{\p x_k}[\pi^\e_j S[h]]\bigg\|_{H^{s-1}}\\
    &\quad+\sum_{k=1}^N\|T_{j,k}^\e\|_{\kL(H^{s-1}(\R^N))}\Big(  \|
     (\p_k\pi^\e_j)S[h] \|_{H^{s-1}}
    +\big\|\pi^\e_j\p_kS[h]-S^\e_{j,k}[\pi^\e_j h]\big\|_{H^{s-1}}\Big)\\
    &\leq \frac{\theta}{2 C_1}\|\pi^\e_j S[h]\|_{H^s}+NC_1\frac{\theta}{2NC_2}\|\pi^\e_j h\|_{H^s}
    +K\|h\|_{H^{s'}}\leq \theta \|\pi^\e_j h\|_{H^s}+K\|h\|_{H^{s'}}.
\end{align*}
\end{proof}

 As a further preparation,  in the next lemma we gather localizations by Fourier multipliers for the operators that essentially constitute $\Psi(\tau)$:
\begin{lemma}\label{le:locAop}
 Given $f,\, \beta\in H^s(\R^N)$,   it holds that
\begin{align}
\bA(f)\circ\nabla&\locsim
\bigg(-{\sf A}_0,\;-(1+|\nabla f(x_j^\e)|^2)\sum_{k=1}^N
D^{\bar\phi,\nabla f(x^\e_j)}_{0,e_k} \frac{\p}{\p x_k}\bigg),\label{locforA}\\
\p\bA(f)[\,\cdot\,][\nabla\beta]&\locsim(0,T_{j}^\e),\label{locderforA}\\
\cA_k(f)\circ\nabla&\locsim(0,0),\label{locsionscA}\\
\p\cA_k(f)][\,\cdot\,][\nabla\beta]&\locsim
\bigg(0,\sum_{i=1}^N\Big(\p_i\beta(x_j^\e)D^{\bar\phi,\nabla f(x_j^\e)}_{0,e_i}
 \frac{\p}{\p x_k}-\sfB_{0,e_i}^{\bar\phi} (f)[\p_i\beta](x_j^\e) \frac{\p}{\p x_k}
\Big)\bigg),\label{locdersions}
\end{align}
where, for $1\leq j\leq  m(\e)$,
\begin{align*}
    T_{j}^\e&:= \sum_{i,\, k =1}^N \sfB^{\bar\phi}_{0,e_i}(f)[\p_k\beta\p_i  f-\p_i\beta\p_k f](x_j^\e)\frac{\p}{\p x_k}-\sum_{ i =1}^N \sfB^{\bar\phi}_{1,0}(f)[\p_i\beta](x_j^\e) \frac{\p}{\p x_i}\\
   &\,\,\quad-\sum_{i,\, k =1}^N(\p_k f\p_i\beta)(x_j^\e)D^{\bar\phi,\nabla f(x_j^\e)}_{0,e_i}\frac{\p}{\p x_k}
-2\sum_{i,\,k =1}^N\big[\p_i\beta\big(1+|\nabla f|^2\big)\big](x_j^\e) D^{\bar\phi',\nabla f(x_j^\e)}_{1,e_i+e_k} \frac{\p}{\p x_k}.
\end{align*}
\end{lemma}
\begin{proof}
    We use  the identities \eqref{sionscA}, \eqref{forA}, \eqref{derforA}, and \eqref{dersions}
    to represent  the  operators on the right of~\eqref{locderforA}-\eqref{locdersions} by operators of the class $B_{n,\nu}^\phi(f)$.
    The statements follow from the  commutator and localization results given in Lemma~\ref{commphi}, Lemma \ref{L:Mix},  Proposition~\ref{T:LOC},  and Lemma~\ref{T:prod}, together with the identity \eqref{Dn-1rep}.
\end{proof}

Given $\tau\in[0,1]$,  we will localize the operator~$\Psi(\tau)$ near $x_j^\e$, $1\leq j\leq m(\e)$, by the Fourier multiplier   $A_{j,\tau}^\e$ given by 
\begin{equation}\label{Ajt}
\begin{aligned}
 A_{j,\tau}^\e&:=- \Lambda {\sf A}(x^\e_j)
 + \tau\Lambda {\sf B}(x^\e_j),\\
 {\sf B}(x^\e_j)&:=
 \sum_{k=1}^N\bigg\{- \sfB^{\bar\phi}_{1,0}(f)[\p_k\beta](x_j^\e)+\sum_{i =1}^N \sfB^{\bar\phi}_{0,e_i}(f)[\p_k\beta\p_i  f-\p_i\beta\p_k f](x_j^\e) \\[1ex]
&\hspace{1.55cm}+2a_\mu \big(1+|\nabla f|^2\big)(x_j^\e) \sum_{i=1}^N \p_i\beta(x_j^\e)D_{0,e_k}^{\bar\phi,\nabla f(x_j^\e)}D_{0,e_i}^{\bar\phi,\nabla f(x_j^\e)}\bigg\}\frac{\p}{\p x_k},
\end{aligned}
\end{equation}
and at infinity by the Fourier multiplier
\begin{equation}\label{A0}
A_{0,\tau}:=A_0:=-\Lambda{\sf A}_0,
\end{equation}
with ${\sf A}$ and ${\sf A_0}$ defined in \eqref{ax0}-\eqref{a0}.
Our proof of Theorem~\ref{T:GP} is based on the fact that this is indeed a proper localization:

\begin{prop}\label{P:Local}
It holds that
\be\label{eq:loc}
\Psi(\tau)\locsim(A_0,A_{j,\tau}^\e),\quad\text{uniformly in $\tau\in[0,1]$.}
\ee
More precisely: For any $\theta>0$, there exists an $\eps_0\in(0,1]$ such that for  each~$\e\in(0,\e_0)$ there is a constant $K=K(\theta,\e,s')>0$  such that for all $0\leq j\leq m(\e)$, $ h\in H^s(\R^N)$, and~${\tau\in[0,1]}$
\be\label{eq:locest}\|\pi^\e_j\Psi(\tau)[h]-A^\e_{j,\tau}[\pi^\e_j h]\|_{H^{s-1}}\leq\theta \|\pi^\e_j h\|_{H^s}+K\|h\|_{H^{s'}},
\ee
where $A_{0,\tau}^\e:=A_0$.
\end{prop}

We prepare the proof of this proposition by showing that the operator $\p\beta(f)$ satisfies the assumptions on $S$ in Lemma \ref{le:loccomp},  and localizing its spatial derivatives.

\begin{lemma}\label{le:estdbeta}
The following properties hold:
\begin{itemize}
    \item[(i)]   $\p\beta(f)\in\kL(H^s(\R^N))\cap \kL(H^{s'}(\R^N))$.
    \item[(ii)] There is a constant   $C_1>0$ and for each~$\e\in(0,1)$ there is a  constant~${K=K(\e)>0}$ such that for all
 $h\in H^s(\R^N)$  and $0\leq j\leq  m(\e)$  
\begin{equation}\label{estsol}
\|\pi_j^\e \p \beta(f)[h]\|_{H^s}\leq C_1\|\pi_j^\e h\|_{H^s} +K\|h\|_{H^{s'}}.
\end{equation}
\item[(iii)] With $\beta=\beta(f)$ as defined in Lemma~\ref{L:61}~{\rm (i)}, for
$1\leq k\leq N$ we have
\begin{equation*}
\p_k\p\beta(f)
\locsim
\bigg(\frac{\p}{\p x_k},\frac{\p}{\p x_k}-2a_\mu
\sum_{i=1}^N\Big(
\p_i\beta(x_j^\e)D^{\bar\phi,\nabla f(x_j^\e)}_{0,e_i}
\frac{\p}{\p x_k}-\sfB_{0,e_i}^{\bar\phi}(f)[\p_i\beta](x_j^\e)\frac{\p}{\p x_k}\Big)\bigg).
\end{equation*}
\end{itemize}
\end{lemma}
\begin{proof}
The assertion (i) is a straightforward consequence of  Lemma~\ref{L:61}~(i). 
In order to establish~(ii), we multiply \eqref{derBeta} by $\pi_j^\e$ to obtain the operator equation
\[(1+2a_\mu\bD( f))[\pi_j^\e\p \beta(f)[h]]=\pi_j^\e h
    -2a_\mu\llbracket \pi_j^\e,\bD( f)\rrbracket [\p \beta(f)[h]]
    -2a_\mu\pi_j^\e \p\bD(f)[h][\beta].\]
Theorem \ref{T:51} implies 
\be\label{piderbeta1}
\|\pi_j^\e\p \beta(f)[h]]\|_{H^s}
\leq C\big(\|\pi_j^\e h\|_{H^s}
+\|\llbracket \pi_j^\e,\bD( f)\rrbracket [\p \beta(f)[h]]\|_{H^s}
+\|\pi_j^\e \p\bD(f)[h][\beta]\|_{H^s}\big),
\ee
and we estimate the last two terms separately.  

 We  combine  Lemma~\ref{L:51}  and Lemma~\ref{L:61} (both with $s$ replaced by $s'$),   Lemma~\ref{commphi},  and the  relations \eqref{equivgrad} and \eqref{reg:cla} to derive that 
 \begin{equation}\label{piderbeta2}
\begin{aligned}
   \|\llbracket \pi_j^\e,\bD( f)\rrbracket [\p \beta(f)[h]]\|_{H^s}
    &\leq C_0\big(\|\llbracket \pi_j^\e,\bD( f)\rrbracket [\p \beta(f)[h]]\|_2
    +\|\nabla \llbracket \pi_j^\e,\bD( f)\rrbracket [\p \beta(f)[h]]
    \|_{H^{s-1}}\big)\\
    &\leq K\big(\|(\nabla\pi_j^\e)\bD( f)[\p \beta(f)[h]]\|_{H^{s-1}}
    +\|\llbracket \pi_j^\e,\cA( f)\rrbracket [ \nabla \p \beta(f)[h]]\|_{H^{s-1}}\\
    &\hspace{1cm}+\|\p \beta(f)[h]\|_{H^{s'}}
    +\|\cA(f)[(\nabla \pi_j^\e)\p \beta(f)[h]])\|_{H^{s-1}}\big)\\
    &\leq  K\big(\|\p \beta(f)[h]\|_{ H^{s'}}
    +\|\nabla\p \beta(f)[h]\|_{H^{s'-1}}\big)\leq K\|h\|_{H^{s'}}.
\end{aligned}
\end{equation}
Using \eqref{equivgrad},   \eqref{regD} (with  $s$ replaced by $s'$), and \eqref{cder'}, we further have
\begin{align*}
    &\|\pi_j^\e \p\bD(f)[h][\beta]\|_{H^s}\\
&\leq C_0\big(\|\pi_j^\e \p\bD(f)[h][\beta]\|_{2}+\|\nabla\big(\pi_j^\e \p\bD(f)[h][\beta]\big)\|_{H^{s-1}}\big)\\  
&\leq C_0\big(\|\pi_j^\e \p\bD(f)[h][\beta]\|_{2}+\|\pi_j^\e \p\cA(f)[h][\nabla\beta]\|_{H^{s-1}}+\|(\nabla\pi_j^\e) \p\bD(f)[h][\beta]\|_{H^{s-1}}\big)
\\
&\leq K\|h\|_{H^{s'}}+ C_0\|\pi_j^\e \p\cA(f)[h][\nabla\beta]\|_{H^{s-1}}.
\end{align*} 
 To estimate the last term, we expand $\p\cA(f)[h][\nabla\beta]$ according to 
\eqref{dersions}. For all individual terms we use the commutator type estimates of  Lemma~\ref{commphi} and Lemma~\ref{L:Mix} together with Lemma~\ref{L:MP1} to obtain 
\be\label{piderbeta3}
\|\pi_j^\e \p\cA(f)[h][\nabla\beta]\|_{H^{s-1}}\leq C_1\|\pi_j^\e h\|_{H^s} +K\|h\|_{H^{s'}},
\ee
with  $C_1$ depending on $\|f\|_{H^s}$ only. 
Summarizing, we  obtain the desired estimate~\eqref{estsol} from~\eqref{piderbeta1}--\eqref{piderbeta3}.

 It remains to establish the localization property (iii). To this end, we recall from \eqref{cder''} that
\[\p_k\p\beta(f)=\frac{\p}{\p x_k}-2a_\mu\big(\cA_k(f)[\nabla\p\beta(f)[\,\cdot\,]]+\p\cA_k(f)[\,\cdot\,][\nabla\beta]\big).\]
We localize the three terms on the right of this identity separately.
From the product rule of differentiation we immediately get
\[\frac{\p}{\p x_k} \locsim \left(\frac{\p}{\p x_k},\frac{\p}{\p x_k}\right).\]
  For the second term, we combine \eqref{locsionscA}, the assertions (i) and~(ii) established above, and Lemma \ref{le:loccomp} (i) with $T:=\cA_k(f)\circ\nabla$ 
and $S:=\p \beta(f)$,  and arrive at
\[\cA_k(f)[\nabla\p\beta(f)[\,\cdot\,]]\locsim(0,0).\]
 Finally, recalling the localization result~\eqref{locdersions} for the third term
$\p\cA_k(f)[\,\cdot\,][\nabla\beta]$,  we have established the remaining property~(iii) and thus completed the proof.
\end{proof}

 We are now in a position to establish the localization result for $\Psi(\tau)$, $\tau\in[0,1]$ announced above.

\begin{proof}[Proof of Proposition \ref{P:Local}]
In view of $\Psi(\tau)=\tau\Psi(1)+(1-\tau)\Psi(0)$, it is sufficient to show
\begin{align}
    -\Lambda^{-1}\Psi(0)&=(1-2a_\mu\widetilde\Phi(f))\sum_{k=1}^N
    \sfB^{\bar\phi}_{0,e_k}(f)\frac{\p}{\p x_k}
    \locsim\big({\sf A}_0,{\sf A}(x_j^\e)\big)\label{locpsi1-1}\\
    -\Lambda^{-1}\Psi(1)&=-\big(\p\bA(f)[\,\cdot\,][\nabla \beta]
    +\bA(f)[\nabla(\p\beta(f)[\,\cdot\,])]\big)
    \locsim\big({\sf A}_0,{\sf A}(x_j^\e)-{\sf B}(x_j^\e)\big).\label{locpsi1}
\end{align}
 The property~\eqref{locpsi1-1} is immediate from Proposition \ref{T:LOC}.
In order to establish~\eqref{locpsi1}, we localize the terms of $\Lambda^{-1}\Psi(1)$ separately.  In view of~\eqref{locderforA}, it remains to localize 
the linear operator~$\bA(f)[\nabla(\p\beta(f)[\,\cdot\,])]$. For this, we are going to
use Lemma~\ref{le:loccomp}~(ii) with~${T:=\bA(f)\circ\nabla}$ and~$S:=\p\beta(f)$.
 Therefore, we recall \eqref{locforA}
and infer from Proposition~\ref{P:D1} below that, due to~$f\in W^{1,\infty}(\R^N)$, the operators
$(1+|\nabla f(x^\e_j)|^2)D_{0,e_k}^{\bar\phi,\nabla f(x^\e_j)}$
are uniformly bounded  with respect to~${\e\in(0,1)}$,
  $0\leq j\leq m(\e)$, and $1\leq k\leq N$ in~${\kL(H^{s-1}(\R^N))}$.

 Recalling Lemma~\ref{le:estdbeta}, we conclude from \eqref{locforA} and Lemma~\ref{le:loccomp}~(ii) that
\begin{align*}
  & \bA(f)[\nabla(\p\beta(f)[\,\cdot\,])]\\
  &\locsim 
   \bigg(-{\sf A}_0,-(1+|\nabla f(x_j^\e)|^2)\Big(1+2a_\mu\sum_{i=1}^N \sfB_{0.e_i}^{\bar\phi}
    (f)[\p_i\beta](x_j^\e)\Big)
   \sum_{k=1}^ND_{0,e_k}^{\bar\phi,\nabla f(x_j^\e)} \frac{\p}{\p x_k}\bigg)\\
   & \,\qquad +\bigg(0,2a_\mu(1+|\nabla f|^2)(x_j^\e)
   \sum_{i,k=1}^N\p_i\beta(x_j^\e)
   D_{0,e_i}^{\bar\phi,\nabla f(x_j^\e)}D_{0,e_k}^{\bar\phi,\nabla f(x_j^\e)}\frac{\p}{\p x_k}\bigg).
\end{align*}
The property \eqref{locpsi1} follows from this relation and \eqref{locderforA} by adding and applying Lemma~\ref{L:devvor}.
\end{proof}

 We now address the question of invertibility of $\omega-\Psi(0)$ for sufficiently large $\omega\in\R$.
\begin{prop}\label{P:inv}
Let $f\in H^s(\R^N)$ and $a\in H^{s-1}(\R^N)$ such that $1+a>0$, and set
\[
T:=(1+a)\sum_{k=1}^N \sfB^{\bar\phi}_{0,e_k}(f) \frac{\p}{\p x_k}
\]
Then there exists $\omega_0\in\R$ such that $\lambda+T\in\kL(H^s(\R^N),H^{s-1}(\R^N))$ is  an isomorphism for all~${\lambda\in [\omega_0,\infty)}$. 
\end{prop}
\begin{proof}
To start, we define the  path 
$[\tau\mapsto T(\tau)] \in{\rm C}([0,1], \kL(H^s(\R^N), H^{s-1}(\R^N)))$ by
\[
T(\tau):=(1+\tau a)\sum_{k=1}^N \sfB^{\bar\phi}_{0,e_k}(\tau f) \frac{\p}{\p x_k}.
\]
Let $\eta\in(0,1)$ be chosen such that for all  $\tau\in[0,1]$ we have
\[
\|\nabla f\|_\infty\leq \eta^{-1}\qquad\text{and}\qquad \eta\leq 1+\tau a\leq \eta^{-1}.
\] 
Next, we introduce the Fourier multipliers
\[
{\sf T}_{\alpha,A}:=\alpha \sum_{k=1}^N D_{0,e_k}^{\bar\phi,A}\frac{\p}{\p x_k},\qquad \alpha\in[\eta,\eta^{-1}],\quad |A|\leq \eta^{-1}, 
\]
and infer from Lemma~\ref{L:SGRT}, by using standard Fourier analysis, that there is a constant~${\kappa\geq 1}$ such that  
\begin{equation}\label{kappa}
\kappa\|(\lambda+{\sf T}_{\alpha,A})[h]\|_{H^{s-1}}\geq \lambda\cdot \|h\|_{H^{s-1}}+\|h\|_{H^{s}}
\end{equation}
for all $\alpha\in[\eta,\eta^{-1}]$, $|A|\leq \eta^{-1}$, and $\lambda\geq 1$. 

 Applying  Proposition~\ref{T:LOC}, we conclude  that there exist~${\e\in(0,1)}$, a   constant~${K>0}$, and Fourier multipliers $T_{j,\tau}^\e\in\kL(H^s(\R^N), H^{s-1}(\R^N))$ such that for all~$0{\leq j\leq m(\e), h\in H^s(\R^N)}$, and~$\tau\in[0,1]$ it holds 
\begin{equation}\label{eq:locinv}
\|\pi_j^\varepsilon T(\tau)[h]-T_{j,\tau}^\e[\pi_j^\varepsilon h]\|_{H^{s-1}}\leq (2\kappa)^{-1}\|\pi_j^\varepsilon h\|_{H^s}+K\|h\|_{H^{s'}},
\end{equation}
where $\kappa \geq 1$ is the constant in \eqref{kappa}.
Moreover, the Fourier multipliers $T_{j,\tau}^\e$ all belong to the set 
$\{ {\sf T}_{\alpha,A}\,:\, \alpha\in[\eta,\eta^{-1}],\, |A|\leq \eta^{-1}\}$.

 From \eqref{kappa} and \eqref{eq:locinv}, we deduce 
for~$0\leq j\leq m(\e)$,  $h\in H^s(\R^N)$, $\tau\in[0,1]$,  and~$\lambda\geq 1$ that
\begin{equation*}
\begin{aligned}
2\kappa\|\pi_j^\varepsilon (\lambda+T(\tau))[h]\|_{H^{s-1}}&\geq 2\kappa\|\pi_j^\varepsilon (\lambda+T_{j,\tau}^\e)[\pi_j^\varepsilon h]\|_{H^{s-1}}
-2\kappa\|\pi_j^\varepsilon T(\tau)[h]-T_{j,\tau}^\e[\pi_j^\varepsilon h]\|_{H^{s-1}}\\
&\geq 2 \lambda\cdot \| \pi_j^\varepsilon h\|_{H^{s-1}}+\|\pi_j^\varepsilon h\|_{H^{s}}-2\kappa K\|h\|_{H^{s'}}.
\end{aligned}
\end{equation*}
Summing over $ 0\leq  j\leq m(\e)$ and using \eqref{IP}, the equivalence of the norm defined in~\eqref{eqnorm} to the standard~$\|\cdot\|_{H^r}$-norm, $r\geq 0$, and Young's inequality, 
we conclude that there exist constants~${\kappa_0\geq 1}$ and $\omega_0>0$ such that 
for all~$h\in H^s(\R^N)$, $\tau\in[0,1]$, and $\lambda\geq \omega_0$  we have 
\begin{equation*}
\kappa_0\|(\lambda+T(\tau))[h]\|_{H^{s-1}}\geq \lambda\cdot \|h\|_{H^{s-1}}+\|h\|_{H^{s}}\geq \|h\|_{H^{s}}.
\end{equation*}
Since $T(0)$ is the Fourier multiplier with symbol~$[z\mapsto|z|/2]$, $z\in\R^N$  (see the proof of Lemma~\ref{L:SGRT}), the operator $\lambda+ T(0)$ is invertible for all $\lambda>0$, and the method of continuity ensures (cf. \cite[Proposition I.1.1.1]{Am95}) that $\lambda+T(1)$ is invertible for all $\lambda\geq \omega_0$, which completes the proof.
\end{proof}

 By Proposition \ref{P:D1} and Lemma \ref{L:SGRT} we can characterize the symbols of the Fourier multipliers $A_{j,\tau}^\e$, $\e\in(0,1)$, $0\leq j\leq m(\e)$, and~$\tau\in[0,1]$ identified in Proposition~\ref{P:Local}. In connection with Lemma~\ref{L:MP1}, Lemma~\ref{L:61}~(i), and~\eqref{dop}, 
 we find from these results that   these Fourier multipliers have symbols of the form
\begin{equation}\label{gens}
\Big[z\mapsto  -m_0(z)|z|+i\sum_{k=1}^N m_k(z)z_k \Big]:\R^N\to\mathbb{C}, 
\end{equation}
with  real-valued functions $m_k=m_{k,j,\tau}^\e\in  L_\infty(\R^N)$, $0\leq k\leq N$, which satisfy, for $z\in\R^N$,
\begin{equation}\label{gesy}
m_0(z)\in[\eta,\eta^{-1}]\qquad\text{and}\qquad |m_k(z)|\leq \eta^{-1},\quad 1\leq k\leq N,
\end{equation}
for some $\eta\in (0,1)$ depending only on $f\in\cO$.

Using Fourier analysis, it is straightforward to prove that if ${\sf M}$ is a 
Fourier multiplier with a symbol satisfying~\eqref{gens} and~\eqref{gesy}, then there exists a constant $\kappa \geq 1$ (which depends only on $\eta$)
such that 
\begin{equation}\label{estFour1}
\kappa\|(\lambda+{\sf  M})[h]\|_{H^{s-1}}\geq \lambda\cdot \|h\|_{H^{s-1}}+\|h\|_{H^{s}},\qquad \re\lambda\geq 1,\, h\in H^{s}(\R^N).
\end{equation} 
\medskip

We are now in a position to prove Theorem~\ref{T:GP}.
\begin{proof}[Proof of Theorem~\ref{T:GP}]
 The proof of Theorem~\ref{T:GP} follows by combining  Proposition~\ref{P:Local},  Proposition~\ref{P:inv}, \eqref{estFour1}, \eqref{IP}, 
 and the equivalence of the norm defined in \eqref{eqnorm} to the standard norm in $H^r(\R^N)$.
The details are similar to those in the proof of \cite[Theorem 6]{AM22}, and therefore we omit them. 
\end{proof}

\subsection*{Proof of Theorem~\ref{MT}}
 The proof applies the well-posedness theory for fully nonlinear  abstract parabolic problems from \cite[Chapter~8]{L95}. 
 This theory uses weighted H\"older spaces in time~${\rm C}_\beta^\beta((0,T],E)$,  where $E$ is a Banach space, $\beta\in(0,1)$,  and $T>0$.  These spaces consist of the
 bounded functions $u:(0,T]\to E$ such that 
 \[
 [t\mapsto t^\beta u(t)]\in {\rm C}^\beta([0,T],E).
\]

\begin{proof}[Proof of Theorem~\ref{MT}]  \,

\noindent{\bf Well-posedness.} Let $\cO$ be the open subset of $H^s(\R^N)$ introduced in~\eqref{dop} and recall from~\eqref{spphi}
that  $\Phi\in {\rm C}^{\infty}(\cO, H^{s-1}(\R^N))$ with $\p\Phi(f)$ being, according to Theorem~\ref{T:GP}, the generator of a strongly continuous  analytic semigroup on $\kL(H^{s-1}(\R^N))$.
Therefore the assumptions of \cite[Theorem 8.1.1]{L95} are satisfied in the context of the evolution problem~\eqref{NNEP}.
 This result ensures that, given $f_0\in\cO$, there exists a strict solution~$f$ to~\eqref{NNEP}  on some time interval $[0,T]$ which satisfies (since \eqref{NNEP} is autonomous)
\[
f\in{\rm C}([0,T], \cO)\cap {\rm C}^1([0,T], H^{s-1}(\R^N))\cap {\rm C}_\beta^\beta((0,T], H^{s}(\R^N))
\]
for all $\beta\in(0,1)$. Moreover, this solution is unique within the set of functions
\[
 \bigcup_{\beta\in(0,1)}{\rm C}([0,T], \cO)\cap {\rm C}^1([0,T], H^{s-1}(\R^N))\cap {\rm C}_\beta^\beta((0,T], H^{s}(\R^N)).
\]
To improve the uniqueness claim, as stated in Theorem~\ref{MT}, we fix  $s'\in(1+N/2,s)$ and set~${\beta:=s-s'\in(0,1)}$. 
Then merely  $f\in{\rm C}([0,T], \cO)\cap {\rm C}^1([0,T], H^{s-1}(\R^N))$ together with the interpolation property~\eqref{IP} ensures that 
\[
f\in {\rm C}^\beta([0,T], H^{s'}(\R^N))\hookrightarrow {\rm C}^\beta_\beta([0,T], H^{s'}(\R^N)).
\]
The uniqueness result  of \cite[Theorem 8.1.1]{L95} applied in the context of \eqref{NNEP} (with $s$ replaced by $s'$) implies now that $f$ is indeed unique within the set 
${\rm C}([0,T], \cO)\cap {\rm C}^1([0,T], H^{s-1}(\R^N))$.  Arguing as in \cite[Section 8.2]{L95},  we can extend this solution to a maximal solution~${f=f(\cdot;f_0)}$ defined on a maximal time interval 
$[0,T^+(f_0))$ with~$T^+(f_0)\in(0,\infty]$.
Moreover, \cite[Proposition~8.2.3]{L95} ensures that the solution map defines a semiflow on $\cO$.
Recalling Proposition~\eqref{P:1}, we established the well-posedness claim.\medskip

\noindent{\bf Parabolic smoothing.} 
Based on the well-posedness property established above and using a parameter trick  applied also to other problems, cf., e.g., \cite{An90, ES96, PSS15, MBV19}, 
we now establish the  parabolic smoothing property in Theorem~\ref{MT}~(ii).

To start, let $f=f(\cdot;f_0)$ denote the maximal solution to \eqref{NNEP} with maximal existence interval $[0,T^+(f_0))$.
 It is sufficient to show that for each $k\in\N$ we have 
\be\label{fsmooth}
f\in {\rm C}^\infty((0,T^+), H^k(\R^N)).
\ee
To establish~\eqref{fsmooth}, we define 
 for each $\lambda:=(\lambda_1,\lambda_2)\in \cV:=(0,\infty)\times\R^{N}$ the function
\[
f_\lambda(t,x):=f(\lambda_1t,x+\lambda_2 t),\qquad x\in\R^N,\quad 0\leq t<T^+(\lambda,f_0):=T^+(f_0)/\lambda_1.
\]
Then $f_{\lambda}\in {\rm C}([0,T^+(\lambda,f_0)),\mathcal{O})\cap {\rm C}^1([0,T^+(\lambda,f_0)), H^{s-1}(\mathbb{R}^N)) $ is a solution to  the  parameter dependent evolution problem  
\begin{equation}\label{PDNNEP}
\frac{{\rm d}f}{{\rm d}t}= \Psi(f,\lambda) ,\quad t\geq0,\qquad f(0)=f_0,
\end{equation}
where $\Psi:\cO\times \cV\to H^{s-1}( \R^N)$ is defined by
\[
 \Psi(f,\lambda)=\lambda_1\Phi(f)+\lambda_2\cdot\nabla f.
\]
It is straightforward to infer from Lemma~\ref{L:61}~(ii) that $\Psi\in  {\rm C}^{\infty}(\cO\times \cV, H^{s-1}(\R^N))$ has partial Fr\'echet derivative with respect to $f$  given by
\[
\p_f\Psi(f,\lambda)=\lambda_1\p\Phi(f)+\lambda_2\cdot\nabla.
\]
Since $\p_{x_j}$, $1\leq j\leq N$, is the Fourier multiplier with symbol~$[z\mapsto i z_j]$, we may argue as in the proof of Theorem~\ref{T:GP} to deduce that 
$\p_f\Psi(f,\lambda)$ generates a strongly continuous  analytic semigroup on $\kL(H^{s-1}(\R^N))$ for each $(f,\lambda)\in \cO\times \cV$.
The arguments in part (i) of the proof together with  \cite[Theorem 8.1.1 and Corollary 8.3.8]{L95} now ensure that \eqref{PDNNEP} has for each~${(f_0,\lambda)\in \cO\times \cV}$  a unique maximal solution
\[
f=f(\cdot;(f_0,\lambda))\in {\rm C}([0,t^+(\lambda,f_0)),\mathcal{O})\cap {\rm C}^1([0,t^+(\lambda,f_0)), H^{s-1}(\mathbb{R}^N)),
\]
where $t^+(\lambda,f_0)\in(0,\infty]$ is the maximal existence time of the solution.
 Moreover, the set
\[
\Xi:=\{(t,f_0,\lambda)\,:\, (f_0,\lambda)\in\cO\times \cV,\, 0<t<t^+(f_0,\lambda)  )\} 
\]
is  open and 
\[
[(t,f_0,\lambda)\mapsto f(t;(f_0,\lambda))]\in {\rm C}^\infty(\Xi, H^s(\R^N)).
\]
In view of the equivalence of the problems \eqref{NNEP} and \eqref{PDNNEP} we may conclude that 
\[
t^+(\lambda,f_0)=\frac{T^+(f_0)}{\lambda_1}\qquad\text{and}\qquad f(t;(f_0,\lambda))=f_\lambda(t),\quad 0\leq t<  \frac{T^+(f_0)}{\lambda_1}. 
\]

Fix $ t_0\in(0,T_+(f_0))$ and choose $\varepsilon>0$ such that $t_0<T^+(f_0)/\lambda_1$ for all $\lambda\in\bB_\varepsilon(e_1)\subset \cV$.
It then follows that $\{t_0\}\times \{f_0\}\times \bB_\varepsilon(e_1)\subset \Xi$ with
\[
[\lambda\mapsto f(t_0;(f_0,\lambda))=f_\lambda(t_0)]\in {\rm C}^\infty( \bB_\varepsilon(e_1),  H^s(\R^N)).
\]
This property immediately implies \eqref{fsmooth}, and the proof is complete.
 \end{proof}

\appendix
%%%%%%%%%%%%%%%%%%%%%%%%%%%%%%%%%%%%%%%%%%%%%%%%%%
%%%%%%%%%%%%%%%%%%%%%%%%%%%%%%%%%%%%%%%%%%%%%%%%%%
%%%%%%%%%%%%%%%%%%%%%%%%%%%%%%%%%%%%%%%%%%%%%%%%%%
%%%%%%%%%%%%%%%%%%%%%%%%%%%%%%%%%%%%%%%%%%%%%%%%%%
 %%%%%%%%%%%%%%%%%%%%%%%%%%%%%%%%%%%%%%%%%%%%%%%%%%
%%%%%%%%%%%%%%%%%%%%%%%%%%%%%%%%%%%%%%%%%%%%%%%%%%
%%%%%%%%%%%%%%%%%%%%%%%%%%%%%%%%%%%%%%%%%%%%%%%%%%
%%%%%%%%%%%%%%%%%%%%%%%%%%%%%%%%%%%%%%%%%%%%%%%%%%
 \section{Layer potentials generated on unbounded graphs  }\label{Sec:A}
%%%%%%%%%%%%%%%%%%%%%%%%%%%%%%%%%%%%%%%%%%%%%%%%%%
%%%%%%%%%%%%%%%%%%%%%%%%%%%%%%%%%%%%%%%%%%%%%%%%%%
%%%%%%%%%%%%%%%%%%%%%%%%%%%%%%%%%%%%%%%%%%%%%%%%%%
%%%%%%%%%%%%%%%%%%%%%%%%%%%%%%%%%%%%%%%%%%%%%%%%%%
 %%%%%%%%%%%%%%%%%%%%%%%%%%%%%%%%%%%%%%%%%%%%%%%%%%
%%%%%%%%%%%%%%%%%%%%%%%%%%%%%%%%%%%%%%%%%%%%%%%%%%
%%%%%%%%%%%%%%%%%%%%%%%%%%%%%%%%%%%%%%%%%%%%%%%%%%
%%%%%%%%%%%%%%%%%%%%%%%%%%%%%%%%%%%%%%%%%%%%%%%%%%

Let $\alpha\in(0,1)$, $p\in(1,\infty)$, and choose  $f,\, \beta:\R^N\to\R$   such that $ \nabla f\in {\rm BUC}^\alpha(\R^N)^N$ 
and~${\beta\in {\rm BUC}^\alpha(\R^N)\cap  L_p(\R^N)}$. In this  appendix we use the notation introduced in Section~\ref{Sec:2}; see in particular~\eqref{opm}-\eqref{rzx}.

We are interested in the properties of the function $V_i:\R^{N+1}\setminus\G\to\R$, $1\leq i\leq N+1$, defined by
\begin{equation}\label{Vei}
V_i(z):=\frac{1}{|\s^N|}\int_{\G}\frac{(z-\ov z)_i}{|z-\ov z|^{N+1}}\wt \beta(\ov z)\,{\rm d}\G(\ov z)
\end{equation}
for $z=(x,y)\in\R^{N+1}\setminus\G$, where $\wt\beta:\G\to\R$ is given by $ \wt\beta:=\beta\circ\Xi^{-1}.$

Since $\beta\in L_p(\R^N)$, we infer from \eqref{Vei} that $V_i$, $1\leq i\leq N+1$, is well-defined and smooth in $\R^{N+1}\setminus\G$.
We  prove in Lemma~\ref{L:W1} below  that  $V_i$  can be also evaluated at any point~${z^0=(x^0,f(x^0))\in\G}$ if the integral \eqref{Vei} is understood as a principle value integral.
To this end we first introduce some notation.
Given $\eta\in(0,1]$, we define  the hypersurfaces
\begin{equation}\label{notation}
\G_\eta:=\{z_\xi\in\G\,:\, \xi\in \bB_\eta(x^0)\}\qquad\text{and}\qquad \G_{\eta,1}:=\G_1\setminus\G_\eta.
\end{equation}

Furthermore, for fixed (but arbitrary)  $z\in\R^{N+1}\setminus\Gamma$ and  $z^0\in \Gamma$ we define the mappings $R:=R(z,\cdot):\Gamma\to\R$ and $R^0:=R(z^0,\cdot):\Gamma\setminus\{z^0\}\to\R$ by 
 \begin{equation}\label{fettR}
 R(\ov z):=R(z,\ov z):=\frac{1}{|\s^N|}\frac{z-\ov z}{|z-\ov z|^{N+1}}\quad\text{and}\quad R^0(\ov z):=R^0(z^0, \ov z):=\frac{1}{|\s^N|}\frac{z^0-\ov z}{|z^0-\ov z|^{N+1}}.
 \end{equation}

The integral \eqref{Vei} can now be written as
\begin{equation}
V_i(z)=\int_{\G}R_i(z,\ov z)\wt\beta (\ov z)\,{\rm d}\G(\ov z)=\int_{\G}R_i\wt\beta \,{\rm d}\G,\qquad z\in\R^{N+1}\setminus\G,
\end{equation}
 where $R_i$ and $R_i^0$  are the $i$-th components of $R$ and $R^0$, $1\leq i\leq N+1$.
Throughout this section we denote by $C$ positive constants that depend at most on~$N$,~$f$, and~$\beta$.

\begin{lemma}\label{L:W1}
Given $1\leq i\leq N+1,$   the limit
\begin{equation}\label{claim1}
\PV\int_{\G}R_i^0\wt\beta\,{\rm d}\G:=\lim_{\eta\searrow 0}\int_{\G\setminus\G_{\eta}}R_i^0\wt\beta\,{\rm d}\G
\end{equation}
exists in $\R$.
Moreover, there exists a  positive constant~$C$ such that for all $\eta \in(0,1]$ and~${z^0 \in\G}$ we have
\begin{equation}\label{claim2}
\Big|\int_{\G_{\eta,1}}R^0_i\wt\beta\,{\rm d}\G\Big|\leq C.
\end{equation}
\end{lemma}
\begin{proof}
We first note that
\[
\int_{\G\setminus\G_{\eta}}R_i^0\wt\beta\,{\rm d}\G=\int_{\G\setminus\G_{1}}R_i^0\wt\beta\,{\rm d}\G
+\int_{\G_{\eta,1}}R_i^0\wt\beta\,{\rm d}\G,
\]
and  H\"older's inequality  together with the assumption $\beta\in L_p(\R^N) $ imply that the first integral exists as
\begin{equation}\label{bou0}
\int_{\G\setminus\G_{1}}\big|R_i^0\wt\beta\big|\,{\rm d}\G\leq   C\int_{\R^n\setminus\bB_1(0)}\frac{|\beta(x^0-\xi)|}{|\xi|^N}\,{\rm d}\xi\leq C\|\beta\|_p\leq C.
\end{equation}
It remains to consider the integral
\[
T(\eta):=|\s^N|\int_{\G_{\eta,1}}R_i^0\wt\beta\,{\rm d}\G=\int_{A_{\eta,1}}\frac{(z^0-{z_\xi})_i}{|z^0-z_\xi|^{N+1}}\wh \beta(\xi)\,{\rm d}\xi,
\]
where  $A_{\eta,1} $ is the annulus $\bB_1(x^0)\setminus \bB_\eta(x^0)$ and $\wh \beta:=\sqrt{1+|\nabla f|^2}\beta\in{\rm BUC}^\alpha(\R^N).$
Observing that  $[\xi\mapsto 2x^0-\xi]:A_{\eta,1}\to A_{\eta,1}$ is a bijection, a change of coordinates leads us to
\begin{equation*}
 T(\eta)=\int_{A_{\eta,1}}\frac{(z^0-{z_{2x^0-\xi}})_i\wh\beta(2x^0-\xi)}{|z^0-z_{2x^0-\xi}|^{N+1}}\,{\rm d}\xi,
\end{equation*}
and therefore
\begin{equation*}
 |2T(\eta)| =\bigg|\int_{A_{\eta,1}} \bigg(\frac{(z^0-{z_\xi})_i\wh\beta(\xi)}{|z^0-z_\xi|^{N+1}}+\frac{(z^0-{z_{2x^0-\xi}})_i\wh\beta(2x^0-\xi)}{|z^0-z_{2x^0-\xi}|^{N+1}}\bigg) {\rm d}\xi\bigg|\leq T_{a}(\eta)+T_{b}(\eta),
\end{equation*}
where
\begin{equation*}
 \begin{aligned}
 T_{a}(\eta)&:=\int_{A_{\eta,1}}\Big|\frac{(z^0-{z_\xi})_i(\wh\beta(\xi)-\wh\beta(2x^0-\xi))}{|z^0-z_\xi|^{N+1}}\Big| \,{\rm d}\xi,\\
T_{b}(\eta)&:=\|\wh\beta\|_\infty\int_{A_{\eta,1}}\Big| \frac{(z^0-{z_\xi})_i }{|z^0-z_\xi|^{N+1}}-\frac{(z_{2x^0-\xi}-z^0)_i }{|z_{2x^0-\xi}-z^0|^{N+1}}\Big|\,{\rm d}\xi.
 \end{aligned}
\end{equation*}
We then estimate
\begin{equation}\label{bou1}
 T_{a}(\eta)\leq 2[\wh\beta]_{\alpha}\int_{\bB_1(0)}|\xi|^{\alpha-N} \,{\rm d}\xi \leq C.
\end{equation}
Moreover, setting $A:=  z^0-z_\xi  $ and $B:= z_{2x^ 0-\xi}-z^0$,  we have that 
\begin{align*}
    A_i-B_i&=0,\qquad1\leq i\leq N,\\
    |A_{N+1}-B_{N+1}|&
    =|f(x^0)-f(\xi) +f(x^0)-f(2x^0-\xi)|\leq C[\nabla f]_\alpha|x^0-\xi|^{1+\alpha},
\end{align*}
and, for $1\leq i\leq N+1$,
\[
\Big|\frac{A_i}{|A|^{N+1}}-\frac{B_i}{|B|^{N+1}}\Big|\leq \frac{|A_i-B_i|}{|A|^{N+1}}+|B|\Big|\frac{1}{|A|^{N+1}}-\frac{1}{|B|^{N+1}}\Big|. 
\]
 The first term is nonzero only if $i=N+1$. In that case
 \[\frac{|A_{N+1}-B_{N+1}|}{|A|^{N+1}}\leq C[\nabla f]_\alpha|x^0-\xi|^{\alpha-N}.\]
For the second term we estimate
\begin{align*}
   |B| \Big|\frac{1}{|A|^{N+1}}-\frac{1}{|B|^{N+1}}\Big|&\leq C|B|\frac{\big||A|-|B|\big|(|A|^N+|B|^N)}{|A|^{N+1}|B|^{N+1}}\\
   &\leq C|A-B||x^0-\xi|^{-N-1}\leq C[\nabla f]_\alpha|x^0-\xi|^{\alpha-N}.
\end{align*}

Gathering these estimates we obtain 
\begin{equation}\label{bou4}
 T_{b}(\eta)\leq  C\|\wh\beta\|_{\infty}[\nabla f]_{\alpha}\int_{\bB_1(0)}|\xi|^{\alpha-N} \,{\rm d}\xi \leq C .
\end{equation}
The claims \eqref{claim1}-\eqref{claim2} follow now from \eqref{bou0}-\eqref{bou4}.
\end{proof}

Having established Lemma~\ref{L:W1}, we now state the main results of this appendix.
\begin{prop}\label{P:W1}
There exists a  constant $C>0$ such that for all $1\leq i\leq N+1$, $z^0\in\G$, and $z\in\0^\pm$ with $|z-z^0|\leq 1/4$ we have
\begin{equation}\label{jumpvelo}
\Big|V_i(z)-\Big(\PV\int_{\G}R_i^0\wt\beta\,{\rm d}\G\pm\frac{\wt\nu^i\wt\beta   }{2}(z^0)\Big)\Big|\leq C|z-z^0|^\alpha.
\end{equation}
Moreover,  $V_i^\pm:=V_i|_{\0^\pm}$ has a continuous extension which belongs to ${\rm C}(\ov {\0^\pm})$.
\end{prop}

According to Proposition~\ref{P:W1}, the continuous extension of $V_i^\pm $ to~$\ov{\0^\pm}$, $1\leq i\leq N+1$, which is denoted again by $V_i^\pm$,  satisfies
\[
V_i^\pm (z^0)=\PV\int_{\G}R_i^0\wt\beta\,{\rm d}\G\pm\frac{ \wt\nu^i\wt\beta}{2}(z^0) \qquad\text{for $z^0\in\G$}.
\] 
As a second important result we prove that $V_i^\pm$ vanishes at infinity.

\begin{prop}\label{P:W2}  We have $V_i^\pm(z)\to0$ as $|z|\to\infty$,  $z\in \Omega^\pm$.
\end{prop}

The proofs of these results are postponed to the end of the section as they require some preparation.
We point out that   it suffices to establish these results for $z\in\0^-$  as the results for $z\in\0^+$ are obtained analogously.
Therefore we restrict our considerations in the following to the case when $z\in\0^-$.

 Our first preliminary result is the following lemma.
\begin{lemma}\label{L:W2}
There is a  constant $C>0$ such that for all $1\leq i\leq N+1$, $z^0\in\G$, and~$z\in\0^-$ with $|z-z^0|\leq 1/4$  we have 
\begin{equation}\label{estlw2}
\Big|\int_{\G\setminus\G_1}R_i\wt\beta\,{\rm d}\G- \int_{\G\setminus\G_1}R_i^0\wt\beta\,{\rm d}\G\Big|\leq C|z-z^0|^\alpha.
\end{equation}
\end{lemma}
\begin{proof}
Observe first  that
\be\label{distzz0}
2|z-\bar z|\geq |z^0-\bar z|\quad\text{for all $z\in \bB_{1/4}(z^0)$, $\bar z\in\Gamma\setminus\Gamma_1$.}
\ee

Define furthermore the function $G:\bB_{1/4}(z^0)\times (\Gamma\setminus\Gamma_1)
\to\R^{N+1}$ by
\[G(z,\bar z):=\frac{1}{|\s^N|}\frac{z-\bar z}{|z-\bar z|^{N+1}},\qquad z\in \bB_{1/4}(z^0),\quad\bar z\in \Gamma\setminus\Gamma_1.\]
This function is differentiable with respect to its first argument, and for the derivative we obtain by direct calculation and \eqref{distzz0}  that there exists a constant $C>0$ such that
\[\| \p_zG(z,\bar z)\|_{\kL(\R^{N+1})}\leq C|z-\bar z|^{-N-1}\leq C|z^0-\bar z|^{-N-1},\]
and consequently
\begin{equation}\label{estGzz}
|G(z,\bar z)-G(z^0,\bar z)|\leq C|z^0-z||z^0-\bar z|^{-N-1},\qquad  z\in \bB_{1/4}(z^0),\quad\bar z\in \Gamma\setminus\Gamma_1.
\end{equation}
Using this  and H\"older's inequality we estimate the left side of \eqref{estlw2}  by 
\begin{align*}
    &\int_{\Gamma\setminus\Gamma_1}|G(z,\bar z)-G(z^0,\bar z)||\wt\beta(\bar z)|\,{\rm d}\Gamma(\bar z)\leq C|z^0-z|\int_{\Gamma\setminus\Gamma_1}\frac{|\wt\beta(\bar z)|}{|z^0-\bar z|^{N+1}}\,{\rm d}\Gamma(\bar z)\\
    &\leq  C|z-z^0|^\alpha\int_{\R^N\setminus\bB_1(x^0) }
\frac{|\beta(\xi)|}{|z-z_\xi|^{N+1}} \,{\rm d}\xi\leq C|z-z^0|^\alpha.
\end{align*}
\end{proof}

It remains to  estimate the  contributions to the (singular) integrals  
in \eqref{jumpvelo} from~$\G_1$.
This  requires some additional preparation.
 To this end, we introduce the   Lipschitz domain
\be\label{po1}
\0_1:=\big\{(x,y)\,:\,x\in\bB_1(x^0),\, y\in\big(f(x^0)-1-\|\nabla f\|_\infty,f(x)\big) \big\},
\ee
 and define
\begin{align*}
u(z):= \int_{ \p\0_1}R\cdot\wt \nu\,{\rm d}\G=\int_{\p\0_1}\sum_{i=1}^{N+1}R_i\wt \nu^i\,{\rm d}\G,\qquad  z\in\Omega_1,
\end{align*}
where $ R:=R(z,\cdot)$ is defined in \eqref{fettR}.
Observing that 
\begin{equation}\label{es121}
|(z_{\xi_1}-z_{\xi_2})\cdot \nu(z_{\xi_2})|\leq [\nabla f]_\alpha |\xi_1-\xi_2|^{1+\alpha},\qquad \xi_1, \xi_2\in\R^N,
\end{equation}
 we may evaluate $u$ also at $z^0$.   In Lemma~\ref{L:W3} we show that this function is constant~$-1$  in~$\0_1$, while $u(z^0)=-1/2$: 
\begin{lemma}\label{L:W3}
 We have
\[
u(z)=
\left\{
\arraycolsep=1.4pt
\begin{array}{clll}
-1&,& \text{if $z\in\0_1,$}\\[1ex]
-1/2&,& \text{if $z=z^0$.}
\end{array}
\right.
\] 
\end{lemma}
\begin{proof}
Let first $z\in\0_1$. 
Given $\e>0$ with $\bB_\e(z)\subset\0_1$, we note that  $R=R(z,\cdot)$  belongs to~${{\rm BUC}^1( \R^{N+1}\setminus\bB_\e(z))}$ and satisfies ${\rm div\,} R=0$.
 Applying Stokes' theorem to this vector field on~$\0_1\setminus \bB_\e(z)$ then yields
\[
u(z)= -\frac{1}{|\s^N|}\int_{\p\bB_\e(z)}\frac{1}{|z-\ov z|^{N}}\,{\rm d} S(\ov z)=-1.
\]
To compute  $u(z^0)$ we extend the unit normal vector field $\wt \nu$ to all of $\p\0_1$, keeping the same notation.
 Recalling \eqref{es121},  we have
\[
u(z^0)=\int_{ \p\Omega_1}R^0\cdot\wt \nu\,{\rm d}\G=\lim_{\e\searrow 0}\int_{\p\0_1\setminus \bB_\e(z^0)}R^0\cdot\wt \nu\,{\rm d}\G .
\] 
Moreover, since  $R^0\in{\rm BUC}^1( \0_1\setminus\bB_\e(z^0))$ satisfies ${{\rm div\,} R^0=0}$, Stokes' theorem yields
\[
\int_{\p\0_1\setminus \bB_\e(z^0)}R^0\cdot\wt \nu \,{\rm d}\G+\int_{\p\bB_\e(z^0) \cap \0_1}R^0(\ov z)\cdot \frac{z^0-\ov z}{|z^0-\ov z|}\,{\rm d}\G(\ov z)=0,
\]
hence
\[
u(z^0)=-\lim_{\e\searrow 0}\frac{|\p\bB_\e(z^0) \cap \0_1|}{\e^N|\s^N|} =-\frac{1}{2},
\]
the last equality being a direct consequence of the fact that $f\in{\rm BUC}^{1+\alpha}(\bB_1(x^0)).$
\end{proof}

With the help of Lemma~\ref{L:W3} we establish the following result.

\begin{lemma}\label{L:W4} 
 There exists a  constant $C>0$ such that for all~${z^0\in\G}$ and~${z\in\0^-}$  that satisfy~${|z-z^0|\leq 1/4}$
we have
\[
\Big|\int_{\G_1}R\cdot\wt \nu\,{\rm d}\G- \Big(\int_{\G_1} R ^0 \cdot\wt\nu \,{\rm d}\G-\frac{1}{2}\Big)\Big|\leq C|z-z^0|^\alpha.
\]
\end{lemma}
\begin{proof}
 Because of Lemma~\ref{L:W3}, we have
\[
\int_{\G_1}R\cdot\wt \nu\,{\rm d}\G=-1- \int_{\p\0_1\setminus\G_1} R\cdot\wt \nu \,{\rm d}\G \qquad\text{and}\qquad
\int_{\G_1} R ^0 \cdot\wt\nu \,{\rm d}\G
=-\frac{1}{2}-\int_{\p\0_1\setminus\G_1} R^0\cdot\wt \nu\,{\rm d}\G.
\]
Therefore,  since \eqref{estGzz} is still valid when replacing $\G\setminus\G_1$ by $\p\Omega_1\setminus\G_1$, we have
\begin{align*}
&\Big|\int_{\G_1}R\cdot\wt \nu\,{\rm d}\G- \Big(\int_{\G_1} R ^0 \cdot\wt\nu \,{\rm d}\G-\frac{1}{2}\Big)\Big|
\leq\frac{1}{|\s^N|}\int_{\p\0_1\setminus\G_1}\Big|\frac{z-\ov z}{|z-\ov z|^{N+1}}-\frac{z^0-\ov z}{|z^0-\ov z|^{N+1}}\Big|\,{\rm d}\G(\ov z)\\
&\leq C\int_{\p\0_1\setminus\G_1}|z^0-z||z^0-\bar z|^{-N-1}\,{\rm d}\G(\ov z)\leq C|z-z^0|.
\end{align*} 
 For the last inequality we used that $|z^0-\ov z|\geq 1$  for $\ov z\in\p\0_1\setminus\G_1$    as well as~${|\p\0_1|\leq C(  \|\nabla f\|_{\infty}})$.
\end{proof}

As a further preliminary result we establish the following lemma.

\begin{lemma}\label{L:W5}
There exists a  constant $C>0$ such that  for all $1\leq i\leq N$, ${z\in\0^-}$, and~${z^0\in\G}$  with $|z-z^0|\leq 1/4$ we have  
\begin{equation}\label{hoeld}
\Big|\int_{\G_1}R_i\wt \nu^{N+1}-R_{N+1}\wt \nu^i\,{\rm d}\G-\PV\int_{\G_1}R_i^0\wt \nu^{N+1}-R^0_{N+1}\wt \nu^i\,{\rm d}\G\Big|\leq C|z-z^0|^\alpha.
\end{equation}  
\end{lemma}
\begin{proof}
Fix $\e\in(0,1)$ and recall \eqref{notation}. Observe that the map
$[\xi\mapsto |z^0- z_\xi|^{-(N-1)}]$ is smooth on $\overline{\bB_1(x^0)}\setminus\bB_\e(x^0)$. Applying Stokes' theorem on this domain yields
\begin{align*}
     \int_{\G_{\e,1}}R_i^0\wt \nu^{N+1}-R^0_{N+1}\wt \nu^i\,{\rm d}\G
    &=\frac{1}{(N-1)|\s^N|}\int_{\bB_1(x^0)\setminus\bB_\e(x^0)} \p_{\xi_i}|z^0-z_\xi|^{-(N-1)}\,{\rm d}\xi\\
    &=\frac{1}{(N-1)|\s^N|}\bigg(\int_{\p\bB_1(x^0)}\frac{\xi_i-x^0_i}{|\xi-x^0|}|z^0-z_\xi|^{-(N-1)}\,{\rm d}S(\xi)\\
    &\hspace{2.9cm}-\int_{\p\bB_\e(x^0)}\frac{\xi_i-x^0_i}{|\xi-x^0|}|z^0-z_\xi|^{-(N-1)}\,{\rm d}S(\xi)
    \bigg)
\end{align*}
for $1\leq i\leq N$. The second boundary integral
vanishes as $\e\to 0$, since
\begin{align*}
&\int_{\p\bB_\e(x^0)}\frac{\xi_i-x^0_i}{|\xi-x^0|}|z^ 0-z_\xi|^{-(N-1)}\,{\rm d}S(\xi)\\
&=\int_{\p\bB_1(0)} \frac{ \omega_i}{\Big(1+\Big|\displaystyle\int_0^1\nabla f(x^0+\tau \e{\omega})\cdot {\omega}\,{\rm d}\tau\Big|^2\Big)^{(N-1)/2}}\,{\rm d}{S(\omega)}\\
&\underset{\e\to0}\to\int_{\p\bB_1( 0)} \frac{ \omega_i}{\big(1+ | \nabla f(x^0)\cdot {  \omega}|^2\big)^{(N-1)/2}} \,{\rm d}{ S(\omega)}=0.
\end{align*}
Hence,
\be\label{rel1}
 \PV\int_{\G_1}R_i^0\nu^{N+1}-R^0_{N+1}\nu^i\,{\rm d}\G
 =\frac{1}{(N-1)|\s^N|}\int_{\p\bB_1(x^0)}\frac{\xi_i-x^0_i}{|\xi-x^0|}|z^0-z_\xi|^{-(N-1)}\,{\rm d}{S}(\xi)
\ee
and similarly
\be\label{rel2}
 \int_{\G_1}R_i \nu^{N+1}-R_{N+1}\nu^i\,{\rm d}\G=\frac{1}{(N-1)|\s^N|}\int_{\p\bB_1(x^0)}\frac{\xi_i-x^0_i}{|\xi-x^0|}|z-z_\xi|^{-(N-1)}\,{\rm d} S(\xi).
\ee
  Observing that the map  \mbox{$|\,\cdot\, -z_\xi|^{-(N-1)}$} is Lipschitz continuous on the set $\bB_{1/4}(z^0)$, uniformly in $\xi\in\p\bB_1(x^0)$,  we obtain  the  estimate \eqref{hoeld}  directly from  \eqref{rel1} and \eqref{rel2}.
\end{proof}

As  a last  preliminary result we establish the following lemma.

\begin{lemma}\label{L:W6} 
Let   $a,\, b \in{\rm BUC}^\alpha(\R^N)$ and set $\wt a:=a\circ\Xi^{-1} $ and $\wt b:=b\circ\Xi^{-1}.$   Given $z^0\in\Gamma$, 
we assign to~$z\in\0^-\cap \mathbb B_{1/4}(z^0)$    a point $z^\G\in\G_1$  such that 
\[
|z-z^\G|=\min_{\ov z\in \G_1}|z-\ov z|.
\]
 Then there is a  constant~$C>0$ such that
for all   $1\leq i\leq N+1$, $z^0\in\G$,    
and~$z\in\0^-\cap \mathbb B_{1/4}(z^0)$  with $|z-z^0|\leq 1/4$ we have
\begin{align*}
\Big|\int_{\G_1} R_i(\wt a(z^\G)-\wt a)\wt b\,{\rm d}\G-\int_{\G_1}R_i^0(\wt a(z^0)-\wt a )\wt b\,{\rm d}\G\Big|\leq  C|z-z^0|^\alpha.
\end{align*} 
\end{lemma}
\begin{proof}
We    set  $\vartheta:=|z-z^0|\leq 1/4$ and estimate
\[
\Big|\int_{\G_1}\frac{(z_i-\ov z_i)(\wt a(z^\G)-\wt a(\ov z))\wt b(\ov z)}{|z-\ov z|^{N+1}} \, {\rm d}\G(\ov z)
-\int_{\G_1}\frac{(z_i^0-\ov z_i)(\wt a(z^0)-\wt a(\ov z))\wt b(\ov z)}{|z^0-\ov z|^{N+1}}  \, {\rm d}\G(\ov z)\Big|\leq I_1+I_2+I_3,
\]
where, using the notation \eqref{notation}, we set
\begin{align*}
I_1&:=\|b\|_{\infty}\int_{\G_{2\vartheta}} \frac{|(z_i-\ov z_i)(\wt a(z^\G)-\wt a(\ov z)) |}{|z-\ov z|^{N+1}} + \frac{|(z_i^0-\ov z_i)(\wt a(z^0)-\wt a(\ov z)) |}{|z^0-\ov z|^{N+1}} \,{\rm d}\G(\ov z),\\
I_2&:=\|b\|_{ \infty}\int_{\G_{2\vartheta,1}}|\wt a(z^\G)-\wt a(\ov z)|\Big|\frac{ z_i-\ov z_i  }{|z-\ov z|^{N+1}}-\frac{ z_i^0-\ov z_i }{|z^0-\ov z|^{N+1}}\Big|\,{\rm d}\G(\ov z)  ,\\
I_3&:=\Big|(\wt a(z^\G)-\wt a(z^0))\int_{\G_{2\vartheta,1}}\frac{(z^0_i-\ov z_i)\wt b(\ov z)}{|z^0-\ov z|^{N+1}}\,{\rm d}\G(\ov z)\Big|. 
\end{align*}
We estimate these terms separately.\medskip

\noindent{\bf Estimate for $I_1$:}
 The definition of $z^\Gamma$ implies 
\be\label{ZGP}
|\ov z-z^\G|\leq |\ov z-z|+|z-z^\G|\leq 2|z-\ov z|,\qquad \ov z\in \G_1,
\ee
and therefore we get
\begin{align*}
I_1&\leq 2^N\|b\|_\infty[a]_\alpha\int_{\G_{2\vartheta}}|z^\G-\bar z|^{\alpha-N}+|z^0-\bar z|^{\alpha-N}\,{\rm d}\G(\ov z)\\
&\leq C\int_{\bB_{2\vartheta}(x^0)}|x^\G- \xi|^{\alpha-N}+|x^0-\xi|^{\alpha-N}\,{\rm d} \xi,
\end{align*}
 where $x^\G:=\Xi^{-1}(z^\G)$. 
To estimate the last integral we observe that due to
\[|x^0 -x^\G|\leq |z^0-z^\G|\leq 2\vartheta\]
we have $\bB_{2\vartheta}(x^0) \subset B_{4\vartheta}(x^\G)$, and hence
\begin{align*}
    I_1&\leq C\bigg(\int_{\bB_{4\vartheta}(x^\G)}|x^\G-\xi|^{\alpha-N}\,{\rm d}\xi+\int_{\bB_{2\vartheta}(x^0)}|x^0-\xi|^{\alpha-N}\,{\rm d}\xi\bigg)\\
    &\leq C\int_0^{4\vartheta}r^{\alpha-1}\,{\rm d}r=C\vartheta^\alpha=C|z-z^0|^\alpha.
\end{align*}

\noindent{\bf Estimate for $I_2$:} 
Let $\ov z\in\G_{2\vartheta,1}$ and $\zeta\in \ov{\bB_\vartheta(z^0)}$. 
Then $|\ov z -z^0|\geq 2\vartheta$ and hence
\[|z^0-\ov z|\leq |z^0-\zeta|+|\zeta-\ov z|\leq 2 |\zeta-\ov z|.\]
Define the smooth function $G:\ov{\bB_\vartheta(z^0)}\longrightarrow\R^{N+1}$ by
\[G(\zeta):=\frac{\zeta-\ov z}{|\zeta-\ov z|^{N+1}},\qquad\zeta \in \ov{\bB_\vartheta(z^0)}.\]
For the  derivative $\p G$ we have 
\[\| \p G(\zeta)\|_{\kL(\R^{N+1})}\leq C|\zeta-\ov z|^{-N-1}
\leq C|z^0-\ov z|^{-N-1},\qquad\zeta \in \bB_\vartheta(z^0), \] 
and therefore 
\[|G_i(z)-G_i(z^0)|\leq C|z-z^0||z^0-\ov z|^{-N-1},\qquad 1\leq i\leq N+1.\]
Further, by \eqref{ZGP},  $|\ov z -z^0|\geq 2\vartheta$, and noticing that
\[
|z-\bar z|\leq |z- z^0|+|z^0-\bar z|=\vartheta+|z^0-\bar z|\leq 2|z^0-\bar z|,
\]
we get
\[|  \wt a(z^\G)- \wt a(\ov z)|\leq [a]_\alpha|z^\G-\ov z|\leq C |z-\ov z|^\alpha\leq C |z^0-\ov z|^\alpha.\]
Combining the last two estimates we obtain
\begin{align*}
    I_2&\leq C|z-z^0|\int_{\bB_1(x^0)\setminus \bB_{2\vartheta}(x^0)}
    |z^0-z_\xi|^{\alpha-1-N}\,{\rm d}\xi\\
    &\leq C|z-z^0|\int_{\bB_1(x^0)\setminus \bB_{2\vartheta}(x^0)}
    |x^0-\xi|^{\alpha-1-N}\,{\rm d}\xi\\
    &=C|z-z^0|\int_{2\vartheta}^1r^{\alpha-2}\,{\rm d}r\leq C|z-z^0|\vartheta^{\alpha-1}=C|z-z^0|^\alpha.
\end{align*}
\noindent{\bf Estimate for $I_3$:}  We  infer from  \eqref{ZGP}  and Lemma~\ref{L:W1}  that
\begin{align*}
I_3&\leq [a]_\alpha|z^\G-z^0|^\alpha\Big|\int_{\G_{2\vartheta,1}}R_i^0\wt b\,{\rm d}\G\Big| \leq C|z^\G-z^0|^\alpha\leq C|z-z^0|^\alpha.
\end{align*}
This completes the proof.
\end{proof}

We are now in a position to establish the proof of Proposition~\ref{P:W1}.

\begin{proof}[Proof of Proposition~\ref{P:W1}]
Set
\begin{align*}
Q_i:=R_i\wt\nu^{N+1}-R_{N+1}\wt\nu^i, \qquad
Q_i^0:=R_i^0\wt\nu^{N+1}-R_{N+1}^0\wt\nu^i,\qquad
\ov Q_i:=0
\end{align*}
 for $1\leq i\leq N$, and
\begin{align*}
Q_{N+1}:=R\cdot\wt\nu,\qquad
Q_{N+1}^0:=R^0\cdot\wt\nu,\qquad
\ov Q_{N+1}:=1/2.
\end{align*}
Then, by Lemma \ref{L:W4} and  Lemma~\ref{L:W5}, 
\be\label{estQ1}
\bigg|\int_{\G_1}Q_i\,{\rm d}\G-\bigg(\PV\int_{\G_1} Q_i^0\,{\rm d}\G-\ov Q_i\bigg)\bigg|\leq C|z-z^0|^\alpha,\qquad 1\leq i\leq N+1,
\ee
and, by Lemma \ref{L:W6}, with the same notation as there, 
\be\label{estQ2}
\bigg|\int_{\G_1}Q_i(\wt a(z^\G)-\wt a)\,{\rm d}\G
-\int_{\G_1}Q_i^0(\wt a(z^0)-\wt a)\,{\rm d}\G\bigg|\leq C|z-z^0|^\alpha
\ee
 for  $1\leq i\leq N+1$ and each $a \in{\rm BUC}^\alpha(\R^N)$, with  $\wt a:=a\circ\Xi^{-1} $. 

For $1\leq i, j\leq N$  we further set 
\[\gamma_{ij}:=\sqrt{1+|\nabla f|^2}(\delta_{ij}-\nu^i\nu^j),\quad\gamma_{i,N+1}:=\nu^i,\quad\gamma_{N+1,j} :=-\nu^j,\quad\gamma_{N+1,N+1}:=\nu^{N+1}.\]
Then   $\gamma_{ij}\in {\rm BUC}^\alpha(\R^N)$ for all $ 1\leq i,j\leq  N+1$. 
Set $\wt\gamma_{ij}:=\gamma_{ij}\circ\Xi^{-1}$.

In view of the identity
\[A_i=\sum_{j=1}^N\gamma_{ij}(A_j \nu^{N+1}-A_{N+1} \nu^j)+\gamma_{i,N+1}(A\cdot\nu),\qquad 1\leq i\leq N+1,\quad A\in\R^{N+1},\]
we have by \eqref{claim2}, \eqref{estQ1}, \eqref{estQ2}, and the estimate $|z^\G-z^0|\leq 2|z-z^0|$ (with summation over $1\leq j\leq N+1$)
\begin{align*}
    &\bigg|\int_{\G_1}R_{i}\wt\beta\,{\rm d}\G-\bigg(\PV\int_{\G_1}R^0_{i}\wt\beta\,{\rm d}\G-\frac{\wt\nu^i\wt\beta}{2}(z^0)\bigg)\bigg|\\
    &=\bigg|\int_{\G_1}\wt\gamma_{ij}Q_j\wt\beta\,{\rm d}\G-\bigg(\PV\int_{\G_1}\wt\gamma_{ij}Q_j^0\wt\beta\,{\rm d}\G-\frac{\wt\nu^i\wt\beta }{2}(z^0)\bigg)\bigg|\\
    &\leq\bigg|\int_{\G_1}(\wt\gamma_{ij}\wt\beta-(\wt\gamma_{ij}\wt\beta)(z^\G))Q_j\,{\rm d}\G-\int_{\G_1}(\wt\gamma_{ij}\wt\beta-(\wt\gamma_{ij}\wt\beta)(z^0)) Q_j^0\,{\rm d}\G\bigg|\\
    &\quad+\bigg|\big((\wt\gamma_{ij}\wt\beta)(z^\G)-
    (\wt\gamma_{ij}\wt\beta)(z^0)\big) \PV\int_{\G_1} Q_j^0\,{\rm d}\G\bigg|\\
    &\quad+\bigg|(\wt\gamma_{ij}\wt\beta) (z^\Gamma)\bigg(\int_{\G_1}Q_j\,{\rm d}\G-\bigg(\PV\int_{\G_1} Q_j^0\,{\rm d}\G-\ov Q_j\bigg)\bigg)\bigg|\\
    &\quad+\bigg|\frac{(\wt\nu^i\wt\beta)(z^0)}{2}-\ov Q_j(\wt\gamma_{ij}\wt\beta)(z^\Gamma)\bigg|\leq C|z-z^0|^\alpha,
\end{align*}
as $\ov Q_j\gamma_{ij}=\nu^i/2.$
 The desired estimate \eqref{jumpvelo} for $z\in\0^-$ follows from this and Lemma~\ref{L:W2}.
The estimate \eqref{jumpvelo} for $z\in\0^+$  may be obtained in an analogous way.
Since the function~$V_i$, $1\leq i\leq N+1$, is smooth in $\0^\pm$, it is a straightforward consequence of~\eqref{jumpvelo} that~${V_i^\pm=V_i|_{\0^\pm}}$ can be continuously extended to $\ov{\0^\pm}$.
\end{proof}

We conclude this section with the proof of Proposition~\ref{P:W2}.
\begin{proof}[Proof of Proposition~\ref{P:W2}]
We prove only the claim for $V_i^-$, $1\leq i\leq N+1$, and divide the proof into two steps.\medskip

\noindent{\bf Step 1:} We prove that $V_i^-(z^0)\to 0$ for $z^0=(x^0,f(x^0))\in\G$ with $|z^0|\to\infty$. 
Taking into account that~$\beta$ vanishes at infinity, it remains to show that  the function $F_i:\G\to\R$ with 
\begin{equation*}
F_i(z^0):=\PV\int_{\G}R_i^0\wt \beta\,{\rm d}\G
\end{equation*}
   has this property.
Let thus $\e>0$ and choose $M>0$ such that
\begin{equation}\label{defM}
\frac{1}{|\s^N|}\|\wh \beta\|_p\bigg(\int_{\{|\xi|>M\}}\frac{1}{|\xi|^{Np'}}\,{\rm d}\xi\bigg)^{1/p'}<\frac{\e}{2},
\end{equation}
where $p'\in(1,\infty)$ is the adjoint exponent to $p$ and $\wh\beta:=\sqrt{1+|\nabla f|^2}\beta.$
After a change of variables  and using H\"older's inequality we get
\begin{equation*}
\begin{aligned}
|F_i(z^0)|<\frac{\e}{2}+ T(z^0),
\end{aligned}
\end{equation*}
where 
\begin{equation*}
T(z^0):=\frac{1}{|\s^N|}\bigg|\PV\int_{\{|\xi|<M\}}\frac{(\xi,f(x^0)-f(x^0-\xi))_i}{\big(|\xi|^2+(f(x^0)-f(x^0-\xi))^2\big)^{(N+1)/2}}\wh \beta(x^0-\xi)\,{\rm d}\xi\bigg|.
\end{equation*}
Changing variables we get
\begin{align*}
2T(z^0)\leq \int_{\{|\xi|<M\}}K(\xi)\,{\rm d}\xi,
\end{align*}
where
\begin{align*}
K(\xi)&:=\frac{1}{|\s^N|}\bigg|\frac{(\xi,f(x^0)-f(x^0-\xi))_i}{\big(|\xi|^2+(f(x^0)-f(x^0-\xi))^2\big)^{(N+1)/2}}\wh\beta(x^0-\xi)\\
&\hspace{1.5cm}+ \frac{(-\xi,f(x^0)-f(x^0+\xi))_i}{\big(|\xi|^2+(f(x^0)-f(x^0+\xi))^2\big)^{(N+1)/2}}\wh\beta(x^0+\xi)\bigg|.
\end{align*}
Since $\beta,\, \nabla f\in{\rm BUC}^\alpha(\R^N)$, the mean value theorem enables us to conclude that there exists a constant $C>0$ such that 
\[
K(\xi)\leq C\bigg(\frac{|\wh\beta(x^0+\xi)|}{|\xi|^{N-\alpha}}+\frac{|\wh \beta(x^0-\xi)-\wh\beta(x^0+\xi)|^{1/2}}{|\xi|^{N-\alpha/2}}\bigg).
\]
 Hence, since $\beta$ vanishes at infinity, there exists  $\eta>0$ such that for all $|x^0|> \eta$ we have
\begin{align*}
2T(z^0)\leq\int_{\{|\xi|<M\}}K(\xi)\,{\rm d}\xi<\e.
\end{align*}
 So, $|F_i(z^0)|<\e$ for all $|x^0|>\eta$, and this proves our first claim.  \medskip

\noindent{\bf Step 2:} We prove that $V_i^-(z)\to 0$ for $z\in\0^-$ as $|z|\to\infty$.  Let thus $\e>0$.
As shown in the previous step, there exists a constant $\eta_1>0$ such that for all $x^0\in\R^N$ with $|x^0|> \eta_1 $ we have~${|V_i^-(x^0,f(x^0))|<\e/2.}$

We next 
set
\[
d_0:=\min\Big\{\frac{1}{4}, \Big(\frac{\e}{2C}\Big)^{1/\alpha}\Big\},
\]
 with $C>0$ the constant from \eqref{jumpvelo}.
Let $z=(x,y)\in\0^-$ be arbitrary such that~$|z|> \eta$, where~$ \eta>0$ satisfies  
\begin{align}
& \eta\geq \max\Big\{  4(1+\|\nabla f\|_\infty)M, \, \frac{1}{2}+2\eta_1,\, \frac{1}{2}+2|f(0)|+2(1+\|\nabla f\|_\infty) \eta_1\Big\}, \label{delta1}\\
&\|\wh\beta\|_{L_\infty(\{|\xi|>\eta/(4(1+\|\nabla f\|_\infty))-M\})}<\frac{\e d_0^N|\s^N| }{2|\bB_M(0)|},\label{delta2}\\
&|f(0)|+ \|\nabla f\|_\infty\bigg(\frac{ \eta}{4(1+\|\nabla f\|_\infty)} +M\bigg)\leq  \frac{\eta}{4},\label{delta3}\\
& \frac{4^N\|\wh\beta\|_\infty|\bB_M(0)|}{\eta^N|\s^N|} <\frac{\e}{2},\label{delta4}
\end{align}
with $\wh\beta:=\sqrt{1+|\nabla f|^2}\beta.$
When estimating $V_i^-(z)$ we distinguish two cases.\medskip

\noindent{\bf Case 1:} We first  assume ${\rm dist} (z,\G)\leq d_0.$

Let  $z^\G=(x^\G,f(x^\G))\in\G$ be chosen such that $|z-z^\G|={\rm dist} (z,\G)$.
Since $|z|>\eta,$ we have that~${|x|>\eta/2}$ or $|y|> \eta/2.$
We  show that in both situations $|x^\G|> \eta_1$.
 Indeed, if  $|x|> \eta/2$ then \eqref{delta1} and the choice of $d_0$ imply 
\[
|x^\G|\geq |x|-|x-x^\G| > \frac{\eta}{2} -|z-z^\G|\geq \frac{\eta}{2} -d_0\geq\frac{\eta}{2} -\frac{1}{4}\geq  \eta_1.
\]
 If $|y|>\eta/2$ then
 \begin{align*}
\frac{1}{4}+|f(0)|+(1+\|\nabla f\|_\infty) \eta_1
& \leq \frac{\eta}{2}<|y| \leq|y-f(x^\G)|+|f(x^\G)-f(0)|+|f(0)|\\
&\leq |z-z^\G|+|f(0)|+\|\nabla f\|_\infty|x^\G|\\
&\leq \frac{1}{4}+|f(0)|+(1+\|\nabla f\|_\infty)|x^\G|,
 \end{align*}
hence again  $|x^\G|>\eta_1$. 
Consequently, we have that $|V_i^-(z^\G)|<\e/2$. 
Proposition~\ref{P:W1}  together with the definition of $d_0$ now yields
\[
|V_i^-(z)|\leq|V_i^-(z^\G)|+|V_i^-(z)-V_i^-(z^\G)|<\frac{\e}{2}+C|z-z^\G|^\alpha\leq\e. 
\]
\noindent{\bf Case 2:} If ${\rm dist} (z,\G)> d_0$, then, using H\"older's inequality 
and \eqref{defM}, we get
\[
|V_i^-(z)|<\frac{\e}{2}+\frac{1}{|\s^N|}\int_{\{|\xi|<M\}}\frac{|\wh \beta(x-\xi)|}{\big(|\xi|^2+(y-f(x-\xi))^2\big)^{N/2}}\,{\rm d}\xi=:\frac{\e}{2}+T(z).
\]
 We distinguish the cases $|x|>\eta/ (4(1+\|\nabla f\|_\infty)$ and $|x|\leq\eta/ (4(1+\|\nabla f\|_\infty)$.

 If $|x|>\eta/(4(1+\|\nabla f\|_\infty)$, we estimate in view of \eqref{delta2}
 \begin{align*}
 T(z)\leq \frac{ \|\wh\beta\|_{L_\infty(\{|\xi|>\eta/(4(1+\|\nabla f\|_\infty))-M\})}|\bB_M(0)| }{d_0^N|\s^N|}<\frac{\e}{2}.
 \end{align*}
 If $|x|\leq\eta/(4(1+\|\nabla f\|_\infty)\leq \eta/2$,  then  $|y|> \eta/2$    and \eqref{delta3} implies for $|\xi|\leq M$
\[
|f(x- \xi)|\leq |f(0)|+\|\nabla f \|_\infty|x-\xi|\leq |f(0)|+ \|\nabla f\|_\infty\bigg(\frac{ \eta}{4(1+\|\nabla f\|_\infty)} +M\bigg)\leq  \frac{\eta}{4},
\] 
hence $|y-f(x-\xi)|\geq  \eta/4.$
The latter estimate together with  \eqref{delta4} leads us to
\begin{align*}
T(z)\leq \frac{4^N\|\wh \beta\|_{ \infty}|\bB_M(0)|}{\eta ^N|\s^N|} <\frac{\e}{2},
\end{align*}
  and this completes the proof.
  
  \end{proof}

  %%%%%%%%%%%%%%%%%%%%%%%%%%%%%%%%%%%%%%%%%%%%%%%%%%
%%%%%%%%%%%%%%%%%%%%%%%%%%%%%%%%%%%%%%%%%%%%%%%%%%
%%%%%%%%%%%%%%%%%%%%%%%%%%%%%%%%%%%%%%%%%%%%%%%%%%
%%%%%%%%%%%%%%%%%%%%%%%%%%%%%%%%%%%%%%%%%%%%%%%%%%
 %%%%%%%%%%%%%%%%%%%%%%%%%%%%%%%%%%%%%%%%%%%%%%%%%%
%%%%%%%%%%%%%%%%%%%%%%%%%%%%%%%%%%%%%%%%%%%%%%%%%%
%%%%%%%%%%%%%%%%%%%%%%%%%%%%%%%%%%%%%%%%%%%%%%%%%%
%%%%%%%%%%%%%%%%%%%%%%%%%%%%%%%%%%%%%%%%%%%%%%%%%%
   \section{An interpolation estimate for multilinear maps}\label{Sec:B}  
  %%%%%%%%%%%%%%%%%%%%%%%%%%%%%%%%%%%%%%%%%%%%%%%%%%
%%%%%%%%%%%%%%%%%%%%%%%%%%%%%%%%%%%%%%%%%%%%%%%%%%
%%%%%%%%%%%%%%%%%%%%%%%%%%%%%%%%%%%%%%%%%%%%%%%%%%
%%%%%%%%%%%%%%%%%%%%%%%%%%%%%%%%%%%%%%%%%%%%%%%%%%
 %%%%%%%%%%%%%%%%%%%%%%%%%%%%%%%%%%%%%%%%%%%%%%%%%%
%%%%%%%%%%%%%%%%%%%%%%%%%%%%%%%%%%%%%%%%%%%%%%%%%%
%%%%%%%%%%%%%%%%%%%%%%%%%%%%%%%%%%%%%%%%%%%%%%%%%%
%%%%%%%%%%%%%%%%%%%%%%%%%%%%%%%%%%%%%%%%%%%%%%%%%%

  Let $X_0$, $X_1$, and $Y$ be  Banach spaces with continuous and dense embedding~${X_1\hookrightarrow X_0}$.
   Let~$[\cdot,\cdot]_\theta$ denote the complex interpolation functor  and set $X_\theta:=[X_0,X_1]_\theta$, $\theta\in (0,1)$. 
   Additionally, we define $[X_0,X_1]_0:=X_0$,  $[X_0,X_1]_1:=X_1$. 
  In the sequel we will use the reiteration theorem
\begin{equation}\label{reit}
  [X_{\theta_0},X_{\theta_1}]_\tau= X_{(1-\tau)\theta_0+\tau\theta_1}, \qquad \theta_0,\,\theta_1,\,\tau\in[0,1];
  \end{equation}
  see e.g. \cite[Section I.2.8]{Am95}. The following multilinear interpolation result  is a convenient tool in our analysis.
  \begin{lemma}\label{le:multint}
  Let $1\leq n\in\N$, $\vartheta\in(0,1]$, $K\geq 0$, and $T\in \kL^{n+1}(X_1,Y)$ such that
  \begin{align*}
  \|T[x_0,\ldots,x_n]\|_Y\leq K\min\big\{\|x_0\|_{X_0}\|x_1\|_{X_1}, \,\|x_0\|_{X_\vartheta}\|x_1\|_{X_{1-\vartheta}} \}\prod_{i=2}^n\|x_i\|_{X_1},\quad x_0,\ldots,x_n\in X_1,
  \end{align*}
and assume that $T$ is symmetric in the arguments $x_1,\ldots,x_n$.
  Let further $\theta_0\in[1-\vartheta,1]$ and~${\theta_1,\ldots\theta_n\in[0,1]}$ satisfy $\theta_0+\ldots+\theta_n=1$.
 We then have
  \[\|T[x_0,\ldots,x_n]\|_Y\leq K\prod_{i=0}^n\|x_i\|_{X_{1-\theta_i}},\qquad x_0,\ldots,x_n\in X_1.\]
  \end{lemma}
  \begin{proof} The proof is by induction over $n$. \medskip
  
 \noindent{\bf Step 1.} To show the result for $n=1$, fix $ \vartheta\in(0,1],\,K\geq 0$, $T\in\kL^2(X_1,Y)$, and $\theta_0,\theta_1$ according to the assumptions. Then $T$ extends to bounded 
  bilinear maps 
  \begin{align*}
 T:&\;X_0\times X_1\to Y,\\
 T:&\;X_\vartheta\times X_{1-\vartheta}\to Y
 \end{align*}
 with corresponding estimates. 
 Set $\tau:=(1-\theta_0)/\vartheta$. Then $\tau\in[0,1]$ and by the multilinear interpolation  result~\cite[Theorem 4.4.1]{BL76},
  the operator~$T$ also extends to a bounded  bilinear map
 ~$T: [X_0,X_\vartheta]_\tau\times[X_1,X_{1-\vartheta}]_\tau\to Y$ and
 \[
\|T\|_{\kL([X_0,X_\vartheta]_\tau\times[X_1,X_{1-\vartheta}]_\tau, Y)} \leq K.
 \]
 The result for $n=1$ follows as
 \[[X_0,X_\vartheta]_\tau=X_{1-\theta_0} \qquad\text{ and}\qquad [X_1,X_{1- \vartheta}]_\tau=X_{1-\theta_1}\]
 by \eqref{reit}. \medskip
 
  \noindent{\bf Step 2.} Assume that, for some arbitrary $n\geq1$, the result is true for 
  all~$\wt n\in\{1,\ldots, n\}$ (and all Banach spaces $X_0,X_1,Y$ satisfying the assumptions). 
 In order to establish the result for~${n+1}$, fix $\vartheta\in(0,1],K\geq0$, $T\in\kL^{n+2}(X_1,Y)$, and $\theta_0,\ldots ,\theta_{n+1}$ according to the assumptions. We can assume $\theta_0<1$ without loss of generality.\medskip

 \noindent{\bf Step 2a.} Define $\wt Y:=\kL^n(X_1,Y)$ and $\wt T\in\kL^2(X_1,\wt Y)$ by
 \[\wt T[x_0,x_1][z_1,\ldots,z_n]:=T[x_0,x_1,z_1,\ldots,z_n],\qquad x_0,x_1,z_1,\ldots,z_n\in X_1.\]
 Then
 \begin{align*}
 \|\wt T[x_0,x_1]\|_{\wt Y}&\leq K\|x_0\|_{X_0}\|x_1\|_{X_1},\\
 \|\wt T[x_0,x_1]\|_{\wt Y}&\leq K\|x_0\|_{X_\vartheta}\|x_1\|_{X_{1- \vartheta}}.
 \end{align*}
 Set further
 \[(\wt X_0,\wt X_1,\wt\vartheta,\wt K):=(X_0,X_1,\vartheta,K),\quad \wt n:=1,\quad\wt \theta_0:=\theta_0,\quad\wt \theta_1:=1-\theta_0.\]
 Application of the induction assumption to $\wt T$ with the variables and spaces denoted with a tilde yields
 \[\|\wt T[x_0,x_1]\|_{\wt Y}\leq K\|x_0\|_{X_{1- \theta_0}}\|x_1\|_{X_{1-\wt \theta_1}},\]
hence
 \begin{equation}\label{Tintest1}
 \|T[x_0,\ldots,x_{n+1}]\|_Y\leq K\|x_0\|_{X_{1-\theta_0}}\|x_1\|_{X_{1-\wt\theta_1}}\|x_2\|_{X_1}\prod_{i=3}^{n+1}\|x_i\|_{X_1}
   \end{equation}
   for $ x_0,\ldots,x_{n+1}\in X_1,$ and by symmetry of $T$ we get
\begin{equation}\label{Tintest2}
\|T[x_0,\ldots,x_{n+1}]\|_Y\leq K\|x_0\|_{X_{1-\theta_0}}\|x_1\|_{X_1}\|x_2\|_{X_{1-\wt \theta_1}}\prod_{i=3}^{n+1}\|x_i\|_{X_1}.
  \end{equation}
 
 \noindent{\bf Step 2b.} Define now $\wh X_0:=X_{1-\wt \theta_1}=X_{\theta_0}$, $\wh Y:=\kL(X_{1-\theta_0},Y)$, and 
 $\wh T \in\kL^{n+1}(X_1,\wh Y)$ by 
 \[\wh T [z_0,\ldots,z_n][x_0]=T[x_0,z_0,\ldots,z_n],\qquad x_0\in \wh X_0,\quad z_0,\ldots,  z_n\in X_1.\]
 Then, by \eqref{Tintest1}  and \eqref{Tintest2}, 
 \begin{align*}
 \|\wh T [z_0,\ldots,z_n]\|_{\wh Y }&\leq  K\min\big\{\|z_0\|_{\wh X_0}\|z_1\|_{X_1},\, \|z_0\|_{X_1}\|z_1\|_{\wh X_0 }\big\}\prod_{i=2}^n\|z_i\|_{X_1}
 \end{align*}
 and $\wh T$ is also symmetric in $z_0,\ldots, z_n.$
 Set further 
 \[\wh X_1:=X_1,\quad\wh\vartheta:=1,\quad\wh n:=n,\quad\wh \theta_i:=\theta_{i+1}/(1-\theta_0),\;0\leq i\leq n.\]
 Note that $\wh \theta_i\in[0,1]=[1-\wh\vartheta,1]$ and $\wh \theta_0+\ldots+\wh \theta_{n+1}=1$. 
 Application of the induction assumption to $\wh T$ with the variables and spaces denoted with a hat yields
 \[\|\wh T[z_0,\ldots,z_n]\|_{\wh Y}\leq K\prod_{i=0}^n\|z_i\|_{[\wh X_0,X_1]_{1-\wh \theta_i}}.\]
 Now, by \eqref{reit},
 \[[\wh X_0,X_1]_{1-\wh \theta_i}=[X_{\theta_0},X_1]_{1-\theta_{i+1}/(1-\theta_0)}=X_{1-\theta_{i+1}},\qquad 0\leq i\leq n,\]
 and the statement follows.
 \end{proof}

%%%%%%%%%%%%%%%%%%%%%%%%%%%%%%%%%%%%%%%%%%%%%%%%%%
%%%%%%%%%%%%%%%%%%%%%%%%%%%%%%%%%%%%%%%%%%%%%%%%%%
%%%%%%%%%%%%%%%%%%%%%%%%%%%%%%%%%%%%%%%%%%%%%%%%%%
%%%%%%%%%%%%%%%%%%%%%%%%%%%%%%%%%%%%%%%%%%%%%%%%%%
 %%%%%%%%%%%%%%%%%%%%%%%%%%%%%%%%%%%%%%%%%%%%%%%%%%
%%%%%%%%%%%%%%%%%%%%%%%%%%%%%%%%%%%%%%%%%%%%%%%%%%
%%%%%%%%%%%%%%%%%%%%%%%%%%%%%%%%%%%%%%%%%%%%%%%%%%
%%%%%%%%%%%%%%%%%%%%%%%%%%%%%%%%%%%%%%%%%%%%%%%%%%
 \section{Mapping properties for a family of generalized Riesz transforms}\label{Sec:C}
%%%%%%%%%%%%%%%%%%%%%%%%%%%%%%%%%%%%%%%%%%%%%%%%%%
%%%%%%%%%%%%%%%%%%%%%%%%%%%%%%%%%%%%%%%%%%%%%%%%%%
%%%%%%%%%%%%%%%%%%%%%%%%%%%%%%%%%%%%%%%%%%%%%%%%%%
%%%%%%%%%%%%%%%%%%%%%%%%%%%%%%%%%%%%%%%%%%%%%%%%%%
 %%%%%%%%%%%%%%%%%%%%%%%%%%%%%%%%%%%%%%%%%%%%%%%%%%
%%%%%%%%%%%%%%%%%%%%%%%%%%%%%%%%%%%%%%%%%%%%%%%%%%
%%%%%%%%%%%%%%%%%%%%%%%%%%%%%%%%%%%%%%%%%%%%%%%%%%
%%%%%%%%%%%%%%%%%%%%%%%%%%%%%%%%%%%%%%%%%%%%%%%%%%

The main aim of this appendix is to show the following result, implying in particular 
Lemma~\ref{L:MP1}. We assume \eqref{eq:s} and recall the definition \eqref{operatorsB} of the generalized Riesz transforms $B^\phi_{n,\nu}$ introduced in Section~\ref{Sec:3}.

\begin{lemma}\label{Bestgen}
Let  $M>0$, $p, n\in\N$, $\phi\in {\rm C}^\infty([0,\infty)^p)$,   and $\nu\in\N^N$ with $p\geq 1$ and   $n+|\nu|$ odd. Let further 
\[
\sigma,\,\sigma_0,\,\ldots,\sigma_n\in[0,s-1]\qquad \text{with} \qquad \sigma_0+\ldots+\sigma_n\leq \sigma.
\]

\hspace{3mm} Then there exists a constant  $C>0$   such  that  for all  
~$a\in H^s(\R^N)^p$  with~${\|a\|_{H^s}\leq M}$, ${\beta\in H^{s-1-\sigma_0}(\R^N)}$, and 
$b_i\in H^{s-\sigma_i}(\R^N)$, $1\leq i\leq n,$
we have~$B_{n,\nu}^\phi(a)[b,\beta]\in H^{s-1-\sigma}(\R^N)$ and
\be\label{eq:Bestgen}
\|B_{n,\nu}^\phi(a)[b,\beta]\|_{H^{s-1-\sigma}}\leq C\|\beta\|_{H^{s-1-\sigma_0}}
\prod_{i=1}^n\|b_i\|_{H^{s-\sigma_i}}.
\ee
\end{lemma}

We will prove this result in three steps, first assuming $\sigma=s-1$, i.e. starting with $L_2$- estimates for the generalized Riesz potentials; see~Lemma~\ref{BestL2},  then for~${\alpha:=s-1-\sigma\in(0,1)}$; see Lemma~\ref{BestHa}, and finally in the general case. In the second and third step, we use shift equivariance of the generalized Riesz potentials and corresponding difference quotients to reduce the estimates in higher norms to the basic case.

\subsection*{Estimates in $L_2(\R^N)$}
We start by proving Lemma~\ref{L:MP0}, which implies the statement of
Lemma~\ref{Bestgen} for $\sigma=\sigma_0=s-1$ and $\sigma_1=\ldots=\sigma_n=0$. The proof uses the method of rotations from harmonic analysis, cf. \cite[Theorem 9.10]{CM97}.

  \begin{proof}[Proof of Lemma~\ref{L:MP0}]
  It suffices to establish the estimate \eqref{E:MP0} under the  additional  assumption~$\|\nabla b_i\|_\infty\leq 1$ for~${1\leq i\leq n}$.
  Given $x, \xi\in\R^N $ with $x\neq \xi$ we set
  \[
K(x,\xi):=\frac{1}{|\s^N|} \phi\left((D_{[x,x-\xi]}a)^{\ov 2}\right)
\bigg[\prod_{i=1}^nD_{[x,x-\xi]}b_i\bigg]\frac{(x-\xi)^\nu}{|x-\xi|^{|\nu|}}\frac{1}{|x-\xi|^N}.
  \]
 Then $B^\phi_{n,\nu}:=B^\phi_{n,\nu}(a)[b,\cdot]$ satisfies
  \[
B^\phi_{n,\nu}[\beta](x)= \PV\int_{\R^N}K(x,\xi)\beta(\xi)\, {\rm d}\xi.
  \]
  Observe that 
  \[K(x,\xi)=F\left(\frac{A(x)-A(\xi)}{|x-\xi|}\right)|x-\xi|^{-N},\]
  where $A:\R^N\to \R^{p+n+|\nu|}$ and $F:\R^{p+n+|\nu|}\to\R$ are given by
  \begin{align*}
  A(x)&=\begin{pmatrix}
      a(x)\\
      b(x)\\
      (\underbrace{x_1,
      \ldots,x_1}_{\nu_1\text{\ components}},\ldots,\underbrace{x_N,
      \ldots,x_N}_{\nu_N\text{\ components}})
  \end{pmatrix},\qquad x\in\R^N,
  \end{align*}
   and
  \begin{align*}
  F(z_1,z_2,z_3)&=\frac{1}{|\s^N|}\phi(z_1^{\overline 2})
 \bigg(\prod_{j=1}^n z_{2,j} \bigg)\prod_{j=1}^{|\nu|}z_{3,j},\qquad z_1\in\R^p,\;z_2\in\R^n,\;z_3\in\R^{|\nu|}.
  \end{align*}
  The map $A$ is Lipschitz  continuous because  $a$ and $b$ are Lipschitz continuous, and $F$ is smooth and odd because~${n+|\nu|}$ is odd. By \cite[Theorem 9.11]{CM97}, this implies 
  $B^\phi_{n,\nu}\in\kL(L_2(\R^N))$, and the proof of this theorem also shows that 
  \[\|B^\phi_{n,\nu}\|_{\kL(L_2(\R^N))}\leq C,\]
  with $C$ depending on $\|\nabla A\|_\infty$ only.  
  As in our application $\|\nabla A\|_\infty\leq C(\phi,N,n,\nu,M)$, 
  this implies \eqref{E:MP0}.
  \end{proof}

 The next result implies Lemma~\ref{Bestgen} in the case $\sigma=\sigma_1=s-1$ and $\sigma_0=\sigma_2=\ldots=\sigma_n=0$. 
 In addition, it provides an auxiliary estimate that compares $B_{n,\nu}^\phi$ with a suitable pointwise multiplication operator. 
 The assumptions on $a$ are slightly more general, allowing also for unbounded Lipschitz functions $a$. 

  \begin{lemma}\label{L:MP1a}
  Let $M>0$, $p, n\in\N$, $\phi\in {\rm C}^\infty([0,\infty)^p)$,  $s'\in (\max\{s_c,s-1\},s)$,
  and~$\nu\in\N^N$ with~$p\geq 1$ and $n+|\nu|$ odd.
Then there exists a constant  $C>0$ such that for all 
 $a\in {\rm C}^1(\R^N)^p$ with~$\|\nabla a\|_{{\rm BUC}^{s-s_c}}\leq M$, 
$b_1\in H^1(\R^N),$ $   b_2,\ldots,b_n\in H^s(\R^N),$ and~${\beta\in H^{s-1}(\R^N)}$
we have
\begin{equation}\label{es:MPI}
\|B^\phi_{n,\nu}(a)[b,\beta]\|_2\leq C\|\beta\|_{H^{s-1}}\|b_1\|_{H^1}\prod_{i=2}^n\|b_i\|_{H^s}
\end{equation}
 and
\begin{equation}\label{es:MPIb}
\begin{aligned}
\Big\|B^\phi_{n,\nu}(a)[b,\beta]-\beta \sum_{j=1}^N B^\phi_{n-1,\nu+e_j}(a)[b_2,\ldots,b_n,\p_j b_1]\Big\|_2\leq C\|\beta\|_{H^{s-1}}\|b_1\|_{H^{s'-s+1}}\prod_{i=2}^n\|b_i\|_{H^s}.
\end{aligned}
\end{equation}
\end{lemma}
\begin{proof} 
It suffices to establish \eqref{es:MPI}  and~\eqref{es:MPIb} for $b_1\in{\rm C}_0^\infty(\R^N)$  under the additional assumption that  $s<s_c+1$. 
Indeed, otherwise we fix
$\theta\in(s-s_c-1,s'-s_c)$,
define $\tilde s:=s-\theta$ and~$\tilde s':=s'-\theta$, so that $\tilde s\in(s_c,s_c+1)$, $\tilde s'\in(\max\{s_c,\tilde s-1\},\tilde s)$, and observe that the estimates~\eqref{es:MPI} and~\eqref{es:MPIb} with $(s,s')$ replaced by $(\tilde s,\tilde s')$ imply the original ones.

Observe first that 
\[|S^N|B^\phi_{n,\nu}(a)[b,\beta]=F_1-F_2,\]
where, given $x\in\R^N$,
\begin{align*}
F_1(x)&:=\beta(x) \PV\int_{\R^N}\phi\left( (D_{[x,\xi]}a )^{\ov 2}\right)\bigg[\prod_{i=2}^n D_{[x,\xi]} b_i\bigg]\frac{\xi^\nu}{|\xi|^{|\nu|}}
\frac{\delta_{[x,\xi]} b_1}{|\xi|^{N+1}}\,{\rm d}\xi,\\
F_2(x)&:=\int_{\R^N}\phi\left( (D_{[x,\xi]}a )^{\ov 2}\right)\bigg[\prod_{i=1}^n D_{[x,\xi]} b_i\bigg]\frac{\xi^\nu}{|\xi|^{|\nu|}}
\frac{\delta_{[x,\xi]} \beta}{|\xi|^N}\,{\rm d}\xi.
\end{align*}
The map $\xi\mapsto  \xi^\nu/|\xi|^{|\nu|}$ is homogeneous of order $0$, therefore its gradient has no radial component and
\[{\rm div}_\xi \Big(\frac{\xi^\nu}{|\xi|^{|\nu|}}\frac{\xi}{|\xi|^N}\Big)
=\frac{\xi^\nu}{|\xi|^{|\nu|}}{\rm div}_\xi\Big(\frac{\xi}{|\xi|^N}\Big)
+\nabla_\xi\Big(\frac{\xi^\nu}{|\xi|^{|\nu|}}\Big)\cdot\frac{\xi}{|\xi|^N}=0, \qquad\xi\neq 0.\]
Using this relation and the identity
\begin{equation}\label{gradD}
\frac{\xi}{|\xi|}\cdot\nabla_\xi (D_{[x,\xi]}  g)=\frac{\xi}{|\xi|^2}\cdot\nabla g(x-\xi)-\frac{\delta_{[x,\xi]} g}{|\xi|^2},\qquad \text{$\xi\neq 0$,}
\end{equation}
  with $g=b_1$, we find via integration by parts that
\[F_3(x):=\beta(x) \PV\int_{\R^N}\nabla_\xi\bigg[\phi\left( (D_{[x,\xi]}a )^{\ov 2}\right)\prod_{i=2}^nD_{[x,\xi]} b_i\bigg]
\cdot \frac{\xi}{|\xi|^N}\frac{\xi^\nu}{|\xi|^{|\nu|}}D_{[x,\xi]} b_1\,{\rm d}\xi=F_1(x)-F_4(x),\]
for $x\in\R^N$, where
\[F_4:=|S^N|\beta\sum_{j=1}^N B^\phi_{n-1,\nu+e_j}(a)[b_2,\ldots,b_n,\p_j b_1].\]
Hence 
\[|S^N|B^\phi_{n,\nu}(a)[b,\beta]=-F_2+F_3+F_4,\]
and to obtain \eqref{es:MPI} and \eqref{es:MPIb}  it remains  to show
\begin{align*}
    \|F_k\|_2&\leq C\|\beta\|_{H^{s-1}}\|b_1\|_{H^{s'-s+1}}\prod_{i=2}^n\|b_i\|_{H^s},\quad k=2,\,3,\\
    \|F_4\|_2&\leq C\|\beta\|_{H^{s-1}}
    \|b_1\|_{H^1}\prod_{i=2}^n\|b_i\|_{H^s}.
\end{align*}

The estimate for $F_4$ is immediate from Lemma \ref{L:MP0} and $\|\beta\|_\infty\leq C\|\beta\|_{H^{s-1}}$. To show the estimates for $F_2$ and $F_3$, observe first that for any 
$ g\in {\rm C}^{1}(\R^N)$ with $\nabla g\in {\rm BUC}^{s-s_c}( \R^N)$ we have
\begin{equation}\label{gradDest}
\left|\frac{\xi}{|\xi|}\cdot\nabla_\xi (D_{[x,\xi]} g)\right|
=\frac{|\delta_{[x,\xi]}g-\xi\cdot\nabla g(x-\xi)|}{|\xi|^2}
\leq\frac{[\nabla g]_{s-1-N/2}}{|\xi|^{2-s+N/2}},\qquad \ \xi\neq 0,
\end{equation}
where  $[\cdot]_\alpha$, $\alpha\in(0,1)$, is the usual H\"older seminorm.
Since $H^{s-1}(\R^N)\hookrightarrow {\rm BUC}^{s-s_c}(\R^N)$, our assumptions ensure that the estimate \eqref{gradDest} is satisfied for $g=a_k$ with $1\leq k\leq p$  and for~$g=b_j$ with $2\leq j\leq n$.

As a further preparation, we consider the weakly singular integral $I$ given by
\[I(x):=\int_{\R^N}\frac{|\delta_{[x,\xi]}b_1|}{|\xi|^{2-s+3N/2}}\,{\rm d}\xi,\qquad x\in\R^N,\]
and  show that
 \begin{equation}\label{estI}
\|I\|_2\leq C\|b_1\|_{H^{s'-s+1}}.
\end{equation}
Indeed, by Minkowski's inequality and Plancherel's theorem 
\begin{align*}
\|I\|_2&=\left(\int_{\R^N}\left(\int_{\R^N}\frac{|\delta_{[x,\xi]}b_1|}{|\xi|^{2-s+3N/2}}\,{\rm d}\xi\right)^2\,{\rm d}x\right)^{1/2}
\leq \int_{\R^N}\left(\int_{\R^N}\left(\frac{\delta_{[x,\xi]}b_1}{|\xi|^{2-s+3N/2}}\right)^2\,{\rm d}x\right)^{1/2}\,{\rm d}\xi\\
&=\int_{\R^N}\frac{1}{|\xi|^{2-s+3N/2}}\left(\int_{\R^N}|\mathcal{F} b_1|^2(\eta)
|e^{i\xi\cdot\eta}-1|^2  \, { \rm d}\eta\right)^{1/2}\,{\rm d}\xi,
\end{align*}
 with $\mathcal{F}$ denoting the Fourier transform.
Estimating
\begin{alignat*}{2}
&|e^{i\xi\cdot\eta}-1|\leq 2&&\text{ for $|\xi|>1$,}\\
&|e^{i\xi\cdot\eta}-1|\leq  C |\xi\cdot\eta|^{s'-s+1}\leq C|\xi|^{s'-s+1} |\eta|^{s'-s+1}\ &&\text{ for $|\xi|<1$,}
\end{alignat*}
  with some fixed $C>0$, we obtain 
\[\|I\|_2\leq 2\|b_1\|_2\int_{\{|\xi|> 1\}}\frac{1}{|\xi|^{2-s+3N/2}}  \, {\rm d}\xi + \|b_1\|_{H^{s'-s+1}} \int_{\{|\xi|<1\}}\frac{1}{|\xi|^{1-s'+3N/2}}  \, {\rm d}\xi.
\]
Both integrals converge since $s<2+N/2$ and $1+N/2<s'$, hence \eqref{estI} is proved.

To estimate $F_2$ we observe  for $x\in \R^N$  and $\xi\in\R^N\setminus\{0\}$ that 
\begin{align*}
|\delta_{[x,\xi]}\beta|&\leq[\beta]_{s-1-N/2}|\xi|^{s-1-N/2}
\leq C \|\beta\|_{H^{s-1}}|\xi|^{s-1-N/2},\\
|D_{[x,\xi]}b_i|&\leq \|\nabla b_i\|_\infty \leq C \|b_i\|_{H^s},\qquad  2\leq i\leq n,\\
\left|\phi\left((D_{[x,\xi]}a)^{\ov 2}\right)\right|&\leq C,
\end{align*}
so that by \eqref{estI} 
\[\|F_2\|_2\leq C\|\beta\|_{H^{s-1}}\|I\|_2\prod_{i=2}^n\|b_i\|_{H^s}
\leq C \|b_1\|_{H^{s'-s+1}}\prod_{i=2}^n\|b_i\|_{H^s}\|\beta\|_{H^{s-1}}.\]

To estimate $F_3$ we carry out the differentiation in the integrand which yields terms of the form
\begin{align*}
F_5(x)&:=\beta(x)\int_{\R^N}\p_k\phi\left((D_{[x,\xi]}a)^{\ov 2}\right)\big(D_{[x,\xi]}a_k\big)
\frac{\xi}{|\xi|}\cdot \nabla_\xi\big(D_{[x,\xi]}a_k\big)
\bigg[\prod_{i=2}^nD_{[x,\xi]} b_i\bigg]\frac{\xi^\nu}{|\xi|^{|\nu|}}
\frac{\delta_{[x,\xi]}b_1}{|\xi|^N}\,{\rm d}\xi,\\
F_6(x)&:=\beta(x)\int_{\R^N}\phi\left((D_{[x,\xi]}a)^{\ov 2}\right)
\bigg[\prod_{i=2,\,i\neq j}^n D_{[x,\xi]}b_i \bigg]
\frac{\xi}{|\xi|}\cdot \nabla_\xi\big(D_{[x,\xi]}b_j\big)
\frac{\xi^\nu}{|\xi|^{|\nu|}}
\frac{\delta_{[x,\xi]}b_1}{|\xi|^N}\,{\rm d}\xi
\end{align*}
with $2\leq j\leq n$ and $1\leq k\leq p$.
To estimate these, we proceed as for~$F_2$, using additionally~\eqref{gradDest}, the  estimates
\[\|\beta\|_\infty\leq C\|\beta\|_{H^{s-1}},\quad [\nabla a_k]_{s-1-N/2}\leq M,\quad
[\nabla b_j]_{s-1-N/2}\leq C\|b_j\|_{H^s},\]
as well as the boundedness  (uniformly  for fixed $M$) of the terms $\p_k\phi\big((D_{[x,\xi]}a)^{\ov 2}\big)$ and
$D_{[x,\xi]}a_k$.  This completes the proof.
\end{proof}

We next obtain more flexible $L_2$-estimates via interpolation, proving Lemma~\ref{Bestgen} in the case $\sigma=s-1$:

\begin{lemma}\label{BestL2}
Under the assumptions of Lemma~{\rm \ref{Bestgen}} with $\sigma=s-1$, there exists a 
constant~$ C>0$ such that for all 
$a\in {\rm C}^1(\R^N)^p$ with~${\|\nabla a\|_{{\rm BUC}^{s-s_c}\leq M}}$,  $\beta\in H^{s-1-\sigma_0}(\R^N)$, and $b_i\in H^{s-\sigma_i}(\R^N)$,  $1\leq i\leq n$, we have
\[\|B^\phi_{n,\nu}(a)[b,\beta]\|_2\leq C\|\beta\|_{H^{s-1-\sigma_0}}\prod_{i=1}^n\|b_i\|_{H^{s-\sigma_i}}.\]
\end{lemma}
\begin{proof}
    We apply Lemma \ref{le:multint} with $X_0:=Y:=L_2(\R^N)$, $ X_1:=H^s(\R^N)$, and
    \[ \vartheta:=(s-1)/s,\qquad\theta_0:=(1+\sigma_0)/s,\qquad\theta_i:=\sigma_i/s,\quad  1\leq i\leq n\] 
    to the operator~$T:=B^\phi_{n,\nu}(a)$, observing that the assumptions of this lemma are satisfied due to Lemma~\ref{L:MP0} and Lemma~\ref{L:MP1a}; see estimate~\eqref{es:MPI}. The claim follows  since, up to norm equivalence, we have~${X_{1-\theta_0}=H^{s-1-\sigma_0}(\R^N)}$ and~$X_{1-\theta_i}=H^{s-\sigma_i}(\R^N)$,~$1\leq i\leq n$.
\end{proof}

\subsection*{Estimates in $H^\alpha(\R^N)$, $\alpha\in (0,1)$} 
We recall \eqref{equivHs} and \eqref{semiHs} and provide the following preparatory result for reference (see \cite[Lemma 7]{AM22} and \cite[Lemma 2.3]{MM21} for the case $N=1$):
  
\begin{lemma}\label{L:prep} Let  $r\geq0$ and $\alpha\in(0,1)$. 
Then there exists a constant~${C>0}$ such that for all~$\beta\in H^{r+\alpha}(\R^N)$ we have
\begin{equation}\label{deed}
\int_{\R^N}\frac{\|\tau_\zeta\beta-\beta\|_{H^r}^2}{|\zeta|^{N+2\alpha}}\,{\rm d}\zeta\leq C\|\beta\|_{H^{r+\alpha}}^2.
\end{equation}
\end{lemma}
\begin{proof} 
 Let $L^r$ be the Fourier multiplier on $\R^N$ with symbol $z\mapsto (1+|z|^2)^{r/2}$. For all~${\rho\geq r}$, this operator is an isomorphism from $H^\rho(\R^N)$ to $H^{\rho-r}(\R^N)$, and it commutes with translations. Therefore, for $\beta\in H^{r+\alpha}(\R^N)$,
\[\int_{\R^N}\frac{\|\tau_\zeta\beta-\beta\|_{H^r}^2}{|\zeta|^{N+2\alpha}}\,{\rm d}\zeta
\leq C\int_{\R^N}\frac{\|\tau_\zeta L^r\beta-L^r\beta\|_2^2}{|\zeta|^{N+2\alpha}}\,{\rm d}\zeta
=C[L^r\beta]_{H^\alpha}^2\leq C\|\beta\|_{H^{r+\alpha}}^2.\]
\end{proof}

 It is straightforward to verify, under the general assumptions of Lemma~\ref{L:MP0}, the identity
\be\label{shiftid}
\tau_\zeta B^\phi_{n,\nu}(a)[b,\beta]
=B^\phi_{n,\nu}(\tau_\zeta a)[\tau_\zeta b,\tau_\zeta\beta],\qquad \beta\in L_2(\R^N),\quad \zeta\in\R^N.
\ee
Hence, by \eqref{difference}, we have
\begin{equation}  \label{shiftdiff}
\begin{aligned}
    (\tau_\zeta-1) B^\phi_{n,\nu}(a)[b,\beta]&=B^\phi_{n,\nu}(\tau_\zeta a)[\tau_\zeta b,\tau_\zeta\beta-\beta]\\
    &\quad\,+\sum_{j=1}^nB^\phi_{n,\nu}(\tau_\zeta a)[b_1,\ldots,  b_{j-1},
    \tau_\zeta b_j-b_j,\tau_\zeta  b_{j+1},\ldots,\tau_\zeta  b_n,\beta]\\
    &\quad\,+\sum_{i=1}^p B^{\phi^i}_{n+2,\nu}(\tau_\zeta a,a)[\tau_\zeta a-a,\tau_\zeta a+a,b,\beta]
\end{aligned}
\end{equation}
with $\phi^i$,  $1\leq i\leq p$, given by \eqref{gis}.    
 Lemma~\ref{L:prep} and~\eqref{shiftdiff} are used in the proof of Lemma~\ref{Bestgen} in the case~$s-1-\sigma\in(0,1)$, which is provided below.

\begin{lemma}\label{BestHa} 
Under the assumptions of Lemma {\rm\ref{Bestgen}} with $ s-1-\sigma=:\alpha\in(0,1)$, there exists a constant $C>0$ such that  for all  
$a\in H^s(\R^N)^p$  with~$\|a\|_{H^s}\leq M$,  $\beta\in H^{s-1-\sigma_0}(\R^N)$,  and~$b_i\in H^{s-\sigma_i}(\R^N)$, $1\leq i\leq n,$ we have
\[\|B^\phi_{n,\nu}(a)[b,\beta]\|_{H^\alpha}\leq C\|\beta\|_{H^{s-1-\sigma_0}}\prod_{i=1}^n\|b_i\|_{H^{s-\sigma_i}}.\]
\end{lemma}
\begin{proof}
    In view of \eqref{equivHs}, we have to show
    \begin{equation}\label{ccxxx0}
    \|B^\phi_{n,\nu}(a)[b,\beta]\|_2^2+\int_{\R^N}\frac{\|(\tau_\zeta-1)B^\phi_{n,\nu}(a)[b,\beta]\|_2^2}{|\zeta|^{N+2\alpha}}\,{\rm d}\zeta
    \leq C\|\beta\|_{H^{s-1-\sigma_0}}^2\prod_{i=1}^n\|b_i\|_{H^{s-\sigma_i}}^2.
    \end{equation}
    For the first term, the estimate is immediate from Lemma~\ref{BestL2}.  
    For the second term in~\eqref{ccxxx0}, we recall~\eqref{shiftdiff} and estimate the terms in this representation separately.
  More precisely,    Lemma~\ref{BestL2} with $\sigma_i$ replaced by $\tilde \sigma_i$, where
    \[\tilde \sigma_0:=\sigma_0+\alpha\qquad\text{and}\qquad\tilde \sigma_i:=\sigma_i,\quad 1\leq i\leq n,\]
    yields
    \begin{equation}\label{ccxxx1}
    \|B^\phi_{n,\nu}(\tau_\zeta a)[\tau_\zeta b,\tau_\zeta\beta-\beta]\|_2^2
    \leq C \|\tau_\zeta\beta-\beta\|^2_{H^{s-1-\sigma_0-\alpha}}
    \prod_{i=1}^n\|b_i\|_{H^{s-\sigma_i}}^2.
      \end{equation}

    Analogously,   for $1\leq j\leq n$, we obtain in view of Lemma ~\ref{BestL2} that
    \begin{equation}\label{ccxxx2}
    \begin{aligned}
        &\|B^\phi_{n,\nu}(\tau_\zeta a) [b_1,\ldots,  b_{j-1},
    \tau_\zeta b_j-b_j,\tau_\zeta  b_{j+1},\ldots,\tau_\zeta  b_n,\beta]\|_2^2\\
        &\leq  C \|\beta\|_{H^{s-1-\sigma_0}}^2
        \|\tau_\zeta b_j-b_j\|^2_{H^{s-\sigma_j-\alpha}}
        \prod_{i=1, i\neq j}^n\|b_i\|^2_{H^{s-\sigma_i}},
    \end{aligned}
      \end{equation}
      and, for $1\leq i\leq p$,
      \begin{equation}\label{ccxxx3}
    \begin{aligned}
    & \|B^{\phi^i}_{n+2,\nu}(\tau_\zeta a,a)[\tau_\zeta a-a,\tau_\zeta a+a,b,\beta]\|_2^2\\
    &\leq  C \|\beta\|_{H^{s-1-\sigma_0}}^2\|\tau_\zeta a-a\|^2_{H^{s-\alpha}}
    \|\tau_\zeta a+a\|^2_{H^s}\prod_{i=1}^n\|b_i\|_{H^{s-\sigma_i}}^2\\
    &\leq C \|\beta\|_{H^{s-1-\sigma_0}}^2\|\tau_\zeta a-a\|^2_{H^{s-\alpha}}
    \prod_{i=1}^n\|b_i\|_{H^{s-\sigma_i}}^2.
    \end{aligned}
      \end{equation}
     Using \eqref{ccxxx1}-\eqref{ccxxx3} and applying Lemma~\ref{L:prep}, 
     we find that the integral term in \eqref{ccxxx0} can be estimated by the right side of \eqref{ccxxx0}, and the proof is complete.
\end{proof}
\subsection*{Estimates in higher norms}
As a preparation for the proof of Lemma~\ref{Bestgen} in generality, we introduce the divided difference operators on $H^r(\R^N)$, $r\geq 0$, by
\[D_\eps^jf:=\frac{\tau^j_\eps f-f}{\eps},\qquad f\in H^r(\R^N),\quad\eps\in\R\setminus\{0\},\quad 1\leq j\leq N,\]
 where  $\tau_\eps^j:=\tau_{\eps e_j}$; see~\eqref{semiHs}.
We recall that $f\in H^{r+1}(\R^N)$ iff $\lim_{\eps\to 0} D^j_\eps f$ exists in $H^r(\R^N)$ for all $1\leq j\leq N$. In this case, $\lim_{\eps\to 0} D^j_\eps f=\p_jf$.

  As a straightforward consequence of \eqref{shiftdiff}, we obtain for  $\e\neq0$ and $1\leq j\leq N$ the representation
\begin{equation}\label{dqrep}
\begin{aligned}
    D_\eps^j B^\phi_{n,\nu}(a)[b,\beta]&=B^\phi_{n,\nu}(\tau^j_\eps a)[\tau_\eps^j b,D_\eps^j\beta]\\
    &\quad\,+\sum_{i=1}^nB^\phi_{n,\nu}(\tau_\eps^j a)[ b_1,\ldots, b_{i-1},
    D_\eps^j b_i,\tau_\eps^j b_{i+1},\ldots,\tau_\eps^j b_n,\beta] \\
    &\quad\,+\sum_{i=1}^p B^{\phi^i}_{n+2,\nu}(\tau_\eps^j a,a)[D_\eps^j a,\tau_\eps^j a+a,b,\beta].
\end{aligned}
\end{equation}

\begin{proof}[Proof of Lemma~\ref{Bestgen}] 
For $k\in\N$ with $k\leq   s-1$, let (H)$_k$ be the following statement:
\[
\left.\text{\begin{minipage}{0.85\textwidth}
For any  $M>0$, $p,n\in\N$ with $p\geq 1$, $\phi\in {\rm C}^\infty([0,\infty)^p)$, $\nu\in\N^N$ with~~$n+|\nu|$ odd,  and 
$\sigma,\sigma_0,\ldots,\sigma_n\in[0,s-1]$ such that 
\[ \sigma_0+\ldots+\sigma_n\leq \sigma,\quad s-1-\sigma-k=:\alpha\in[0,1)\]
there is a constant  $C>0$ such that for all $a\in H^s(\R^N)^p$  with~$\|a\|_{H^s}\leq M$, 
$\beta\in H^{s-1-\sigma_0}(\R^N),$ and~$ b_i\in H^{s-\sigma_i}(\R^N)$, $1\leq i\leq n$,
the function~$B_{n,\nu}^\phi(a)[b,\beta]$ belongs to $ H^{s-1-\sigma}(\R^N)$ and
\[
\|B_{n,\nu}^\phi(a)[b,\beta]\|_{H^{k+\alpha}}=\|B_{n,\nu}^\phi(a)[b,\beta]\|_{H^{s-1-\sigma}}\leq C\|\beta\|_{H^{s-1-\sigma_0}}
\prod_{i=1}^n\|b_i\|_{H^{s-\sigma_i}}.
\]
\end{minipage}}\;\right\}\eqno{\textnormal{(H)$_k$}}
\]

We prove Lemma~\ref{Bestgen} by showing (H)$_k$ for all  $k\leq  s-1$. 
We proceed by induction over~$k$. 

Statement (H)$_0$ holds by Lemma \ref{BestL2} and Lemma~\ref{BestHa}.

Assume  now  (H)$_k$ for  some $k\leq  s-2$. From \eqref{difference} we conclude (for any $p$,  $n$, $\phi$, $\nu$, $\sigma$, $\sigma_i$ satisfying the assumptions of (H)$_k$) that the mapping $[a\mapsto B_{n,\nu}^\phi(a)]$ belongs to 
\begin{align}
& {\rm C}^{1-}\big(H^s(\R^N)^p,
\kL^n(H^{s-\sigma_1}(\R^N),\ldots,H^{s-\sigma_n}(\R^N),\kL(H^{s-1-\sigma_0}(\R^N), H^{s-1-\sigma}(\R^N)))\big).
\label{eq:acont}
\end{align}

Fix now  $M$, $p$, $n$, $\phi$,  $\nu$, $\sigma$, $\sigma_i$ according to the assumptions of (H)$_{k+1}$. In view of~\eqref{dqrep} and the remarks about  the divided difference operators $D^j_\eps$ it suffices to show for~$1\leq j\leq N$, $a\in H^s(\R^N)^p$  with~$\|a\|_{H^s}\leq M$, 
$ b_i\in H^{s-\sigma_i}(\R^N)$, $1\leq i\leq n$, and~$\beta\in H^{s-1-\sigma_0}(\R^N)$  
that, as~$\e\to 0$,  we have the following convergences in $H^{s-2-\sigma}(\R^N)$:
\begin{align}
  &B^\phi_{n,\nu}(\tau^j_\eps a)[\tau_\eps^j b,D_\eps^j\beta]\to B^\phi_{n,\nu}(a)[b,\p_j\beta] \label{conv1}\\
  &B^\phi_{n,\nu}(\tau_\eps^j a) [ b_1,\ldots, b_{i-1},
    D_\eps^j b_i,\tau_\eps^j b_{i+1},\ldots,\tau_\eps^j b_n,\beta] \nonumber\\
  &\to B^\phi_{n,\nu}(a)[b_1,\ldots,b_{i-1}, \p_j b_i,b_{i+1},\ldots,b_n,\beta],\qquad 1\leq i\leq n,
 \label{conv2} \\
  &B^{\phi^i}_{n+2,\nu}(\tau_\eps^j a,a)[D_\eps^j a,\tau_\eps^j a+a,b,\beta]\to
    2B^{\phi^i}_{n+2,\nu}(a,a)[\p_j a,a,b,\beta],\qquad 1\leq i\leq p,\label{conv3}
\end{align}
and that there is a constant $C>0$ such that, uniformly in $a$, $\beta$, and  $ b_i$, $1\leq i\leq n$,
\begin{equation}\label{extconv}
    \begin{aligned}
    &\|B^\phi_{n,\nu}(a)[b,\p_j\beta]\|_{H^{s-2-\sigma}}   +\sum_{i=1}^n\|B^\phi_{n,\nu}(a)[b_1,\ldots,b_{i-1}, \p_j b_i,b_{i+1},\ldots,b_n,\beta]\|_{H^{s-2-\sigma}} \\
    &+\|B^\phi_{n,\nu}(a)[b,\beta]\|_{H^{s-2-\sigma}} +\sum_{i=1}^p  \|B^{\p_i\phi}_{n+2,\nu}(a)[\p_j a,a,b,\beta]\|_{H^{s-2-\sigma}}\leq C\|\beta\|_{H^{s-1-\sigma_0}}\prod_{i=1}^n\|b_i\|_{H^{s-\sigma_i}},
    \end{aligned}
\end{equation}
as $ B^{\phi^i}_{n+2,\nu}(a,a)=B^{\p_i\phi}_{ n+2,\nu}(a)$, $1\leq i\leq p$, by  \eqref{gis}.

  To start, we note that (H)$_{k}$ with $\sigma$ replaced $\tilde \sigma=\sigma+1$ (and $\sigma_i$, $0\leq i\leq n$, unchanged)  ensures that 
$B^\phi_{n,\nu}(a)[b,\beta]\in H^{s-2-\sigma}(\R^N)$  can be estimated according to \eqref{extconv}.

To show the convergence~\eqref{conv1} and the corresponding estimate for the limit function in~\eqref{extconv}, we  note  that for~${1\leq j\leq N}$ 
\be\label{epsconv}
\left.\begin{array}{lcl}
    \tau_\eps^j a\to a&&\text{in $H^{s}(\R^N)^p$,}\\
    \tau_\eps^j b_i\to b_i&&\text{in $H^{s-\sigma_i}(\R^N)$,}\quad 1\leq i\leq n,\\
    D^j_\eps \beta\to \p_j \beta &&\text{in $H^{s-2-\sigma_0}(\R^N)$.}
    \end{array}\right\}
\ee
 The induction assumption (H)$_k$ with $\sigma$ replaced by $\tilde\sigma=\sigma+1$ and $\sigma_0$ replaced by~${\tilde\sigma_0=\sigma_0+1}$ (and all other variables unchanged) then immediately provides \eqref{conv1}; see~\eqref{eq:acont}, together with the desired estimate for the limit in \eqref{extconv}.

The proofs for \eqref{conv2} and \eqref{conv3}  and the estimates for the corresponding limits in \eqref{extconv} are similar. More precisely, for \eqref{conv2} we use for each $1\leq i\leq n$, the assumption~(H)$_k$  with~$\sigma$ replaced by $\tilde\sigma=\sigma+1$  and $\sigma_i$ replaced by $\tilde \sigma_i:=\sigma_i+1$ (and the other variables unchanged), while  for~\eqref{conv3} we use for each~$1\leq i\leq p$ the assumption (H)$_k$ with the variables replaced by
\[
\wt p:=2p,\quad\wt\phi:=\phi^i,  \quad \wt M:=\sqrt{2}M,\quad \wt n:=n+2,\quad\tilde\sigma_0=\sigma_0,\]
\[
 \; \wt \sigma_1:=1,\quad\wt \sigma_2:=0,\quad \wt \sigma_{i+2}:=\sigma_i, \;1\leq i\leq n,
\quad \wt \sigma:=\sigma+1.\]
\end{proof}

\begin{proof}[Proof of Lemma \ref{L:MP1b}]
The proof of Lemma \ref{L:MP1b}  is contained in the proof of Lemma~\ref{Bestgen}.
\end{proof}

 We conclude this section with the proof of Lemma~\ref{L:MP1d}.
\begin{proof}[Proof of Lemma~\ref{L:MP1d}] It suffices to show that the map $[ a\mapsto B^\phi_{n,\nu}(a)]$ is Fr\'echet differentiable with 
 \begin{equation}\label{derB}
    \partial B^\phi_{n,\nu}(a)[h][b,\cdot]=2B^{\phi'}_{n+2,\nu}(a)[b,a,h,\cdot].
  \end{equation}
  for $a,\, h\in H^s(\R^N)$ and $b=(b_1,\ldots, b_n)\in H^s(\R^N)^n$.
  The infinite differentiability result will follow then from an induction argument. 
  
  In order to establish \eqref{derB} we infer from \eqref{difference} and the fundamental theorem of calculus that  for each~$\beta\in H^{s-1}(\R^N)$ we have 
 \begin{align*}
&\big(B^\phi_{n,\nu}(a+h)-B^\phi_{n,\nu}(a)\big)[b,\beta]-2B^{\phi'}_{n+2,\nu}(a)[b,a,h,\beta]\\
&=B^{\wt \phi}_{n+2,\nu}(a+h,a) [b, 2a+h,h,\beta]-2B^{\phi'}_{n+2,\nu}(a)[b,a,h,\beta]\\
&=B^{\wt \phi}_{n+2,\nu}(a+h,a) [b, h^{[2]},\beta]+2\big(B^{\wt \phi}_{n+2,\nu}(a+h,a)  -2B^{\phi'}_{n+2,\nu}(a)\big)[b,a,h,\beta]\\
&=B^{\wt \phi}_{n+2,\nu}(a+h,a) [b, h^{[2]},\beta]+2B^{\wh \phi}_{n+4,\nu}(a+h,a) [b,a,2a+h, h^{[2]},\beta],
 \end{align*}
 where $\wt\phi,\, \wh\phi \in {\rm C}^\infty( [0,\infty)^2)$ are given by
  \[
\wt\phi(x,y)=\int_0^1\phi'(sx+(1-s)y)\,{\rm d}s\qquad \text{and}\qquad \wh\phi(x,y)=\int_0^1\int_0^1s\phi''(\tau sx+(1-\tau s)y)\,{\rm d}s\,{\rm d}\tau
  \]
for $(x,y)\in [0,\infty)^2$.  Hence, for $\|h\|_{H^s}\leq 1$, we infer from Lemma~\ref{L:MP1} that 
 \begin{align*}
\big\|\big(B^\phi_{n,\nu}(a+h)-B^\phi_{n,\nu}(a)\big)[b,\beta]-2B^{\phi'}_{n+2,\nu}(a)[b,a,h,\beta]\big\|_{H^{s-1}}
 \leq C\|\beta\|_{H^{s-1}}  \|h\|_{H^s}^2\prod_{i=1}^n \|b_i\|_{H^{s}},
 \end{align*}
 and the claim follows.
\end{proof}

%%%%%%%%%%%%%%%%%%%%%%%%%%%%%%%%%%%%%%%%%%%%%%%%%%
%%%%%%%%%%%%%%%%%%%%%%%%%%%%%%%%%%%%%%%%%%%%%%%%%%
%%%%%%%%%%%%%%%%%%%%%%%%%%%%%%%%%%%%%%%%%%%%%%%%%%
%%%%%%%%%%%%%%%%%%%%%%%%%%%%%%%%%%%%%%%%%%%%%%%%%%
 \section{ Localization of singular integral operators}\label{Sec:D}
%%%%%%%%%%%%%%%%%%%%%%%%%%%%%%%%%%%%%%%%%%%%%%%%%%
%%%%%%%%%%%%%%%%%%%%%%%%%%%%%%%%%%%%%%%%%%%%%%%%%%
%%%%%%%%%%%%%%%%%%%%%%%%%%%%%%%%%%%%%%%%%%%%%%%%%%
%%%%%%%%%%%%%%%%%%%%%%%%%%%%%%%%%%%%%%%%%%%%%%%%%%
The primary objective of this appendix is to localize the singular integral  operators~\( \sfB^{\phi}_{n,\nu}(f) \), introduced in \eqref{trueRO}.  

The central result established in Proposition~\ref{T:LOC} is a crucial tool in the analysis carried out in Section~\ref{Sec:5}.
 We first derive several commutator estimates for the operators~\( \sfB^{\phi}_{n,\nu}(f) \), cf.  Lemma~\ref{commphi} and Lemma~\ref{L:Mix},  
which  are essential for the localization result presented in Proposition~\ref{P:Local}.
We then investigate the operators \( D^{\phi,A}_{n,\nu} \), defined in \eqref{dfnb}, and prove that they are Fourier multipliers, providing suitable estimates for their symbols. 
 Finally, in Proposition~\ref{T:LOC} and Lemma~\ref{T:prod}  we prove the announced localization results.
In this  appendix, we assume again that $s$ satisfies \eqref{eq:s}.
 
 \subsection*{Commutator type properties}
 
We establish several commutator properties that are crucial in the analysis. 
We start  by  estimating  the commutator $\llbracket\varphi, \sfB^{\phi}_{n,\nu}(f)\rrbracket$ in  suitable norms.

\begin{lemma}\label{commphi} 
Let  $M>0$, $n\in\N$,   $\phi \in {\rm C}^\infty([0,\infty))$, and $\nu\in\N^N$  with $n+|\nu|$ odd.
\begin{itemize}
    \item [\rm (i)] 
    There exists a constant $C>0$  such that for all  $\varphi\in {\rm BUC}^1(\R^N)$, $\beta\in L_2(\R^N)$, and~${f\in{\rm C}^1(\R^n)}$  
    with~${\|\nabla f\|_{{\rm BUC}^{s-s_c}}\leq M}$,  we have $\llbracket\varphi, \sfB^{\phi}_{n,\nu}(f)\rrbracket[\beta]\in H^1(\R^N)$ and 
    \begin{equation}\label{comestalpha}
\| \llbracket\varphi, \sfB^{\phi}_{n,\nu}(f)\rrbracket[\beta]\|_{H^1}
\leq C \|\varphi\|_{{\rm BUC}^1}\|\beta\|_{2}.
\end{equation}
    \item[\rm (ii)] 
    Let  $s'\in(\max\{ s_c,s-1\},s)$. Then there exists a constant~${C>0}$ such that for all~${f\in H^s(\R^N)}$ with $\|f\|_{H^s}\leq M$   and~$\varphi,\beta\in H^{s-1}(\R^N)$ the function~$\llbracket\varphi, \sfB^{\phi}_{n,\nu}(f)\rrbracket[\beta]$ 
     belongs to $ H^{s-1}(\R^N)$  and 
    \begin{equation}\label{comests'}
\| \llbracket\varphi,  \sfB^{\phi}_{n,\nu}(f)\rrbracket[\beta]\|_{H^{s-1}}
\leq C\|\varphi\|_{H^{s-1}}\|\beta\|_{H^{s'-1}}.
\end{equation}
\end{itemize}
    
\end{lemma}
\begin{proof}
     It suffices to prove \eqref{comestalpha}-\eqref{comests'} for $\beta\in {\rm C}_0^\infty(\R^N)$, assuming in the case of (ii) additionally that $\varphi\in {\rm C}_0^\infty(\R^N)$.
    Observing that  Lemma \ref{L:MP0},  Lemma \ref{L:mult}, and Lemma~\ref{Bestgen} imply 
    \begin{align*}
        \|\llbracket\varphi, \sfB^{\phi}_{n,\nu}(f)\rrbracket[\beta]\|_2&\leq 
        C \|\varphi\|_\infty\|\beta\|_2,\\
        \|\llbracket\varphi, \sfB^{\phi}_{n,\nu}(f)\rrbracket[\beta]\|_{H^{s-2}}&\leq C \|\varphi\|_{H^{s-1}}\|\beta\|_{H^{s-2}},
    \end{align*}
 in view of \eqref{equivgrad}, it remains to estimate the partial derivatives~$\p_j\llbracket\varphi, B^{\phi}_{n,\nu}(f)\rrbracket[\beta]$, $1\leq j\leq N$, in~$L_2(\R^N)$ and in~$H^{s-2}(\R^N)$, respectively.
 In order to show that these derivatives actually exist, we compute, using \eqref{dqrep}, for $1\leq j\leq N$ and $\varepsilon\neq 0$, 
\begin{equation}\label{dqrep0}
\begin{aligned}
    D_\eps^j \big(\llbracket\varphi, \sfB^{\phi}_{n,\nu}(f)\rrbracket[\beta]\big)
    &=\llbracket  D_\eps^j\varphi, \sfB^{\phi}_{n,\nu}(\tau_\eps^j f)\rrbracket[\tau_\eps^j\beta]
  + \llbracket  \varphi, \sfB^{\phi}_{n,\nu}(\tau_\eps^j f)\rrbracket[D_\eps^j\beta]\\
    &\quad\,+\sum_{i=1}^n\llbracket\varphi,B^\phi_{n,\nu}(\tau_\eps^j f)
    [D_\eps^j f, f^{[i-1]},  \tau_\eps^j f^{[n-i]},\cdot]\rrbracket[\beta] \\
    &\quad\,+\llbracket\varphi,B^{\phi^1}_{n+2,\nu}(\tau_\eps^j f,f)[ D_\eps^j f,\tau_\eps^j f+f, f^{[n]},\cdot]\rrbracket[\beta],
\end{aligned}
\end{equation}
with $\phi^1$ defined in \eqref{gis}.
As for $1\leq i\leq n$   we have
\begin{align*}
    \llbracket\varphi,B^\phi_{n,\nu}(\tau_\eps^j f)[D_\eps^j f, f^{[i-1]}, \tau_\eps^j f^{[n-i]},\cdot]\rrbracket=\llbracket D_\eps^j f,B^\phi_{n,\nu}(\tau_\eps^j f)[\varphi,f^{[i-1]},   \tau_\eps^j f^{[n-i]},\cdot]\rrbracket
\end{align*}
and
\begin{align*}
    &\llbracket\varphi,B^{\phi^1}_{n+2,\nu}(\tau_\eps^j f,f)
    [ D_\eps^j f,\tau_\eps^j f+f,f^{[n]},\cdot]\rrbracket=\llbracket D_\eps^j f ,B^{\phi^1}_{n+2,\nu}(\tau_\eps^j f,f)
    [ \varphi,\tau_\eps^j f+f,f^{[n]},\cdot]\rrbracket,
\end{align*}
we may pass  to $\eps\to 0$  in \eqref{dqrep0}, using   \eqref{Lipclasses} and the convergences $D_\eps^j \varphi\to \p_j\varphi$ and~$D_\eps^j f\to \p_jf$ in $L_\infty(\R^N)$, to obtain that 
$\llbracket\varphi, \sfB^{\phi}_{n,\nu}(f)\rrbracket[\beta]\in H^1(\R^N)$ and, for $1\leq j\leq N$,
\begin{equation}\label{comprep1a}
\begin{aligned}
     \p_j( \llbracket\varphi, \sfB^{\phi}_{n,\nu}(f)\rrbracket[\beta]\big)
    =&\llbracket \p_j\varphi, \sfB^{\phi}_{n,\nu}(f)\rrbracket[\beta]
    +n\llbracket \p_j f,B^\phi_{n,\nu}(f)[\varphi,f^{[n-1]},\cdot]\rrbracket[\beta] \\
    &+\llbracket \varphi, \sfB^{\phi}_{n,\nu}(f)\rrbracket[\p_j\beta]+2\llbracket \p_jf ,B^{\phi'}_{n+2,\nu}(f,f)[ \varphi,f^{[n+1]},\cdot]\rrbracket[\beta].
\end{aligned}
\end{equation}
Further, for $x\in\R^N$ we have
 \[\llbracket\varphi, \sfB^{\phi}_{n,\nu}(f)\rrbracket[\p_j\beta] (x)
 =\frac{1}{|\s^N|}\PV\int_{\R^N}\phi\big((D_{[x,\xi]}f)^2\big)
 \big(D_{[x,\xi]}f\big)^n\frac{\xi^\nu}{|\xi|^{|\nu|}}D_{[x,\xi]}\varphi\frac{\p_j\beta(x-\xi)}{|\xi|^{N-1}}\,{\rm d}\xi,\]
 and we apply integration by parts to the PV integral to  rewrite
 \begin{equation} \label{commrep4}
 \begin{aligned}
  \llbracket\varphi, B^{\phi}_{n,\nu}(f)\rrbracket[ \p_j\beta]
 &=2B^{\phi'}_{n+2,\nu}(f)[\varphi, f^{[n+1]},\beta\p_jf]
 -2B^{\phi'}_{n+3,\nu+e_j}(f)[\varphi, f^{[n+2]},\beta]\\
 &\quad+nB^{\phi}_{n,\nu}(f)[\varphi,f^{[n-1]},\beta\p_jf ]
 -nB^{\phi}_{n+1,\nu+e_j}(f)[\varphi, f^{[n]},\beta]\\
  &\quad+\nu_jB^{\phi}_{n+1,\nu-e_j}(f)[\varphi, f^{[n]},\beta]
  -B^{\phi}_{n+1,\nu+e_j}(f)[\varphi, f^{[n]},\beta] \\
  & \quad+ \sfB^{\phi}_{n,\nu}(f)[\beta\p_j\varphi ]
  -(|\nu|+N-1)B^{\phi}_{n+1,\nu+e_j}(f)[\varphi, f^{[n]},\beta].
 \end{aligned}
 \end{equation}
 The estimate \eqref{comestalpha} is now a straightforward consequence of \eqref{comprep1a}-\eqref{commrep4}  and Lemma~\ref{L:MP0}.

 For the estimate~\eqref{comests'}  we use Lemma \ref{L:mult} and Lemma~\ref{Bestgen} to obtain
 \begin{align*}
     \|\p_j\varphi \sfB^{\phi}_{n,\nu}(f)[\beta]\|_{H^{s-2}}&\leq C\|\varphi\|_{H^{s-1}}\|B^{\phi}_{n,\nu}(f)[\beta]\|_{H^{s'-1}}\leq C \|\varphi\|_{H^{s-1}}\|\beta\|_{H^{s'-1}},\\
     \|\sfB^{\phi}_{n,\nu}(f)[\beta\p_j\varphi]\|_{H^{s-2}}&\leq C\|\beta\p_j\varphi \|_{H^{s-2}}\leq C \|\varphi\|_{H^{s-1}}\|\beta\|_{H^{s'-1}},\\
     \|\p_jfB^\phi_{n,\nu}(f)[\varphi, f^{[n-1]},\beta]\|_{H^{s-2}}&\leq C\|\varphi\|_{H^{s-1}}\|\beta\|_{H^{s'-1}}
     \leq C\|\varphi\|_{H^{s-1}}\|\beta\|_{H^{s'-1}},\\
     \|B^\phi_{n,\nu}(f)[\varphi,f^{[n-1]},\beta\p_jf]\|_{H^{s-2}}&\leq C\|\varphi\|_{H^{s-1}}\|\beta\p_jf\|_{H^{s'-1}}\leq C\|\varphi\|_{H^{s-1}}\|\beta\|_{H^{s'-1}}.
\end{align*}
 In particular, to handle the last two terms above, Lemma~\ref{Bestgen} is applied with~$s$
replaced by~$s'$ and with~$\sigma=\sigma_1:=1-(s-s')$.
 The estimates for the remaining terms in \eqref{comprep1a}-\eqref{commrep4} are analogous. 
 
 %The result for general $\beta$ follows by a standard density argument. 
\end{proof}
 
We now provide a product estimate for fractional norms which will be useful in the proof of the main localization result.
\begin{lemma}\label{L:alg}
 Let $\alpha\in(0,1],$ $\alpha'\in(0,\alpha)$,  and $r>N/2+\alpha-\alpha'$.
Then there exists a constant~$C>0$ such that 
\begin{equation*} 
\|fg\|_{H^{\alpha}}\leq C\big( \|f\|_\infty\|g\|_{H^{\alpha}}+\|f\|_{H^r}\|g\|_{H^{\alpha'}}\big),\qquad f\in H^r(\R^N),\quad g\in H^{\alpha}(\R^N).
\end{equation*}
\end{lemma} 
\begin{proof}
Let $\alpha\in(0,1)$ first.
Since $\|fg\|_2\leq \|f\|_\infty\|g\|_2$, it remains to estimate the seminorm~$[fg]_{H^{ \alpha}}$.
 Recalling the definition~\eqref{semiHs}, we infer  from Lemma \ref{L:mult}  and  Lemma \ref{L:prep} that 
\begin{align*}
 [fg]_{H^{\alpha}}^2& =\int_{\R^N}\frac{\|\tau_\zeta(fg)-fg\|_2^2}{|\zeta|^{N+2{\alpha}}}\,{\rm d\zeta}\\
& \leq C\|f\|_\infty^2\int_{\R^N}\frac{\|\tau_\zeta g-g\|_2^2}{|\zeta|^{N+2{\alpha}}}\,{\rm d\zeta}
+C\int_{\R^N}\frac{\|(\tau_\zeta f-f)g\|_2^2}{|\zeta|^{N+2{\alpha}}}\,{\rm d\zeta}\\
& \leq C\|f\|_\infty^2\|g\|_{H^{\alpha}}^2+C\|g\|_{H^{\alpha'}}^2\int_{\R^N}\frac{\|\tau_\zeta f-f\|_{H^{r-{\alpha}}}^2}{|\zeta|^{N+2{\alpha}}}\,{\rm d\zeta}\\
&\leq C \|f\|_\infty^2\|g\|_{H^{\alpha}}^2+C\|g\|_{H^{\alpha'}}^2\|f\|_{H^r}^2,
\end{align*}
which proves the claim. 

For $\alpha=1$ we have from Lemma \ref{L:mult}
\[\|fg\|_{H^1}\leq C\big(\|f\|_\infty \|g\|_{H^1}+\|g\nabla f\|_2\big)
\leq C\big(\|f\|_\infty \|g\|_{H^1}+\|\nabla f\|_{H^{r-1}} \|g\|_{H^{\alpha'}}\big),
\]
which proves the claim as well.
\end{proof}

As a further result, which we use in Section~\ref{Sec:5}, we establish the following commutator type property.

  \begin{lemma}\label{L:Mix}
Given  $M>0$,   $ s'\in(\max\{s-1, s_c\},s)$,  $\phi \in {\rm C}^\infty([0,\infty))$,  $\nu\in\N^N$,  and $n\in\N$ with $n\geq 1$ and $n+|\nu|$ odd, there exists a  constant~$C>0$ such that for  all $f\in H^s(\R^N)$ with   $\|f\|_{H^{s}}\leq M$, 
$h\in H^s(\R^N)$, and~${\beta\in H^{s-1}(\R^N)}$ we have
\begin{equation}\label{es:Mix} 
\bigg\|B^\phi_{n,\nu}(f)[ f^{[n-1]},h,\beta]-\beta\sum_{j=1}^N \sfB^\phi_{n-1,\nu+e_j}(f)[\p_j h]\bigg\|_{H^{s-1}}\leq C\|\beta\|_{H^{s-1}}\|h\|_{H^{s'}}.
\end{equation}
\end{lemma}
\begin{proof} 
We fix  the functions~$f$ and $\beta$ according to the assumptions and consider the 
linear operator~$ T\in \mathcal{L}(H^s(\R^N),H^{s-1}(\R^N))$ defined by
\[
T [h]:=B^\phi_{n,\nu}(f)[ f^{[n-1]},h,\beta]-\beta\sum_{j =1}^N \sfB^\phi_{n-1,\nu+e_j}(f)[\p_j h],\qquad h\in H^s(\R^N).\]
For $1 \leq \ell\leq N$, we use the chain rule \eqref{F:MP2a} to represent   the commutator of~$T$ with the spatial {derivative~$\p_\ell$} as
\begin{align*}
    \llbracket\p_\ell,T\rrbracket [h]&=(n-1)B_{n,\nu}^\phi(f)[\p_\ell f, f^{[n-2]},h,\beta]+2B_{n+2,\nu}^{\phi'}(f)[\p_\ell f, f^{[n]},h,\beta]\\
    &\quad\,+B_{n,\nu}^\phi(f)[ f^{[n-1]},h,\p_\ell\beta]
    -\sum_{j=1}^N\Big(
    \p_\ell\beta \sfB_{n-1,\nu+e_j}^\phi(f)[\p_jh]\\
    &\quad\,+(n-1)\beta B_{n-1,\nu+e_j}^\phi(f)[\p_\ell f, f^{[n-2]},\p_jh]
    +2\beta B^{\phi'}_{n+1,\nu+e_j}[\p_\ell f, f^{[n]},\p_jh]\Big).
\end{align*}

We fix $\theta\in[1,s-1]$ and estimate all terms on the right in the norm of $H^{\theta-1}(\R^N)$, using Lemma  \ref{L:mult} and  Lemma~\ref{Bestgen} with $s$ replaced by $s'$ and appropriate choices of $\sigma$ and $\sigma_i$. We obtain
\begin{align*}
    \|B_{n,\nu}^\phi(f)[\p_\ell f, f^{[n-2]},h,\beta]\|_{H^{\theta-1}}
    &\leq C\|\p_\ell f\|_{H^{s-1}}\|h\|_{H^{\theta+1-(s-s')}}
    \|\beta\|_{H^{s'-1}}\\
    &\leq C\|h\|_{H^{\theta+1-(s-s')}} \|\beta\|_{H^{s-1}},\\
    \|B_{n,\nu}^\phi(f)[  f^{[n-1]},h,\p_\ell\beta]\|_{H^{\theta-1}}
    &\leq C\|h\|_{H^{\theta+1-(s-s')}}\|\p_\ell\beta\|_{H^{s-2}},\\
    \|\p_\ell\beta  \sfB_{n-1,\nu+e_j}^\phi(f)[\p_jh]\|_{H^{\theta-1}}
    &\leq C\|\p_\ell\beta\|_{H^{s-2}}\| \sfB_{n-1,\nu+e_j}^\phi(f)[\p_jh]\|_{H^{\theta-(s-s')}}\\
    &\leq C\|h\|_{H^{\theta+1-(s-s')}}\|\beta\|_{H^{s-1}},\\
    \|\beta B_{n-1,\nu+e_j}^\phi(f)[\p_\ell f,f^{[n-2]},\p_jh]\|_{H^{\theta-1}}
    &\leq C\|\beta\|_{H^{s-1}}\|B_{n-1,\nu+e_j}^\phi(f)[\p_\ell f, f^{[n-2]},\p_jh]\|_{H^{\theta-1}}\\
    &\leq C\|h\|_{H^{\theta+1-(s-s')}}\|\beta\|_{H^{s-1}}.
\end{align*}
The remaining terms can be estimated analogously, and we obtain
\be\label{commdlT}
\|\llbracket\p_\ell,T\rrbracket[h]\|_{H^{\theta-1}}
\leq C\|h\|_{H^{\theta+1-(s-s')}}\|\beta\|_{H^{s-1}}.
\ee

We are going to show 
\be\label{indTh}
\|T[h]\|_{H^{k+\alpha}}\leq C\|\beta\|_{H^{s-1}}\|h\|_{H^{k+\alpha+1-(s-s')}}
\ee
for all $k\in\N$ and $\alpha\in[0,1)$ with $k+\alpha\leq s-1$,  the desired estimate \eqref{es:Mix} being the special case $k+\alpha=s-1$.

For $k=\alpha=0$, the result is given by \eqref{es:MPIb}.
To obtain it for $k=1$  and~$\alpha=0$ we use~\eqref{indTh}  with $k=\alpha=0$ together with \eqref{commdlT} with $\theta=1$ to obtain
\begin{align*}
\|T[h]\|_{H^1}&\leq C\Big(\|Th\|_{2}+\sum_{\ell=1}^N\|\p_\ell T[h]\|_2\Big)
\leq C\Big(\|T[h]\|_2+\sum_{\ell=1}^N\big(\|T[\p_\ell  h]\|_{L^2}+\|\llbracket\p_\ell,T\rrbracket [h]\|_2\big)\Big)\\
&\leq C\|\beta\|_{H^{s-1}}\|h\|_{H^{2-(s-s')}}.
\end{align*}
From this estimate and \eqref{indTh}  with $k=\alpha=0$ we obtain by interpolation that \eqref{indTh} holds with $k=0$ and $\alpha\in[0,1)$. Now the general result follows by induction over $k$. Indeed, let~$k\in\N$ and~$\alpha\in[0,1)$ be such that $k+1+\alpha\leq s-1$ and assume that \eqref{indTh} holds for this~$k$ and~$\alpha$. Then, by the induction assumption and \eqref{commdlT} with $\theta=k+\alpha+1$,  we conclude
\begin{align*}
\|T[h]\|_{H^{k+\alpha+1}}&\leq C\Big(\|T[h]\|_{H^{k+\alpha}}+\sum_{\ell=1}^N\|\p_\ell T[h]\|_{H^{k+\alpha}}\Big)\\[-0.75ex]
&\leq C\Big(\|T[h]\|_{H^{k+\alpha}}+\sum_{\ell=1}^N\big(\|T[\p_\ell  h]\|_{H^{k+\alpha}}+\|\llbracket\p_\ell,T\rrbracket [h]\|_{H^{k+\alpha}}\big)\Big)\\
&\leq C\|\beta\|_{H^{s-1}}\|h\|_{H^{k+\alpha+2-(s-s')}}.
\end{align*}
 \end{proof}
 \subsection*{The Fourier multipliers $D^{\phi,A}_{n,\nu}$} 
 
 In Proposition~\ref{T:LOC}, we prove that the operator~$\sfB^{\phi}_{n,\nu}(f)$ 
 can be locally approximated by the singular integral operator~\( D^{\phi,A}_{n,\nu} \),  
defined in \eqref{dfnb} (with suitable $A\in\R^N$).  
The properties of the latter are investigated in Proposition~\ref{P:D1}, where we show that \( D^{\phi,A}_{n,\nu} \),  assuming $\phi \in {\rm C}^\infty([0,\infty))$,    $n\in\N$, and $\nu\in\N^N$ with   $n+|\nu|$ odd, is a Fourier multiplier with a purely imaginary bounded symbol.  
Moreover, in Lemma~\ref{L:SGRT}, we provide estimates for a Fourier multiplier that appears in Proposition~\ref{P:Local} and involves certain operators \( D^{\phi,A}_{n,\nu}\).

\begin{prop}\label{P:D1}  
  The operator $D:=D^{\phi,A}_{n,\nu}$ is a Fourier multiplier with symbol
  \[
  [z\mapsto im(z)] \qquad \text{with} \qquad m\in L_\infty(\R^N,\R).
  \] 
  Moreover,  for any constant $L>0$ there is a constant $C=C(L)>0$ 
  such that $|A|\leq L$ implies $\|m\|_\infty\leq C$.
  \end{prop}

\begin{proof} Fix $L>0$ and $A\in\R^N$  with $|A|\leq L$.
We  define the kernel
\begin{equation}\label{polar}
\begin{aligned}
    K(\xi)&:=\frac{\mK(\omega)}{r^N},\qquad\xi\in\R^N\setminus\{0\},
    \quad r:=|\xi|,\quad\omega:=\xi/|\xi|\in\s^{N-1},\\
    \mK(\omega)&:=\frac{1}{|\s^N|}\phi((A\cdot\omega)^2)(A\cdot\omega)^n\,\omega^\nu,
\end{aligned}
\end{equation}
  and note that $K$ is odd because $n+|\nu|$ is odd. Moreover,   there is a constant~$C=C(L)>0$ such that
  \[|\mK(\omega)|\leq C,\qquad \omega\in\s^{N-1}.\]
For $\delta>0$, we introduce the truncated kernel
\[K_\delta:=\mathbf{1}_{\{ \delta<|\xi|\}} K\in L_2(\R^N),\]
where $\mathbf{1}_{\{ \delta<|\xi|\}}$ is the characteristic function of the set 
$\{\xi\in\R^N\,:\, \delta<|\xi|\}$, and the corresponding convolution operator~$D_\delta$ given by
\[D_\delta[\beta]:=K_\delta\ast\beta,\qquad\beta\in L_2(\R^N).\]
By the definition of PV integrals, for $\delta\to 0$ we have
\be\label{DdeltoD}
D_\delta[\beta](x)\to D [\beta](x),\qquad  x\in\R^N,\quad \beta\in {\rm C}^\infty_0(\R^N).
\ee

For $\beta\in L_2(\R^N)\cap L_1(\R^N)$, we have by the properties of the Fourier transform 
\be\label{Foumult}
\clf[D_\delta[\beta]]= (2 \pi)^{N/2}\clf [K_\delta]\cdot\clf[\beta].
\ee
We will show that, in fact,
\be\label{FKdelta}
im_\delta:=(2 \pi)^{N/2}\clf [K_\delta]\in L_\infty(\R^N),
\ee
so that \eqref{Foumult} extends by a standard density argument  to   $ L_2(\R^N)$, and hence~$D_\delta\in\mathcal{L}(L_2(\R^N))$ is the Fourier multiplier with symbol $im_\delta$.

To show $\eqref{FKdelta}$, we introduce for $\eta>\delta$
\[K_{\delta,\eta}:=\mathbf{1}_{\{\delta<|\xi|<\eta\}} K\in L_2(\R^N)\cap L_1(\R^N).\]
By dominated convergence we have
\[K_{\delta,\eta}\to K_\delta
\quad\text{in $L_2(\R^N)$   as $\eta\to \infty$}\]
and hence
\be\label{L2convsym}
i m_{\delta,\eta}:= (2 \pi)^{N/2}\clf  [K_{\delta,\eta}]\to (2 \pi)^{N/2}\clf  [K_{\delta}]=im_\delta
\qquad\text{in $ L_2(\R^N)$ as $\eta\to \infty$}.
\ee

As $K_{\delta,\eta}$ is odd, we obtain for $z\in\R^N$ by introducing polar coordinates $(r,\omega)$ (cf. \eqref{polar}) and the substitution $\tau:= r|\omega\cdot z|$
\begin{align*}
   (2 \pi)^{N/2} \clf [K_{\delta,\eta}](z)&=-i\int_{\{\delta<|\xi|<\eta\}}
    K(\xi)\sin(\xi\cdot z)\,{\rm d}\xi\\
    &    =-i\int_{\s^{N-1}}
    \Big(\int_\delta^\eta\frac{\sin(r\omega\cdot z)}{r}\,{\rm d}r\Big)
    \mK(\omega)\,{\rm d}S(\omega)\\
    &=-i\int_{\s^{N-1}}\Big(\int_{\delta |\xi\cdot z|}^{\eta |\xi\cdot z|}
    \frac{\sin\tau}{\tau}\,{\rm d}\tau\Big)\mathop{\rm sgn}(\omega\cdot z)\mK(\omega)\,{\rm d} S(\omega).
\end{align*}
As the integral over $\tau$ is bounded independently of the integration limits we find that
\[m_{\delta,\eta}:=-i(2 \pi)^{N/2}\clf  [K_{\delta,\eta}]\in L_\infty(\R^N)\cap  L_2(\R^N)\qquad\text{and}\qquad
\quad\|m_{\delta,\eta}\|_\infty\leq C,\]
with a constant  $C=C(L)>0$ independent of $\delta$ and $\eta$. By dominated convergence, 
\[m_{\delta,\eta}(z)\to
-\int_{\s^{N-1}}\Big(\int_{\delta |\xi\cdot z|}^{\infty}
\frac{\sin\tau}{\tau}\,{\rm d}\tau\Big)\mathop{\rm sgn}(\omega\cdot z)\mK(\omega)\,{\rm d} S(\omega)=:\wt m_\delta(z),\qquad z\in\R^N,\]
 for~$\eta\to \infty$.
 From \eqref{L2convsym} we have that also $m_{\delta,\eta_k}\to m_\delta$ pointwise almost everywhere (a.e) for some sequence $\eta_k\to\infty$. 
Thus, $m_\delta=\wt m_\delta$ a.e., and \eqref{FKdelta} is shown. Moreover,
\be\label{mdelunif}
\|m_\delta\|_\infty\leq C
\ee
with  $C=C(L)>0$ independent of $\delta>0$.

Further, again by dominated convergence,  we have for $z\in\R^N$,  as~$\delta\to 0$,
\[m_\delta(z)\to
-\frac{\pi}{2}\int_{\s^{N-1}}\mathop{\rm sgn}(\omega\cdot z)\mK(\omega)\,{\rm d} S(\omega)
=:m(z)\qquad\text{as}\qquad 
\int_0^\infty \frac{\sin\tau}{\tau}\,{\rm d}\tau=\pi/2.
\]
Let $\wt D\in\mathcal{L}(L_2(\R^N))$ denote the Fourier multiplier with symbol $im$. Then, for  $\beta\in{\rm C}^\infty_0(\R^N)$, dominated convergence  and \eqref{mdelunif} lead us to
\[
\|D_{\delta}[\beta]- \wt D[\beta]\|_2=\|(m_{\delta}-m)\mathcal{F}[\beta]\|_2 \to 0 \qquad\text{ for~$\delta\to \infty$}. 
\]
Therefore,  $D_{\delta_k}[\beta]\to\wt D[\beta]$ pointwise a.e. along some sequence $\delta_k\to 0$,
so that $ \wt D[\beta]= D[\beta]$ a.e. by \eqref{DdeltoD}.
This completes the proof. 
\end{proof}

 We recall the definition of $\bar\phi\in {\rm C}^\infty([0,\infty))$ from \eqref{deefphi} and establish now some conclusions on the Fourier multipliers that occur as localizations of the operators $\Psi(\tau)$, cf. \eqref{Psi}, Lemma~\ref{le:locAop}, and Proposition \ref{P:Local}. They are crucial for proving Theorem~\ref{T:GP} and Proposition~\ref{P:inv}.

\begin{lemma}\label{L:SGRT}
Let $L>0$ be given. For all $A\in\R^N$ such that $|A|\leq L$, the operator 
\[
T:=\sum_{k=1}^N D_{0,e_k}^{\bar\phi,A} \frac{\p}{\p x_k}
\]
is a Fourier multiplier with symbol  $m_T\in L_\infty(\R^N,\R)$   and there is a constant~${\eta=\eta(L)\in(0,1)}$ such that
\[\eta|z|\leq m_T(z)\leq \eta^{-1}|z|,\qquad z\in\R^N.\]
\end{lemma}
\begin{proof}
    As shown  in the proof of Proposition \ref{P:D1},  the operator $D_{0,e_k}^{\bar\phi,A}$ is a Fourier multiplier with symbol
    \[z\mapsto-\frac{i\pi}{2|\s^N|}\int_{\s^{N-1}}\omega_k\mathop{\rm sgn}(\omega\cdot z)\bar\phi((A\cdot\omega)^2)\,{\rm d}\omega,\qquad 1\leq k\leq N.\]
Since $(\mathcal{F}[\p_k h])(z)=iz_k(\mathcal{F}[h])(z)$ for $z\in\R^N$, one straightforwardly calculates
\[m_T(z)=\frac{\pi}{2|\s^N|}\int_{\s^{N-1}}|\omega\cdot z|\bar\phi((A\cdot\omega)^2)
\,{\rm d}\omega,\]
and the estimates follow immediately from $0<\bar\phi(L^2)\leq\bar\phi(|A|^2)\leq\bar\phi((A\cdot\omega)^2)\leq 1$ and
\[\int_{\s^{N-1}}|\omega\cdot z|\,{\rm d}\omega=|z|\int_{\s^{N-1}}|\omega_1|\,{\rm d}\omega.\]
\end{proof}

\subsection*{Localization results for the operators $\sfB^{\phi}_{n,\nu}(f)$. }
 We recall the  definition of the $\eps$-local\-i\-za\-tion family  from Section \ref{Sec:5} and   first provide a  localization result for the  operators  $\sfB_{n,\nu}^\phi(f)$ in lower order 
  Sobolev (semi)norms; see  Lemma~\ref{L:locHa}. It allows to control the error incurred by replacing an argument $f$ by    a linear function~$\bar f$.  More precisely,  for~$f\in H^s(\R^N)$, $\eps\in(0,1)$,  and~$0\leq j\leq m(\eps)$, we define $\bar f:=\bar f_j^\eps\in{\rm Lip}(\R^N,\R)$
by
\begin{equation}\label{barf}
\bar f(x)=
\left\{\begin{array}{ll}
    \nabla f(x_j^\eps)\cdot x,&\;1\leq j\leq m(\eps),\\
    0&\;j=0.
    \end{array}
\right.
\end{equation}

The core of our localization results is the following:

\begin{lemma}\label{L:L2loc}\ 
  Let $p, n\in\N$ with $n\geq 1$, $\phi\in {\rm C}^\infty([0,\infty)^p)$,  $f\in H^s(\R^N)$,
 $\nu\in \N^N$ such that~$n+|\nu|$ is odd, $M>0 $,  and~$\eta\in(0,1)$. Then, for all sufficiently small $\e>0$  and all   Lipschitz continuous functions~~${a:\R^N\to\R^p}$, $b:\R^N\to\R^{n-1}$ with $\|\nabla a\|_\infty\leq M$,  ${0\leq j\leq m(\e)}$, and~${\beta\in L_2(\R^N)}$ with $\supp \beta\subset \supp\chi^\e_j$ (i.e. $\beta=0$ a.e. outside $\supp\chi^\e_j$), we have
\[\|\chi^\e_j  B_{n,\nu}^\phi(a) [b,f-\bar f,\beta]\|_2\leq \eta\|\beta\|_2 \prod_{i=1}^{n-1}\|\nabla b_i\|_\infty.\]
\end{lemma}

\begin{proof}
Define for $\e\in(0,1)$  and $1\leq j\leq m(\e)$
the function 
$ F_j^\e:\R^N\to\R$ by
%\begin{equation}\label{eq:Fje}
\[ F_j^\e(x):=\left\{
\arraycolsep=1.4pt
\begin{array}{llll}
f(x)&,&x\in\overline{\bB}_{2\e}(x_j^\e),\\
f\Big(x_j^\e+2\e\frac{x-x_j^\e}{|x-x_j^\e|}\Big)+\nabla f(x_j^\e)\cdot\Big(x-x_j^\e-2\e\frac{x-x_j^\e}{|x-x_j^\e|}\Big)&,&x\not\in\overline{\bB}_{2\e}(x_j^\e).
\end{array}
\right.\]
%\end{equation}
For $j=0$, corresponding to localization near infinity, we define $F_0^\e:\R^N\to\R$ by
\[F_0^\e(x):=\left\{
\arraycolsep=1.4pt
\begin{array}{llll}
f(x)&,&|x|\geq \e^{-1}-\e,\\[1ex]
\frac{|x|}{\e^{-1}-\e}f\Big(\frac{x}{|x|}(\e^{-1}-\e)\Big) &,&|x|\leq \e^{-1}-\e.
\end{array}
\right.
\]
Then, for  $0\leq j\leq m(\e)$, $ F_j^\e$ is Lipschitz continuous, and,  recalling~\eqref{barf}, we get for $\e\to0$
 \be\label{T2fa}
 \begin{aligned}
     \|\nabla F_j^\eps-\nabla \bar f_j^\eps\|_\infty&=
     \| \nabla  f-\nabla f(x_j^\eps)\|_{ L_\infty(\bB_{2\e}(x_j^\e))}
     \leq C\eps^{ s-s_c}\to0,\qquad  1\leq j\leq m(\e),\\
     \|\nabla F_0^\eps-\nabla \bar f_0^\eps\|_\infty&=
     \|\nabla F_0^\eps\|_\infty\leq C\|f\|_{W^{1,\infty}(\{|x|\geq \eps^{-1}-\eps\})}\to0.
 \end{aligned}
 \ee
 For  $\beta\in L_2(\R^N)$ with $ \supp\beta\subset\supp\chi_j^\e$,  we observe that
\[\chi^\e_j  B_{n,\nu}^\phi(a)[b,f-\bar f,\beta]=\chi^\e_j   B_{n,\nu}^\phi(a)[b,F_j^\e-\bar f,\beta].\]
Indeed, (up to zero sets) the integrands defining both terms (including the cutoff $\chi_j^\eps$) are  nonzero only if~$x\in \supp \chi_j^\eps$ and $x-\xi\in \supp \chi_j^\eps$.
In that case, however, $\delta_{[x,\xi]}f=\delta_{[x,\xi]}F_j^\e$.
Consequently, by Lemma \ref{L:MP0},
\begin{equation*}
\|\chi^\e_j  B_{n,\nu}^\phi(a)[b,f-\bar f,\beta]|_2=\|\chi^\e_j B_{n,\nu}^\phi(a)[b,F_j^\e-\bar f,\beta]\|_2
\leq C \|\nabla F_j^\eps-\nabla \bar f_j^\eps\|_\infty\|\beta\|_2 \prod_{i=1}^{n-1}\|\nabla b_i\|_\infty,
\end{equation*}
and the result follows from \eqref{T2fa}. 
\end{proof}

 The estimate given in Lemma \ref{L:D6x} below is an intermediate result used in Lemma~\ref{L:locHa} to treat differences of the form $\chi_j^\e\big(\sfB_{n,\nu}^\phi(f)-\sfB_{n,\nu}^\phi(\bar f)\big)[\pi_j^\e\beta]$.

 In the arguments that follow we will use  the algebraic identity
\begin{equation}\label{idbaf}B_{n,\nu}^\phi(a)[b,\bar f,\beta]=\sum_{k=1}^N\p_kf(x_j^\eps)
B_{n-1,\nu+e_k}^\phi(a)[b,\beta]
\end{equation}
valid for any Lipschitz continuous functions $a:\R^N\to\R^p$ and $b:\R^N\to\R^{n-1}$ and 
~${\beta\in L_2(\R^N)}$ (again with $n+|\nu|$ being odd and  $\phi\in {\rm C}^\infty([0,\infty)^p)$).

\begin{lemma}\label{L:D6x} Let  $\alpha\in(0,1)$, $\alpha'\in\big(\max\{0,\alpha-(s- s_c)\},\alpha\big)$, $n\geq 1$, $f\in H^s(\R^N)$, $\nu\in \N^N$ such that $n+|\nu|$ is odd,    $0\leq k\leq n-1$,   and~$\eta\in(0,1)$.

 Given   $\beta\in L_2(\R^N)$, $\e\in(0,1)$, and~$0\leq j\leq m(\e)$ define further
\begin{align*}&{\rm (i)}\quad T^\eps_j (f)[\beta]:=\chi_j^\eps B_{n,\nu}^\phi(f,\bar f)[\bar f^{[k]}, f^{[n-1-k]},f-\bar f,\pi_j^\eps\beta]\qquad\text{for $\phi\in  {\rm C}^\infty([0,\infty)^2)$, or}\\
&{\rm (ii)}\quad T^\eps_j (f)[\beta]:=\chi_j^\eps B_{n,\nu}^\phi(f)[ \bar f^{[k]}, f^{[n-1-k]},f-\bar f,\pi_j^\eps\beta]\qquad\text{for $\phi\in {\rm C}^\infty([0,\infty))$}.
\end{align*}

 In both cases, for each sufficiently small $\eps\in(0,1)$, there  is a constant  $K=K(\eps)>0$ such that  for all~$0\leq j\leq m(\e)$  and $\beta\in H^{\alpha}(\R^N)$ we have
\be\label{Tjesta}
    [T^\eps_j (f)[\beta]]_{H^{\alpha}}\leq \eta\|\pi_j^\eps\beta\|_{H^\alpha}+ K\|\beta\|_{H^{\alpha'}}.
\ee
\end{lemma}
\begin{proof} In this proof, constants denoted by $C$ are independent of $\e\in(0,1)$. 
    We give the proof in Case (i) as Case (ii) is similar and simpler.  It is sufficient to prove the estimate for $k=0$, as the general case will follow from that by repeated application of~\eqref{idbaf} and~$\|\nabla f\|_\infty\leq C$. Set
    \[T:=\chi_j^\eps B_{n,\nu}^\phi(f,\bar f)[{f^{[n-1]}},f-\bar f,\pi_j^\eps\beta].\]
    We recall from \eqref{semiHs} the definition of the translation operators  
    $\tau_\zeta$, $\zeta\in\R^N$, and observe that $\tau_\zeta\bar f-\bar f$ 
    does not depend on $x$.
Hence, by \eqref{difference}  and \eqref{shiftid}, we have 
    \[\tau_\zeta T-T=T_1+T_2+T_3,\]
    where
    \begin{align*}
T_1&:=(\tau_\zeta \chi_j^\e-\chi_j^\e)\tau_\zeta  B^{\phi}_{n,\nu}(f,\bar f)[f^{[n-1]},f-\bar f, \pi_j^\e \beta],\\
T_2&:= \chi_j^\e  B^{\phi}_{n,\nu}(f,\bar f)[ f^{[n-1]},f-\bar f,\tau_\zeta(\pi_j^\e \beta)- \pi_j^\e \beta],\\
T_3&:=\chi_j^\e B^{\phi}_{n,\nu}(\tau_\zeta f,\bar f )[\tau_\zeta f^{[n-1]},\tau_\zeta f- f,\tau_\zeta(\pi_j^\e \beta)]\\
&\,\,\quad+ \sum_{i=0}^{n-2}\chi_j^\e B^{\phi}_{n,\nu}(\tau_\zeta f,\bar f )
[\tau_\zeta f-f,\tau_\zeta f^{[i]},f^{[n-2-i]}, f- \bar f,\tau_\zeta(\pi_j^\e \beta)]\\
&\,\,\quad+\chi_j^\e B^{\phi^1}_{n+2,\nu}(\tau_\zeta f,\bar f,f,\bar f )[ \tau_\zeta f-f,\tau_\zeta f+f, f^{[n-1]}, f- \bar f,\tau_\zeta(\pi_j^\e \beta)],
\end{align*}
with $\phi^1\in {\rm C}^\infty([0,\infty)^4)$ defined in \eqref{gis}.

 We recall  from \eqref{semiHs} that for any $\e\in(0,1)$
\be\label{Hasplit}
[T_{j}^\e(f) [\beta]]_{H^{\alpha}}^2 \leq C_0\bigg(\int_{\R^N}\frac{\|T_1+T_3\|_2^2}{|\zeta|^{N+2\alpha}}\,{\rm d\zeta}+ \int_{\{|\zeta|< \e\}}\frac{\|T_2\|_2^2}{|\zeta|^{N+2\alpha}}\,{\rm d\zeta}
+ \int_{\{|\zeta|\geq \e\}}\frac{\|T_2\|_2^2}{|\zeta|^{N+2\alpha}}\, {\rm d}\zeta\bigg).
\ee

To estimate $T_1$, we apply Lemma \ref{L:MP0} and obtain
\begin{equation}\label{T1}
\|T_1\|_2\leq C\|\tau_\zeta \chi_j^\e-\chi_j^\e\|_\infty\|\beta\|_2\leq C\|\tau_\zeta \chi_j^\e-\chi_j^\e\|_{H^{s-1}}\|\beta\|_{H^{\alpha'}},\qquad\zeta\in\R^N.
\end{equation}

 We estimate the terms of $T_3$ separately. Using the identity~\eqref{idbaf} and 
Lemma \ref{BestL2} with~$s$ replaced by~$s':=s-(\alpha-\alpha')$, $\sigma_0:=s'-1-\alpha'$, and $\sigma_1:=\alpha-(s-s')$ we obtain
\begin{align*}
    &\|\chi_j^\e B^{\phi}_{n,\nu}(\tau_\zeta f,\bar f )
[\tau_\zeta f-f, \tau_\zeta f^{[i]},f^{[n-2-i]}, f- \bar f,\tau_\zeta(\pi_j^\e \beta)]\|_2\\
&\leq \|B^{\phi}_{n,\nu}(\tau_\zeta f,\bar f )
[\tau_\zeta f-f,  \tau_\zeta f^{[i]},f^{[n-1-i]},\tau_\zeta(\pi_j^\e \beta)]\|_2
\\
&\quad\,\,+\sum_{k=1}^N|\p_kf(x_j^\e)|\|B^{\phi}_{n-1,\nu+e_k}(\tau_\zeta f,\bar f )
[\tau_\zeta f-f,\tau_\zeta f^{[i]},f^{[n-2-i]},\tau_\zeta(\pi_j^\e \beta)]\|_2
\\
&\leq  C\|\tau_\zeta f-f\|_{H^{s-\alpha}}\|\beta\|_{H^{\alpha'}}.
\end{align*}
All other terms in $T_3$ can be estimated  in a similar or simpler way. Hence,
\begin{equation}\label{T3}
\|T_3\|_2\leq C\|\tau_\zeta f-f\|_{H^{s-\alpha}}\|\beta\|_{H^{\alpha'}},\qquad\zeta\in\R^N.
\end{equation}

To estimate $T_2$ we distinguish the cases  $|\zeta|\geq\eps$ and $|\zeta|<\eps$. 

 If $|\zeta|\geq\eps$, we apply Lemma \ref{L:MP0} and obtain
\begin{equation}\label{T2g}
\|T_2\|_2\leq C\|\beta\|_{2}.
\end{equation}

 If $|\zeta|<\eps$,  then $\supp\big(\tau_\zeta(\pi_j^\e\beta)-\pi_j^\e\beta\big)\subset\supp\chi_j^\e$, and, by Lemma \ref{L:L2loc},
 we obtain
 \be\label{T2f}
 \|T_2\|_2\leq (\eta/\sqrt{C_0})\|\tau_\zeta(\pi_j^\e\beta)-\pi_j^\e\beta\|_2,
 \ee
 with $C_0$ from \eqref{Hasplit}, provided  that~$\eps$ is chosen small enough.
 
The  desired estimate  \eqref{Tjesta} follows from \eqref{Hasplit}--\eqref{T2f} 
and Lemma~\ref{L:prep}.
\end{proof}

We are ready now to estimate the errors incurred by localizing terms of the form
\[g\sfB_{n,\nu}^\phi(f)[h\beta]\]
 with respect to our $\e$-localization family,   assuming that
 \begin{equation}\label{assloc}
 \text{$\phi\in {\rm C}^\infty([0,\infty))$,\, $n\in\N$,\, $\nu\in\N^N$ with $n+|\nu|$ odd,\,  $f\in H^s(\R^N)$,\, $g,h\in H^{s-1}(\R^N)\cup\{1\}$.}
\end{equation}
Distinguishing the three cases
\begin{itemize}
    \item[(i)] $1\leq  j\leq m(\e)$ (localization in small balls),
    \item[(ii)] $j=0$ and ($g\in H^{s-1}(\R^N)$ or $h\in H^{s-1}(\R^N)$ or $ n\geq 1$) (terms vanishing near infinity),
    \item[(iii)] $j=n=0$, $ g\equiv h\equiv 1$ (principal terms near infinity),
\end{itemize} 
we define the error terms
\begin{equation}\label{eq:Rje}
\cR_j^\e[\beta]:=\cR_{n,\nu,j}^{\phi,g,h,\e}(f)[\beta]:=
\left\{
\begin{array}{ll}
    \pi_j^\e g\sfB_{n,\nu}^{\phi}(f)[h\beta]
    -(gh)(x_j^\e)D_{n,\nu}^{\phi,\nabla f(x_j^\e)}[\pi_j^\e \beta]&\text{in Case (i),}\\
    \pi_0^\e g\sfB_{n,\nu}^{\phi}(f)[h\beta]&\text{in Case (ii),}\\
    \pi_0^\e \sfB_{0,\nu}^{\phi}(f)[\beta]
    -D_{0,\nu}^{\phi,0}[\pi_0^\e \beta]&\text{in Case (iii),}
\end{array}\right.
\end{equation}
 and start by estimating them in lower order  Sobolev (semi)norms.

\begin{lemma}\label{L:locHa} 
 Assume~\eqref{assloc}. Let $\eta\in(0,1)$,  $\alpha\in(0,1)$,  $\alpha'\in\big(\max\{0,\alpha-( s-s_c)\},\alpha\big)$, and~$\theta\in\big(\max\{0,1-(s-s_c)\},1\big)$.
 Then, for each sufficiently small~${\e\in(0,1)}$, there exists a constant~$K=K(\e)>0$ such that
 for all~$0\leq j\leq m(\e)$  and~${\beta\in H^{\alpha}(\R^N)}$ we have
\begin{equation}\label{locHa}
[\cR_j^\e [\beta]]_{H^\alpha}\leq \eta\|\pi_j^\e\beta\|_{H^{\alpha}}+K\|\beta\|_{H^{\alpha'}},
\end{equation}
 and for all $\beta\in H^1(\R^N)$ 
\begin{equation}\label{locH1}
\|\cR_j^\e [\beta]\|_{H^1}\leq \eta\|\pi_j^\e\beta\|_{H^1}+K\|\beta\|_{H^\theta}.
\end{equation}
\end{lemma}
 \begin{proof}

    In this proof, constants which are independent of $\e\in(0,1)$ are denoted by~$C$  and constants that may depend on $\e$ are denoted by $K$.\smallskip

  \noindent{\bf Case (i):} We give the proof of \eqref{locHa}  and \eqref{locH1} for $g,\,h\in H^{s-1}(\R^N)$. (If $g\equiv 1$ or $h\equiv 1$, the proof can be given in a similar or simpler way).
 
Fix $ 1\leq j\leq m(\e)$. We rewrite 
\begin{equation*} 
 \cR_j^\e [\beta]=g(T_a+T_b) +h(x_j^\e)(T_c+g(x_j^\e)T_d),
\end{equation*}
 where 
 \begin{equation*} 
 \begin{aligned}
&T_a:=  \llbracket\pi_j^\e, \sfB^{\phi}_{n,\nu}(f)\rrbracket [(h-h(x_j^\e))\beta],
&&\quad T_b:= \sfB^{\phi}_{n,\nu}(f)[\pi_j^\e (h-h(x_j^\e))\beta],\\
&T_c:=\pi_j^\e g \sfB^{\phi}_{n,\nu}(f)[\beta]- g(x_j^\e)\sfB^{\phi}_{n,\nu}(f)[\pi_j^\e \beta],
&&\quad T_d:=\sfB^{\phi}_{n,\nu}(f)[\pi_j^\e \beta]-D^{\phi, \nabla f(x_j^\e)}_{n,\nu}[\pi_j^\e\beta].
\end{aligned}
\end{equation*}
We estimate these terms separately.

From Lemma~\ref{L:mult} and  Lemma~\ref{commphi}~(i) we have
\begin{align}\label{TES4}
[gT_a]_{H^{\alpha}}\leq C\|gT_a\|_{H^1}\leq C\|g\|_{H^{s-1}}\|T_a\|_{H^1}\leq K\|(h-h(x_j^\e))\beta\|_2\leq K \|\beta\|_{2}.
\end{align}

 For  the term $gT_b$, using Lemma \ref{L:mult}, Lemma~\ref{Bestgen}, Lemma~\ref{L:alg},  the H\" older continuity of $h$, and the 
identity~$\chi_j^\e\pi_j^\e=\pi_j^\e$  we get, for  sufficiently small~$\e\in(0,1)$,
\begin{subequations}\label{TES5}
\begin{equation}\label{TES5a}
  \begin{aligned}
\relax[gT_b]_{H^{\alpha}}&\leq C\|gT_b\|_{H^{\alpha}}\leq C\|g\|_{H^{s-1}}\|T_b\|_{H^{\alpha}}\leq C\|\pi_j^\e (h-h(x_j^\e))\beta\|_{H^{\alpha}}\\
&\leq C\|\chi_j^\e (h-h(x_j^\e))\|_\infty\|\pi_j^\e\beta\|_{H^{\alpha}}+K\|\pi_j^\e\beta\|_{H^{\alpha'}}\\
&\leq  (\eta/3)\|\pi_j^\e\beta\|_{H^{\alpha}}+K\|\beta\|_{H^{\alpha'}}
\end{aligned}
\end{equation}
 and
\begin{equation}\label{TES5b}
  \begin{aligned}
 \|gT_b\|_{H^1}&\leq C\|g\|_{H^{s-1}}\|T_b\|_{H^1}\\
 &\leq C\big(
\|\chi_j^\e (h-h(x_j^\e))\|_\infty\|\pi_j^\e\beta\|_{H^1}+
\|\pi_j^\e(h-h(x_j^\e))\|_{H^{s-1}}\|\beta\|_{H^\theta}\big)\\
&\leq (\eta/3)\|\pi_j^\e\beta\|_{H^1}+K\|\beta\|_{H^\theta}.
\end{aligned}
\end{equation}
\end{subequations}

To estimate $T_c$, we  split
\[
T_{c}=T_{c,1}+T_{c,2}+T_{c,3},
\]
where
\begin{align*}
T_{c,1}&:=\chi_j^\e g  \llbracket\pi_j^\e, \sfB^{\phi}_{n,\nu}( f)\rrbracket [\beta],\\
T_{c,2}&:=\chi_j^\e (g -g(x_j^\e)) \sfB^{\phi}_{n,\nu}(f)[\pi_j^\e \beta] ,\\
T_{c,3}&:=  g(x_j^\e) \llbracket\chi_j^\e, \sfB^{\phi}_{n,\nu}( f)\rrbracket [\pi_j^\e\beta].
\end{align*}
Lemma~\ref{commphi} together with Lemma~\ref{L:mult} ensures that 
\begin{align*}
[h(x_j^\e)(T_{c,1}+T_{c,3})]_{H^{\alpha}}\leq C \|h(x_j^\e)(T_{c,1}+T_{c,3})\|_{H^1}
\leq C\|T_{c,1}+T_{c,3} \|_{H^{1}}\leq  K\|\beta\|_2.
\end{align*}
Moreover, using Lemma~\ref{L:alg},  Lemma~\ref{Bestgen}, the H\" older continuity of $g$,  we have, as in \eqref{TES5}
\begin{subequations}\label{TES6}
\begin{equation}\label{TES6a}
  \begin{aligned}
\relax[h(x_j^\e)T_{c,2}]_{H^{\alpha}}&\leq C\|T_{c,2}\|_{H^{\alpha}}\leq C\|\chi_j^\e (g-g(x_j^\e))\|_\infty\|\pi_j^\e\beta\|_{H^{\alpha}}+K\|\pi_j^\e\beta\|_{H^{\alpha'}}\\
&\leq  (\eta/3)\|\pi_j^\e\beta\|_{H^{\alpha}}+K\|\beta\|_{H^{\alpha'}}
\end{aligned}
\end{equation} 
 and
\begin{equation}\label{TES6b}
  \begin{aligned}
 \|h(x_j^\e)T_{c,2}\|_{H^1}
&\leq  C \|\chi_j^\e (g-g(x_j^\e))\sfB^{\phi}_{n,\nu}(f)[\pi_j^\e \beta]\|_{H^1}\\
&\leq C\big(\|\chi_j^\e (g-g(x_j^\e))\|_\infty
\|\sfB^{\phi}_{n,\nu}(f)[\pi_j^\e \beta]\|_{H^1}\\
&\quad+\|[\chi_j^\e (g-g(x_j^\e))]\|_{H^{s-1}}
\|\sfB^{\phi}_{n,\nu}(f)[\pi_j^\e \beta]\|_{H^\theta}\big)\\
&\leq (\eta/3)\|\pi_j^\e\beta\|_{H^1}+K\|\beta\|_{H^\theta},
\end{aligned}
\end{equation} 
\end{subequations}
provided that $\e\in(0,1)$ is sufficiently small.

It remains to consider the term $T_d$. 
Recalling \eqref{identif} and \eqref{barf},  we write
\[T_d=\sfB^{\phi}_{n,\nu}(f)[\pi_j^\e \beta]-\sfB^{\phi}_{n,\nu}(\bar f)[\pi_j^\e \beta]=T_{d,1}+ S[\pi_j^\e\beta],\] 
where
\begin{align*}
T_{d,1}&:= \llbracket\chi_j^\e, \sfB^{\phi}_{n,\nu}(\bar f)\rrbracket [\pi_j^\e\beta] -\llbracket\chi_j^\e, \sfB^{\phi}_{n,\nu}(f)\rrbracket [\pi_j^\e\beta],\\
S&:=\chi_j^\e\big( \sfB^{\phi}_{n,\nu}(f)- \sfB^{\phi}_{n,\nu}(\bar f)\big).
\end{align*}
Invoking Lemma~\ref{commphi} again, we have
\be\label{TES71}
[(gh)(x_j^\e)T_{d,1}]_{H^{\alpha}}
\leq C\|(gh)(x_j^\e)T_{d,1}\|_{H^1}
\leq C\| T_{d,1}\|_{H^{1}}\leq K \|\beta\|_{2}.
\ee
 To estimate $[S[\pi_j^\e\beta]]_{H^{\alpha}}$, we infer from \eqref{difference} that 
\be\label{splitS}
S=\chi_j^\e B_{n+2,\nu}^{\phi^1}(f,\bar f)[f+\bar f,\bar f^{[n]},f-\bar f,\cdot]
+\sum_{i=0}^{n-1}\chi_j^\e B_{n,\nu}^{\phi}( f)[ f^{[i]},\bar f^{[n-1-i]},f-\bar f,\cdot],
\ee
where $\phi^1\in{\rm C}^\infty([0,\infty)^2)$ is defined in \eqref{gis}.
Applying  \eqref{idbaf}  and Lemma~\ref{L:D6x}, we conclude that for  sufficiently small $\e\in(0,1)$
\begin{subequations}\label{TES72}
\begin{equation}\label{TES72a}
[ (gh)(x_j^\e) S[\pi^e_j\beta]]_{H^{\alpha}}\leq  (\eta/3)\|\pi_j^\e\beta\|_{H^{\alpha}}+K\|\beta\|_{H^{\alpha'}}.
\end{equation}

 The estimate~\eqref{locHa} in Case (i) follows now  from  \eqref{TES4}, \eqref{TES5a}, \eqref{TES6a}, \eqref{TES71}, and~\eqref{TES72a}.

It remains to estimate $\|S[\pi_j^\e\beta]\|_{H^1}$. In view of 
\eqref{equivgrad} and Lemma \ref{L:MP0}, it suffices to consider  the term~$\|\p_iS[\pi_j^\e\beta]\|_2$  for~$1\leq i\leq N$. We have
\[\p_iS[\pi_j^\e\beta]=\p_i\chi_j^\e\big( \sfB^{\phi}_{n,\nu}(f)- \sfB^{\phi}_{n,\nu}(\bar f)\big)[\pi_j^\e\beta]
+\chi_j^\e\llbracket \p_i,\sfB_{n,\nu}^\phi(f)\rrbracket[\pi_j^\e\beta]
+S[\p_i(\pi_j^\e\beta)].\]

Using \eqref{commsfb}, Lemma \ref{L:MP0}, Lemma \ref{Bestgen} with $s$ replaced by $s-(1-\theta)$, and, for the last term,~\eqref{splitS} and Lemma \ref{L:L2loc}, we obtain
\begin{equation}\label{TES72b}
\|(gh)(x_j^\e)\p_iS[\pi_j^\e\beta]\|_2\leq (\eta/3)\|\p_i(\pi_j^\e\beta)\|_2+K\|\beta\|_{H^\theta}
\end{equation}
\end{subequations}
for  sufficiently small $\e\in(0,1)$.

The estimate~\eqref{locH1} in Case (i) follows from \eqref{TES4}, \eqref{TES5b}, \eqref{TES6b}, \eqref{TES71}, and~\eqref{TES72b}.

\smallskip

\noindent{\bf Case (ii):} We rewrite, using   the identity $\pi^\e_0=\chi^\e_0\pi^\e_0$,
\[\pi_0^\e g \sfB^{\phi}_{n,\nu}(f)[h\beta]=
 \chi_0^\e g \llbracket\pi_0^\e, \sfB^{\phi}_{n,\nu}( f)\rrbracket [ h\beta]
 +\chi_0^\e g \sfB^{\phi}_{n,\nu}(f)[\pi_0^\e h\beta].\]
 Analogously to \eqref{TES4}, we obtain for the first term
 \begin{equation*}
 [\chi_0^\e g\llbracket\pi_0^\e, \sfB^{\phi}_{n,\nu}( f)\rrbracket [h\beta]]_{H^{\alpha}}
 \leq C\|\chi_0^\e g\llbracket\pi_0^\e, \sfB^{\phi}_{n,\nu}( f)\rrbracket [h\beta]\|_{H^1} 
 \leq K  \|\beta\|_{2}.
 \end{equation*}
 It remains to estimate the second term, for which we distinguish three cases:
 
If $g\in H^{s-1}(\R^N)$, then $\|\chi_0^\e g\|_\infty\to 0$ as  $\e\to 0$,  and  we deduce, from Lemma~\ref{L:mult},   Lemma~\ref{Bestgen}, and Lemma  \ref{L:alg},  for $\e\in(0,1)$ sufficiently small, that
\begin{equation}  
\begin{aligned}
 &[\chi_0^\e g \sfB^{\phi}_{n,\nu}(f)[\pi_0^\e h\beta]]_{H^{\alpha}}\\
&\leq  C\big(\|\chi_0^\e g\|_\infty\| \sfB^{\phi}_{n,\nu}(f)[\pi_0^\e h\beta]\|_{H^{\alpha}}
  +\|\chi_0^\e g\|_{H^{s-1}}\| \sfB^{\phi}_{n,\nu}(f)[\pi_0^\e h\beta]\|_{H^{\alpha'}}\big)\\
  &\leq  C\|\chi_0^\e g\|_\infty\| \sfB^{\phi}_{n,\nu}(f)[\pi_0^\e h\beta]\|_{H^{\alpha}}
  +K\| \pi_0^\e h\beta\|_{H^{\alpha'}}\\
 & \leq \eta\| \pi_0^\e \beta \|_{H^{\alpha}} +K\|  \beta \|_{H^{\alpha'}}
  \end{aligned}
  \end{equation}
   and
  \begin{equation}  
\begin{aligned}
  & \|\chi_0^\e g \sfB^{\phi}_{n,\nu}(f)[\pi_0^\e h\beta]\|_{H^1}\\
  &\leq  C\big(\|\chi_0^\e g\|_\infty\| \sfB^{\phi}_{n,\nu}(f)[\pi_0^\e h\beta]\|_{H^1}
  +\|\chi_0^\e g\|_{H^{s-1}}\| \sfB^{\phi}_{n,\nu}(f)[\pi_0^\e h\beta]\|_{H^\theta}\big)\\
  &\leq  C\|\chi_0^\e g\|_\infty\| \sfB^{\phi}_{n,\nu}(f)[\pi_0^\e h\beta]\|_{H^1}
  +K\| \pi_0^\e h\beta\|_{H^\theta}\\
&  \leq \eta\| \pi_0^\e \beta \|_{H^1} +K\|  \beta \|_{H^\theta}.
  \end{aligned}
  \end{equation}

If $g\equiv 1$ and $h\in H^{s-1}(\R^N)$, 
then, arguing as above, since  $\chi_0^\e-1\in  H^{s-1}(\R^N)$,  we obtain for sufficiently small~$\e\in(0,1)$,  
\begin{equation*}  
\begin{aligned}
 &[\chi_0^\e g \sfB^{\phi}_{n,\nu}(f)[\pi_0^\e h\beta]]_{H^{\alpha}}
 \leq C\|(\chi_0^\e-1+1)  \sfB^{\phi}_{n,\nu}(f)[\pi_0^\e h\beta]\|_{H^{\alpha}}\\
 & \leq  C\|\chi_0^\e-1 \|_\infty\| \sfB^{\phi}_{n,\nu}(f)[\pi_0^\e h\beta]\|_{H^{\alpha}}+K\| \sfB^{\phi}_{n,\nu}(f)[\pi_0^\e h\beta]\|_{H^{\alpha'}}
  +C\| \sfB^{\phi}_{n,\nu}(f)[\pi_0^\e h\beta]\|_{H^{\alpha}}\\
  &\leq C \| (\chi_0^\e h) (\pi_0^\e\beta) \|_{H^{\alpha}}
  +K\|  \beta \|_{H^{\alpha'}}
   \leq  C \| \chi_0^\e h\|_\infty \|\pi_0^\e\beta  \|_{H^{\alpha}}
  +K\|  \beta \|_{H^{\alpha'}}\\ 
  &\leq \eta \|  \pi_0^\e \beta \|_{H^{\alpha}}
  +K\|  \beta \|_{H^{\alpha'}}
\end{aligned}
\end{equation*}
 and analogously
\[\|\chi_0^\e g \sfB^{\phi}_{n,\nu}(f)[\pi_0^\e h\beta]\|_{H^1}
\leq \eta \|  \pi_0^\e \beta \|_{H^1}
  +K\|  \beta \|_{H^\theta}.\]

Finally, if $g\equiv h\equiv 1$ and $n\geq 1$,  the estimate \eqref{locHa} is  established in  Lemma \ref{L:D6x}~(ii). To establish \eqref{locH1}, it is sufficient to estimate $\|\p_i\big(\chi_0^\e \sfB^{\phi}_{n,\nu}(f)[\pi_0^\e \beta]\big)\|_2$  for $1\leq i\leq N$. We have 
\[\p_i\big(\chi_0^\e \sfB^{\phi}_{n,\nu}(f)[\pi_0^\e \beta]\big)
=\p_i\chi_0^\e \sfB^{\phi}_{n,\nu}(f)[\pi_0^\e \beta]
+\chi_0^\e\llbracket\p_i,\sfB^{\phi}_{n,\nu}(f)\rrbracket[\pi_0^\e \beta]
+\chi_0^\e \sfB^{\phi}_{n,\nu}(f)[\p_i(\pi_0^\e \beta)].
\]
Using \eqref{commsfb}, Lemma \ref{L:MP0}, Lemma \ref{Bestgen} with $s$ replaced by 
$s-(1-\theta)$, and Lemma \ref{L:L2loc} we obtain
\[\|\p_i\big(\chi_0^\e \sfB^{\phi}_{n,\nu}(f)[\pi_0^\e \beta]\big) \|_2
\leq \eta\|\p_i (\pi_0^\e\beta)\|_2+K\|\beta\|_{H^\theta}\]
for $\eps\in(0,1)$ sufficiently small. This implies \eqref{locH1} in Case (ii).
\smallskip 

\noindent{\bf Case (iii):} We rewrite
\[\pi_0^\e \sfB^{\phi}_{0,\nu}(f)[\beta]
-D_{0,\nu}^{\phi,0}[\pi_0^\e\beta]=\chi_0^\e\llbracket\pi_0^\e,\sfB^{\phi}_{0,\nu}(f)\rrbracket
[\beta]+\llbracket\chi_0^\e,D_{0,\nu}^{\phi,0}\rrbracket[\pi_0^\e\beta]
+\chi_0^\e\big(\sfB^{\phi}_{0,\nu}(f)[\pi_0^\e\beta]-D_{0,\nu}^{\phi,0}[\pi_0^\e\beta]\big).
\]
By Lemma \ref{commphi}, we have for the commutator terms
\begin{align*}
\big[\chi_0^\e\llbracket\pi_0^\e,\sfB^{\phi}_{0,\nu}(f)\rrbracket
[\beta]+\llbracket\chi_0^\e,D_{0,\nu}^{\phi,0}\rrbracket[\pi_0^\e\beta]\big]_{H^\alpha}&\leq  C\| \chi_0^\e\llbracket\pi_0^\e,\sfB^{\phi}_{0,\nu}(f)\rrbracket
[\beta]+\llbracket\chi_0^\e,D_{0,\nu}^{\phi,0}\rrbracket[\pi_0^\e\beta]\big\|_{H^1}\\
&\leq  K\|\beta\|_2.
\end{align*}
For the  difference term, we  recall from \eqref{identif} that  $D^{\phi, 0}_{0,\nu}=B^{\phi}_{0,\nu}(0)$ and infer from~\eqref{difference} 
that
\[\chi_0^\e\big( \sfB^{\phi}_{0,\nu}(f)-  D^{\phi, 0}_{0,\nu}\big)[ \pi_0^\e\beta]
=\chi_0^\e\sfB^{\tilde\phi}_{2,\nu}(f)[ \pi_0^\e\beta],
\]
where $\tilde\phi\in {\rm C}^\infty([0,\infty))$ is given by
\[\tilde\phi(x)=(\phi(x)-\phi(0))/x\quad\text{ if $x>0$,}\qquad \wt\phi(0)=\phi'(0).\]
 Such terms have been estimated in Case (ii).
Thus, \eqref{locHa}  and \eqref{locH1} hold in  Case (iii) as well, and the proof is complete.
\end{proof}

 As a further preparation, we estimate the commutators of  partial derivatives with the localization error terms defined in \eqref{eq:Rje}.
\begin{lemma}\label{L:commR}
 Assume~\eqref{assloc}. Let $\theta\in [0,s-2]$  and~$\theta'\in\big(\max\{\theta-1,\theta-(s-s_c)\},\theta\big)$.
Then, for each ~$\e\in(0,1)$, there exist a constant $K=K(\e)>0$ such that for all~$0\leq j\leq m(\e)$, $1\leq i\leq N$, and $\beta\in  H^{\theta}(\R^N)$ it holds that
\[\|\llbracket\p_i,\cR_j^\e\rrbracket[\beta]\|_{H^\theta}
\leq K\|\beta\|_{H^{1+\theta'}}.\]
\end{lemma}
\begin{proof}
    We give the details for Case~(i) with~$g,\, h\in H^{s-1}(\R^N)$ and $n\geq 1$, the proof in the other cases being similar and simpler.  Below, we use the same convention regarding the notation for constants as in the proof of Lemma \ref{L:locHa}.

    For the commutators $\llbracket\p_i,\cR_j^\e\rrbracket$ we have the representation 
   \begin{align*}
    \llbracket\p_i,\cR_j^\e\rrbracket[\beta]&= 
    \p_i(\pi_j^\e g) \sfB_{n,\nu}^\phi(f)[h\beta]
    +\pi_j^\e g\llbracket\p_i,\sfB_{n,\nu}^\phi(f)\rrbracket[h\beta]
    +\pi_j^\e g\sfB_{n,\nu}^\phi(f)[\beta\p_ih]\\
    &\quad \,\,  -(gh)(x_j^\e)D_{n,\nu}^{\phi,\nabla f(x_j^\e)}[\beta\p_i\pi_j^\e].
\end{align*}
   Using  Lemma~\ref{L:mult}, Lemma \ref{Bestgen}, and Proposition~\ref{P:D1},  we  get
\begin{align*}
    \|\p_i(\pi_j^\e g)\sfB_{n,\nu}^\phi(f)[h\beta]  \|_{H^\theta}&\leq C\|\p_i(\pi_j^\e g)\|_{H^{s-2}}\|\sfB_{n,\nu}^\phi(f)[h\beta]\|_{H^{1+\theta'}}\\
    &\leq K\|h\beta\|_{H^{1+\theta'}}\leq K \|h\|_{H^{s-1}}\|\beta\|_{H^{1+\theta'}}
    \leq K \|\beta\|_{H^{1+\theta'}},\\
    \|\pi_j^\e g\sfB_{n,\nu}^\phi(f) [\beta\p_ih]\|_{H^\theta}
    &\leq C\|\pi_j^\e g\|_{H^{s-1}}\|\sfB_{n,\nu}^\phi(f) [\beta\p_ih]\|_{H^\theta}\\
    &\leq  K \|\beta\p_ih\|_{H^\theta}
    \leq K\|\p_i h\|_{H^{s-2}}\|\beta]\|_{H^{1+\theta'}}
    \leq  K \|\beta\|_{H^{1+\theta'}},\\
    \| (gh)(x_j^\e)D_{n,\nu}^{\phi,\nabla f(x_j^\e)}[\beta\p_i\pi_j^\e]\|_{H^\theta}
       &\leq C\|\beta\p_i\pi_j^\e\|_{H^{\theta}}\leq K \|\beta\|_{H^{1+\theta'}}.
\end{align*}
 To estimate the remaining term 
$\pi_j^\e g\llbracket\p_i,\sfB_{n,\nu}^\phi(f)\rrbracket[h\beta]$ we recall \eqref{commsfb} and use Lemma \ref{L:mult} and Lemma~\ref{Bestgen} with $s$ replaced by $s':=s-\theta+\theta'$,  $\sigma_0:=s-\theta-2$,  $\sigma_1=1-\theta+\theta'$, and~$\sigma:=s-2\theta+\theta'-1$.   Thus we obtain
\begin{align*}
   &\|\pi_j^\e g B_{n,\nu}^\phi[\p_if,f^{[n-1]},h\beta]\|_{H^\theta}\leq C
   \|\pi_j^\e g\|_{H^{s-1}}
   \|B_{n,\nu}^\phi[\p_if,f^{[n-1]},h\beta]\|_{H^\theta}\\
   &\leq K\|\p_if\|_{H^{s-1}}
   \|h\beta\|_{H^{1+\theta'}}\leq K \|\beta\|_{H^{1+\theta'}},
\end{align*}
and analogously for the other term  originating from \eqref{commsfb}. 
\end{proof}

We are now ready to state and prove the main localization result for the singular integral operators~$\sfB_{n,\nu}^\phi(f)$ by estimating the error terms $\cR_j^\e[\beta]$ defined in~\eqref{eq:Rje} in   $H^{s-1}(\R^N)$.
\begin{prop}\label{T:LOC}
Assume~\eqref{assloc}. 
 Let  $\eta\in(0,1)$ and  $s'\in (\max\{s_c,s-1\},s)$.
 Then, for each sufficiently small $\e\in(0,1)$, there exists a constant~$K=K(\e)>0$ such that
 for all~$0\leq j\leq m(\e)$ and~${\beta\in H^{ s-1}(\R^N)}$ we have
\begin{equation}\label{es:desrje}
\|\cR_j^\e[\beta]\|_{H^{s-1}}\leq \eta\|\pi_j^\e\beta\|_{H^{s-1}}
+K\|\beta\|_{H^{s'-1}}.
\end{equation}
\end{prop}
\begin{proof}
      % We give the details for Case~(i) with~$g,\, h\in H^{s-1}(\R^N)$ and $n\geq 1$, the proof in the other cases being similar and simpler. 

Let $\alpha\in(0,1]$ and $\alpha'\in\big(\max\{0,\alpha-(s-s_c)\},\alpha\big)$.
We are going to show the following more general statement for $k\in\N$ satisfying
      $0<k+\alpha\leq s-1$:
\[\left.\text{
\begin{minipage}{0.86\textwidth}  
For any $\eta \in(0,1)$ and $f\in H^s(\R^N)$, there exists an $\eps_0\in(0,1)$ such that for  each~$\eps\in(0,\eps_0)$  there is a constant ${K=K(\e)>0}$ such that 
for all~${0\leq j\leq m(\e)}$ and~${\beta\in H^{ k+\alpha}(\R^N)}$: 
\[\|\cR_j^\e[\beta]\|_{H^{k+\alpha}}\leq \eta\|\pi_j^\e\beta\|_{H^{k+\alpha}}
+K\|\beta\|_{H^{k+\alpha'}}.\]
\end{minipage}}\right\}\eqno{\textnormal{(H)$_{k,\alpha}$}}\]

Let first $\alpha\in(0,1)$. To show (H)$_{0,\alpha}$ we fix $\eta>0$ and infer from Lemma \ref{L:MP0} and Lemma~\ref{L:locHa} (with $\eta$ replaced by   $\eta_0:=\eta/C_1$ with~$C_1$ from~\eqref{equivHs})   that for  sufficiently small $\e\in(0,1)$
\begin{align*}
    \|\cR_j^\e[\beta]\|_{H^{\alpha}}
&\leq {C_1}\big(\|\cR_j^\e[\beta]\|_{L^2}+\big[\cR_j^\e[\beta]\big]_{H^{\alpha}}\big)
\leq  C_1\eta_0\|\pi_j^\e\beta\|_{H^{\alpha}}+K\|\beta\|_{H^{\alpha'}}\\
&\leq \eta\|\pi_j^\e\beta\|_{H^{\alpha}}+K\|\beta\|_{H^{\alpha'}}. 
\end{align*}

The statement (H)$_{1,0}$ has been shown in Lemma \ref{L:locHa}.

To prove the complete result it is sufficient now to show the implication
\[\text{(H)$_{k-1,\alpha}$}\;\Longrightarrow\;\text{(H)$_{k,\alpha}$},
\qquad 0< k-1+\alpha\leq s-2.\]
For this, assume $0< k-1+\alpha\leq s-2$ and (H)$_{k-1,\alpha}$.
Then, using \eqref{equivgrad}, 
\be\label{splitRje}
\|\cR_j^\e[\beta]\|_{H^{k+\alpha}}\leq  C_0\bigg(\|\cR_j^\e[\beta]\|_{H^{k-1+\alpha}}
+\sum_{i=1}^N\Big(\|\cR_j^\e[\p_i\beta]\|_{H^{k-1+\alpha}}
+\|\llbracket\p_i,\cR_j^\e\rrbracket[\beta]\|_{H^{k-1+\alpha}}\Big)\bigg).
\ee
We estimate the terms on the right separately.  
By the induction assumption  (with $\eta$ replaced by   $\eta_0:=\eta/(N C_0)$  with $C_0$ from~\eqref{equivgrad})  and Lemma \ref{L:commR}, for sufficiently small $\e\in(0,1)$ we have
\begin{align*}
 \|\cR_j^\e[\beta]\|_{H^{k-1+\alpha}}&\leq \eta_0\|\pi_j^\e\beta\|_{H^{k-1+\alpha}}
+K\|\beta\|_{H^{k-1+\alpha'}}\leq K\|\beta\|_{H^{k+\alpha'}},\\
\|\cR_j^\e[\p_i\beta]\|_{H^{k-1+\alpha}}&\leq \eta_0\|\pi_j^\e\p_i\beta\|_{H^{k-1+\alpha}}+K\|\p_i\beta\|_{H^{k-1+\alpha'}}\\&
\leq \eta_0\|\p_i(\pi_j^\e\beta)\|_{H^{k-1+\alpha}}+K\|\beta\|_{H^{k+\alpha'}},\\
\|\llbracket\p_i,\cR_j^\e\rrbracket[\beta]\|_{H^{k-1+\alpha}}& \leq K\|\beta\|_{H^{k+\alpha'}}.
\end{align*}
 The statement (H)$_{k,\alpha}$ follows from these estimates together with 
\eqref{splitRje}.
  \end{proof}

We conclude this section with a result concerning the localization of a product of two functions.

\begin{lemma}\label{T:prod}
Let  $g\in H^{s-1}(\R^N)$ and  $s'\in(s_c,s)$.
Then, given~${\eta>0}$, for each sufficiently small $\e\in(0,1)$, there is a constant $K=K(\e)>0$ such that for all~${\beta\in H^{s-1}(\R^N)}$ we have
\begin{equation}\label{Prod1}
\|\pi_0^\e g\beta\|_{H^{s-1}}\leq \eta\|\pi_0^\e\beta\|_{H^{s-1}}+K\|\beta\|_{H^{s'-1}} 
\end{equation}
and
\begin{equation}\label{Prod2}
\|\pi_j^\e (g-g(x_j^\e))\beta\|_{H^{s-1}}\leq \eta\|\pi_j^\e\beta\|_{H^{s-1}}+K\|\beta\|_{H^{s'-1}}, \qquad 1\leq j\leq  m(\e). 
\end{equation}
\end{lemma}
\begin{proof}
Fix $\eta>0$ and $1\leq j\leq  m(\e)$. 
Using the Kato-Ponce estimate from \cite[ Lemma X.4]{KP88} and the fact that $\|\chi_j^\e(g-g(x_j^\e))\|_\infty\leq  C|\e|^{s-s_c}$, with a constant $C>0$ depending only on the H\"older seminorm $[g]_{s-s_c}$, we get
\begin{align*}
    &\|\pi_j^\e (g-g(x_j^\e))\beta\|_{H^{s-1}}=
    \|\chi_j^\e (g-g(x_j^\e))\pi_j^\e\beta\|_{H^{s-1}}\\
    &\leq C(\|\chi_j^\e (g-g(x_j^\e))\|_\infty\|\pi_j^\e\beta\|_{H^{s-1}}
    +\|\chi_j^\e (g-g(x_j^\e))\|_{H^{s-1}}\|\pi_j^\e\beta\|_\infty)\\
    &\leq \eta\|\pi_j^\e\beta\|_{H^{s-1}}+K\|\beta\|_{H^{s'-1}}
\end{align*}
for $\e$ sufficiently small. The estimate \eqref{Prod1} is obtained analogously, using~$\|\chi_0^\e g\|_\infty\to 0$  as~$\e\to 0$.
\end{proof}

\bibliographystyle{siam}
\bibliography{Literature}
\end{document}